\documentclass[11pt,reqno]{amsart}

\newtheorem{thm}{Theorem}[section]
\newtheorem*{thm*}{Theorem}
\newtheorem{lem}[thm]{Lemma}
\newtheorem*{lem*}{Lemma}

\newtheorem{cor}[thm]{Corollary}
\newtheorem{claim}[thm]{Claim}
\newtheorem{prop}[thm]{Proposition}

\theoremstyle{definition}

\newtheorem{assump}[thm]{Assumption}
\newtheorem*{case*}{Case}

\newtheorem{defn}[thm]{Definition}
\newtheorem*{defn*}{Definition}
\newtheorem{exmp}[thm]{Example}
\newtheorem*{exmp*}{Example}

\newtheorem{hyp}[thm]{Hypothesis}

\newtheorem{step}{Step}\renewcommand{\thestep}{}

\theoremstyle{remark}

\newtheorem{case}{Case}\renewcommand{\thecase}{}

\newtheorem{rmk}[thm]{Remark}
\newtheorem*{rmk*}{Remark}

\makeatletter
\def\alphenumi{
  \def\theenumi{\alph{enumi}}
  \def\p@enumi{\theenumi}
  \def\labelenumi{(\@alph\c@enumi)}}
\makeatother

\makeatletter
\def\thecase{\@arabic\c@case}
\makeatother

\makeatletter
\def\thestep{\@arabic\c@step}
\makeatother

\newcount\hh
\newcount\mm
\mm=\time
\hh=\time
\divide\hh by 60
\divide\mm by 60
\multiply\mm by 60
\mm=-\mm
\advance\mm by \time
\def\hhmm{\number\hh:\ifnum\mm<10{}0\fi\number\mm}

\let\oldmarginpar\marginpar
\renewcommand\marginpar[1]{\-\oldmarginpar[\raggedleft\footnotesize #1]%
{\raggedright\footnotesize #1}}

\renewcommand\emptyset{\varnothing}

\newcommand\BB{\mathbb{B}}

\newcommand\EE{\mathbb{E}}

\newcommand\HH{\mathbb{H}}

\newcommand\NN{\mathbb{N}}

\newcommand\RR{\mathbb{R}}

\newcommand\fa{{\mathfrak{a}}}

\newcommand\fw{{\mathfrak{w}}}

\newcommand\sO{{\mathscr{O}}}

\newcommand\sU{{\mathscr{U}}}
\newcommand\sV{{\mathscr{V}}}

\newcommand\eps{\varepsilon}

\newcommand\less{\setminus}

\newcommand\diam{\operatorname{diam}}

\newcommand\dist{\operatorname{dist}}

\newcommand{\essinf}{\operatornamewithlimits{ess\ inf}}
\newcommand{\esssup}{\operatornamewithlimits{ess\ sup}}

\DeclareMathOperator{\height}{height}

\DeclareMathOperator{\Int}{int}

\newcommand{\osc}{\operatornamewithlimits{osc}}

\newcommand\sign{\operatorname{sign}}

\newcommand\supp{\operatorname{supp}}

\newcommand\vol{\operatorname{vol}}

\newcommand\apriori{{\emph{a priori }}}

\newcommand\loc{{\mathrm{loc}}}

\numberwithin{equation}{section}
\usepackage{amssymb}
\usepackage{hyperref}
\usepackage{graphicx}
\usepackage{mathrsfs}
\usepackage[usenames]{color}
\usepackage{url}
\usepackage{verbatim}

\hypersetup{pdftitle={Boundary-degenerate elliptic operators and H\"older continuity for solutions to variational equations and inequalities}}
\hypersetup{pdfauthor={Paul M. N. Feehan and Camelia A. Pop}}

\begin{document}

\title[H\"older continuity for solutions to variational equations and inequalities]
{Boundary-degenerate elliptic operators and H\"older continuity for solutions to variational equations and inequalities}

\author[P. M. N. Feehan]{Paul M. N. Feehan}
\address[PF]{Department of Mathematics, Rutgers, The State University of New Jersey, 110 Frelinghuysen Road, Piscataway, NJ 08854-8019}
\email{feehan@math.rutgers.edu}
\curraddr{School of Mathematics, Institute for Advanced Study, Princeton, NJ 08540}
\email{feehan@math.ias.edu}

%CP 1.11.2016: Changed affiliation
%PF 2-27-2016: OK
\author[C. A. Pop]{Camelia A. Pop}
\address[CP]{School of Mathematics, University of Minnesota, Vincent Hall, 206 Church St. SE, Minneapolis, MN 55455}
\curraddr{Institute for Mathematics and Its Applications, Lind Hall, 207 Church St. SE, Minneapolis, MN 55455}
\email{capop@umn.edu}

%\date{\today{ }\hhmm}
\date{March 9, 2016}

\begin{abstract}
We prove local supremum bounds, a Harnack inequality, H\"older continuity up to the boundary, and a strong maximum principle for solutions to a variational equation defined by an elliptic operator which becomes degenerate along a portion of the domain boundary and where no boundary condition is prescribed, regardless of the sign of the Fichera function. In addition, we prove H\"older continuity up to the boundary for solutions to variational inequalities defined by this boundary-degenerate elliptic operator.
\end{abstract}

% AMS 2010 subject classifications (used in AMS journals)
% Primary
% 35J70  Degenerate elliptic equations
% 35J86  Linear elliptic unilateral problems and linear elliptic variational inequalities
% 49J40  Variational methods including variational inequalities
% 35R45  Partial differential inequalities
%
% Secondary
% 35R35  Free boundary problems
% 49J20  Optimal control problems involving partial differential equations
% 60J60  Diffusion processes


% AMS keywords (used in AMS journals)

% Acknowledge support
\thanks{PF was partially supported by NSF grant DMS-1059206.}
%CP 1.11.2016: Removed
%CP was partially supported by a Rutgers University fellowship. }

\maketitle
%\tableofcontents
%\listoffigures

\section{Introduction}
\label{sec:Introduction}
\subsection{Overview}
\label{subsec:Overview}
There is a distinguished history of research on local supremum estimates, Harnack inequalities, and local $C^\alpha$ estimates and $C^\alpha$ regularity for weak solutions to equations,
$$
Au = f \quad\hbox{a.e. on }\sO, \quad u = g  \quad\hbox{on } \partial\sO,
$$
defined by an elliptic partial differential operator,\footnote{We employ the Einstein summation convention with $1 \leq \mu,\nu\leq n$.}
\begin{equation}
\label{eq:A_nondivergence_form}
%PF 2-25-2016 changed \tilde b^i to b^i for consistency with later usage
%CP 2.25.2016: Ok
%PF 2-27-2016: Using Greek letters for 1 \leq \alpha, \beta \leq n to avoid conflict with later 1 \leq i, j \leq n-1
%PF 2-29-2016: Rewritten
Au=-\bar a^{\mu\nu}u_{z_\mu z_\nu}-b^\mu u_{z_\mu}+cu,
\end{equation}
whose coefficient matrix, $(\bar a^{\mu\nu})$, is Lipschitz but which fails to be strictly or uniformly elliptic on an open subset $\sO\subset\RR^n$ (for $n\geq 2$), in the sense of \cite[p. 31]{GilbargTrudinger}. For a selection of such results, see \cite{Chanillo_Wheeden_1986, Fabes_Kenig_Serapioni_1982a, DiFazio_Fanciullo_Zamboni_2010, DiFazio_Zamboni_2006, Franchi_Serapioni_1987, Kohn_Nirenberg_1967, Mohammed_2002, Murthy_Stampacchia_1968, Pingen_2008, Stredulinsky, Zamboni_2002} and references contained therein. In those articles, Dirichlet boundary conditions are imposed on the \emph{full} boundary, $\partial\sO$, in order to obtain local supremum estimates and $C^\alpha$ regularity which hold up to $\partial\sO$.

However, it is known from work of G. Fichera \cite{Fichera_1956, Fichera_1960} and O. A. Ole{\u\i}nik and E. V. Radkevi{\v{c}} \cite{Oleinik_Radkevic, Radkevich_2009a, Radkevich_2009b}, building on prior observations of M. V. Keldy{\v{s}} \cite{Keldys_1951}, that when $A$ is \emph{boundary-degenerate} --- that is, $(\bar a^{\mu\nu})$ fails to be locally strictly elliptic along a non-empty open portion $\Gamma_0\subseteqq\partial\sO$ of the boundary --- then refined weak maximum principles imply that the boundary value problem or associated variational equation may have a unique solution, $u$ in $C^2(\sO)\cap C(\bar\sO)$ or $W^{1,2}(\sO)$ respectively, with Dirichlet boundary condition prescribed only along a \emph{part} of the boundary, $\Gamma_1 := \partial\sO\less\bar\Gamma_0$ (the `non-degenerate boundary') and no boundary condition along $\Gamma_0$ (the `degenerate boundary'). However, the development of local supremum estimates, Harnack inequalities, and H\"older continuity up to $\Gamma_0$ for solutions to variational equations defined by boundary-degenerate elliptic partial differential operators --- where \emph{no} boundary condition is imposed along $\Gamma_0$ --- is far less well developed and,
%CP 1.11.2016: Maybe we should add here the work of Daskalopoulos-Hamilton, Epstein-Mazzeo
%PF 2-27-2016: These authors consider variational equations/inequalities and we do mention there work later on, but if you would like to add more, please do as you are more familiar with Epstein-Mazzeo's work
with the exception of the Habilitation thesis of H. Koch \cite{Koch} (about which we shall say more below), there are far fewer results despite the need from important applications.

We shall consider suitably defined \emph{weak} solutions, $u$, to the elliptic boundary value problem,
\begin{equation}
\label{eq:IntroBoundaryValueProblem}
Au = f \quad\hbox{on }\sO, \quad u = g  \quad\hbox{on } \Gamma_1,
\end{equation}
and the elliptic obstacle problem with \emph{partial} Dirichlet boundary condition (see Figure \ref{fig:domain}),
\begin{equation}
\label{eq:IntroObstacleProblem}
\min\{Au-f,u-\psi\} = 0  \quad\hbox{a.e on }\sO, \quad u = g  \quad\hbox{on } \Gamma_1,
\end{equation}
where $\psi:\sO\to\RR$ is an obstacle function which is compatible with the Dirichlet boundary condition in the sense that
\begin{equation}
\label{eq:ObstacleFunctionLessThanZero}
\psi \leq g \quad\hbox{on }\Gamma_1.
\end{equation}
We note that obstacle problems are not considered by Koch in \cite{Koch}.
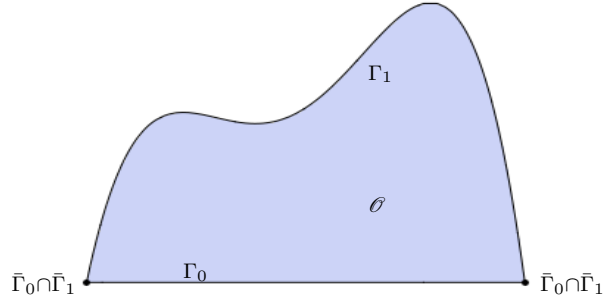
\begin{figure}
\centering
\begin{picture}(200,150)(0,0)
\put(0,0){\includegraphics[width=200pt,height=150pt,clip=true,trim=0pt 0pt 0pt 50pt]{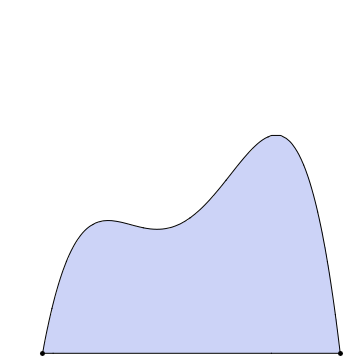}}
\put(-5,0){$\scriptstyle \bar\Gamma_0\cap \bar\Gamma_1$}
\put(60,5){$\scriptstyle \Gamma_0$}
\put(130,80){$\scriptstyle \Gamma_1$}
\put(130,30){$\scriptstyle \sO$}
\put(195,0){$\scriptstyle \bar\Gamma_0\cap \bar\Gamma_1$}
\end{picture}
\caption{Boundaries and corner points for the
%CP 1.23.2016: Removed 'Heston'
%Heston
%PF 2-27-2016: OK
elliptic boundary value and obstacle problems.}
\label{fig:domain}
\end{figure}
The purpose of this article is then to establish the following results for a variational equation corresponding to \eqref{eq:IntroBoundaryValueProblem}
defined by
%CP 1.23.2016: Rephrased below
%PF 2-27-2016: OK
%an example of a boundary-degenerate operator --- the Heston operator \cite{Heston1993} --- which has wide application in mathematical finance:
a class of boundary-degenerate operators that includes the Heston operator \cite{Heston1993}, which has wide application in mathematical finance:
\begin{enumerate}
\item Local supremum estimate up to $\partial\sO$ for a subsolution;
\item A Harnack inequality for a non-negative solution on open subsets $\sO'\Subset\sO\cup\Gamma_0$ when $f=0$ on $\sO$;
\item A strong maximum principle for a subsolution;
\end{enumerate}
and, in the case of a solution, $u$, to a variational equation corresponding to \eqref{eq:IntroBoundaryValueProblem} or variational inequality corresponding to \eqref{eq:IntroObstacleProblem},
\begin{enumerate}
\item[(4)]\setcounter{enumi}{4} $C^\alpha$ regularity up to $\partial\sO$, including the `corner points' where $\Gamma_0$ and $\Gamma_1$ meet, and a local $C^\alpha$ estimate;
\end{enumerate}
where in each of these results, points in $\Gamma_0$ have the same role as those in the interior, $\sO$, and \emph{no} boundary condition is prescribed along $\Gamma_0$. The supremum and $C^\alpha$ estimates for $u$ are expressed in terms of integral norms of $u$, the source function, $f$, the boundary data, $g$, and, in the case of the variational inequality, the obstacle function, $\psi$. Unlike the analogous classical results described by Gilbarg and Trudinger \cite{GilbargTrudinger} for strictly elliptic operators --- for example, local supremum estimates up to $\partial\sO$ \cite[Theorem 8.25]{GilbargTrudinger} or local $C^\alpha$ estimates and regularity up to $\partial\sO$ \cite[Theorem 8.29]{GilbargTrudinger} --- or their analogues for degenerate-elliptic operators in the articles cited above (aside from \cite{Koch}), we do not need to assume that $u$ is bounded or $C^\alpha$ along $\Gamma_0$: those properties are implied by the variational equation alone. In \S \ref{subsec:Survey}, we provide a detailed comparison with previous related results for solutions to variational equations defined by `degenerate elliptic' operators. Our companion article \cite{Feehan_Pop_higherregularityweaksoln} develops higher-order regularity properties up to $\Gamma_0$ for weak solutions.
\begin{comment}
%PF 2-18-2016 Next sentence seems redundant now
%CP 2.24.2016: Ok
We obtain the results in this article using methods which should readily allow generalizations (in a subsequent article) to a much broader class of boundary-degenerate operators which can also be expressed in divergence form and which is described in \S \ref{subsec:Class_noncoercive_bilinear_maps}.
\end{comment}

Some of the motivation for developing these results can be inferred from the work of P. Daskalopoulos and R. Hamilton \cite{DaskalHamilton1998}, C. L. Epstein and R. Mazzeo \cite{Epstein_Mazzeo_2010, Epstein_Mazzeo_annmathstudies}, H. Koch \cite{Koch}, and the authors \cite{Feehan_Pop_higherregularityweaksoln, Feehan_Pop_elliptichestonschauder, Feehan_Pop_mimickingdegen_pde}, where one discovers that the imposition of a Dirichlet boundary condition along $\Gamma_0$ can limit the regularity of the solution, $u$, to be at most $C^\alpha$ up to $\Gamma_0$, whereas employing suitable weighted H\"older or Sobolev spaces to facilitate solving the partial boundary problem (with Dirichlet boundary condition prescribed only along $\Gamma_1$) will yield a solution which is $C^\infty$ up to $\Gamma_0$ (if the coefficients of $A$ and source function $f$ are also $C^\infty$ up to $\Gamma_0$). Applications illustrate that the imposition of a boundary condition along $\Gamma_0$ is often not physically justified, as exemplified in work of Daskalopoulos and Hamilton and Koch on the porous medium equation, Daskalopoulos and the author \cite{Daskalopoulos_Feehan_statvarineqheston, Daskalopoulos_Feehan_optimalregstatheston} on stochastic volatility models in mathematical finance, E. Ekstr\"om and J. Tysk \cite{Ekstrom_Tysk_bcsftse} on interest-rate models in mathematical finance, and Epstein and Mazzeo on Wright-Fisher diffusion models in mathematical biology, and many other examples. Instead, the relevant physical property sought is rather that the solution, $u$, be sufficiently smooth up to $\Gamma_0$.

When the boundary-degenerate operator, $A$, can be expressed in both divergence and non-divergence forms (as we assume here), one has a choice of employing a Schauder approach to existence and regularity theory, as in \cite{DaskalHamilton1998, Epstein_Mazzeo_annmathstudies, Feehan_classical_perron_elliptic, Feehan_Pop_elliptichestonschauder, Feehan_Pop_mimickingdegen_pde}, or a variational approach as in \cite{Daskalopoulos_Feehan_statvarineqheston, Feehan_Pop_higherregularityweaksoln, Koch}. However, for certain questions, the variational approach can have advantages over a Schauder approach. For example, it appears to be a challenging problem to use purely Schauder methods to prove that the solution, $u$, is $C^\alpha$ up to the `corner points', where the degenerate and non-degenerate boundary portions, $\Gamma_0$ and $\Gamma_1$, meet; see \cite{Feehan_classical_perron_elliptic, Feehan_Pop_higherregularityweaksoln, Feehan_Pop_elliptichestonschauder} for discussions of this difficulty. As shown by Daskalopoulos and one of the authors (Feehan) \cite{Daskalopoulos_Feehan_statvarineqheston}, a framework for solving a non-coercive variational equation defined with the aid of appropriate weighted Sobolev spaces is readily extended to include variational inequalities.

Furthermore, Daskalopoulos and Feehan use the Harnack inequality and continuity (up to $\Gamma_0$) developed in this article for a solution, $u$, to a variational inequality as important stepping stones in their proof of $C^{1,1}$ regularity (up to $\Gamma_0$) of a solution to an obstacle problem arising in mathematical finance.
When $A$ is as in \eqref{eq:OperatorHestonIntro} and $f=0$, the solution, $u$, to the obstacle problem \eqref{eq:IntroObstacleProblem} can be interpreted
as the value function for a perpetual American-style barrier option on a
%CP 1.23.2016: Added the below to include more general processes
%PF 2-27-2016: OK
generalization of the Heston stochastic volatility asset price process \cite{Heston1993}, with payoff function $\psi$ and barrier condition $g$ on $\Gamma_1$. The choice $\psi(x,y)=(K-e^x)^+$, for $(x,y)\in\RR\times\RR_+$, yields the price of an American-style put, where $x$ represents the asset log-price, $y$ is the asset variance, and $K>0$ is the strike.

\subsection{Summary of main results}
\label{subsec:Summary}
We shall state a selection of our main results here and then refer the reader to our guide to this article in \S \ref{subsec:Guide} for more of our results on existence, uniqueness and regularity of solutions to variational equations and inequalities and corresponding obstacle problems.

\subsubsection{Mathematical preliminaries}
\label{subsubsec:IntroPreliminaries}
%CP 1.23.2016: Introduced here the generalization to multidimensions and variable coefficients
%PF 2-27-2016: OK; added footnote
%PF 2-29-2016: Simplified description of A and assumptions
In this article, we shall study boundary-degenerate elliptic operators \eqref{eq:A_nondivergence_form} of the specific form, for all $v \in C^\infty(\sO)$,
%\footnote{We employ the Einstein summation convention with $1 \leq i,j \leq n-1$.}
%PF 2-23-2016 missing "+ c(z)v(z)" added and changed Au --> Av on LHS
%CP 2.25.2016: Ok
%PF 2-23-2016: In the discussion of the variable coefficient operator, you have to be consistent in the use of index ranges: Is it 1 \leq i \leq n or n-1? If the former, then the summation convention below does not make sense; if the latter, the coefficient bound \eqrefl{eq:Operator_A_boundedness} is incomplete. Also, there is a notation conflict with later usage -- if x_i is used for a coordinate, then it cannot also mean a marked point (needs to be x^i or vice versa).
%CP 2.25.2016: What x_i means in the later section of the paper is explained. I don't find it confusing. If you think that it is please feel free to change the notation.
%PF 2-27-2016: The coefficients should not assumed to be defined on all of \HH; changed below
\begin{equation}
\label{eq:Operator_simplified}
Av(z) = -ya^{\mu\nu}(z)v_{z_\mu z_\nu}(z) - b^\mu(z)v_{z_\mu}(z) + c(z)v(z),
\quad\text{a.e. } z \in \sO,
\end{equation}
so $\bar a^{\mu\nu} = ya^{\mu\nu}$, where we denote $z=(z_1,\ldots,z_n) = (x,y)\in\HH$ with $x = (x_1,\ldots,x_{n-1}) \in\RR^{n-1}$ and $x_n = y\in\RR_+$. We require that the coefficients of the operator $A$ satisfy
\begin{assump}
\label{assump:Operator_A}
There are positive constants, $\beta$, $\Lambda$ and $\nu_0$, with the following significance.
\begin{enumerate}
%PF 2-18-2016 I think it best not to ``hardcode'' the enumeration format.
%CP 2.24.2016: Ok
%PF 2-27-2016: We can't have i \leq n in one context, i \leq n-1 in another if we're using summation convention
\item The coefficients $a^{\mu\nu}$ belong to $W^{1,\infty}(\sO)$ and $c$ belongs to $L^{\infty}(\sO)$;
\item %PF 2-18-2016 Description of the matrix coefficients seemed a bit odd
%Letting $a^{ni}:=a^{in}$, for all $1\leq i\leq n-1$, the matrix coefficient $(a^{ij}(z))$ is strictly elliptic, that is,
%CP 2.24.2016: Ok
%PF 2-25-2016: We may as well just assume the coefficient matrix is symmetric
%CP 2.25.2016: Ok
The coefficient matrix $(a^{\mu\nu}(z))$
%obeys $a^{ni}=a^{in}$ for $1\leq i\leq n-1$ and is
is symmetric and strictly elliptic,
\begin{equation}
\label{eq:Operator_A_ellipticity}
%PF 2-27-2016: Missing terms added
\nu_0|\xi|^2\leq a^{\mu\nu}(z)\xi_\mu\xi_\nu,
\quad\forall\, \xi\in\RR^n,\quad\hbox{for a.e. } z\in \sO;
\end{equation}
\item There are functions $\hat b_\mu\in L^{\infty}(\sO)$ such that
\begin{equation}
\label{eq:Operator_A_b}
%PF 2-25-2016: According the definition of A and the integration by parts formula, there should be a distinction between b^i (factor 2\beta below) and b^n (factor beta, stated separately)
%CP 2.25.2016: I don't get that. When we integrate by parts the term 2a^{in} u_{x_i y}, it should be split into a^{in} u_{x_i y} and a^{in} u_{x_i y}. In the first term, we use x_i to integrate by parts, and in the second term, use y. In this way, we obtain the part containing u_{x_i}v_{x_j} to be coercive, and there is no 2.
%PF 2-27-2016: OK, but changed below since range for i must be 1 \leq i \leq n-1; otherwise, summation convention will not be correct
%CP 3.7.2016: Changed the sign in front on \hat b^{\mu} from + to -. If the sign is not changed here, then it must be changed in the bilinear form
%PF 3-7-2016: OK
b^\mu = \beta a^{\mu n} - y \hat b^\mu, \quad 1\leq \mu \leq n;
\end{equation}
%PF 2-27-2016: We can't have i \leq n in one context, i \leq n-1 in another if we're using summation convention
\item The coefficients obey the bound
\begin{equation}
\label{eq:Operator_A_boundedness}
\max_{1\leq \mu,\nu \leq n}\|a^{\mu\nu}\|_{W^{1,\infty}(\sO)} +
 \max_{1\leq
%CP 2.29.2016: Changed i to \mu
\mu
%i
\leq n}
%CP 2.28.2016: a^{in} and a^{nn} should have norm W^{1,\infty}(\sO) instead of L^{\infty}(\sO). Changed below
%\|a^{in}\|_{L^{\infty}(\sO)}
\|\hat b^\mu\|_{L^{\infty}(\sO)}
+
%\|a^{nn}\|_{L^{\infty}(\sO)}
\|c\|_{L^{\infty}(\sO)} \leq \Lambda.
\end{equation}
\end{enumerate}
\end{assump}

%CP 1.23.2016: Moved this condition to the example about the Heston process
%PF 2-23-2016: This changes the numbering of the results, so we should replace by a remark or comment or something to avoid that
%CP 2.24.2016: The numbering was not changed by removing the assumptions below because they were replaced by 2 other (Assumption 1.1 and Example 1.2)
\begin{comment}
Throughout this article, the coefficients of the Heston operator, $A$ in \eqref{eq:OperatorHestonIntro}, are required to obey

\begin{assump}[Ellipticity condition for the coefficients of the Heston operator]
\label{assump:HestonCoefficients}
The coefficients defining $A$ in \eqref{eq:OperatorHestonIntro} are constants obeying
\begin{equation}
\label{eq:EllipticHeston}
\sigma \neq 0, \quad -1< \varrho < 1,
\end{equation}
and $\kappa>0$, $\theta>0$, $r\geq 0$, and $q \in \RR$.
\end{assump}
\end{comment}

We shall consider variational solutions to \eqref{eq:IntroBoundaryValueProblem} and \eqref{eq:IntroObstacleProblem}, so we introduce our weighted Sobolev spaces. For $1\leq q<\infty$, let
%PF 2-25-2016: These weighted Sobolev spaces now need motivation. Why are they the "right" ones?
%CP 2.25.2016: Because we can write the equation in divergence form and obtain a coercive and continuous bilinear form. This is written in the sentence at the end of the paragraph that introduces the spaces and the bilinear form.
%CP 2.25.2016: I replaced coercive with the Garding inequality at the end of the paragraph.
\begin{subequations}
\begin{align}
\label{eq:LqWeightedSpace}
L^q(\sO,\fw) &:= \{u\in L^1_{\loc}(\sO): \|u\|_{L^q(\sO,\fw)} < \infty\},
\\
\label{eq:H1WeightedSobolevSpace}
H^1(\sO,\fw) &:= \{u \in L^2(\sO,\fw): (1+y)^{1/2}u, \ y^{1/2}|Du| \in L^2(\sO,\fw)\},
\\
\label{eq:H2WeightedSobolevSpace}
H^2(\sO,\fw) &:= \{u \in L^2(\sO,\fw): (1+y)^{1/2}u, \ (1+y)|Du|, \ y|D^2u| \in L^2(\sO,\fw)\},
\end{align}
\end{subequations}
where
%CP 1.23.2015: Changed
%PF 2-27-2016: OK
$Du$ denotes the gradient of $u$, $D^2u$ denotes the Hessian of $u$,
%$Du = (u_x,u_y)$, $D^2u = (u_{xx}, u_{xy}, u_{yx}, u_{yy})$,
with all derivatives of $u$ being defined in the sense of distributions, and
\begin{subequations}
\begin{align}
\label{eq:LqNormHeston}
\|u\|_{L^q(\sO,\fw)}^q &:= \int_\sO |u|^q\fw\,dx\,dy,
\\
\label{eq:H1NormHeston}
\|u\|_{H^1(\sO,\fw)}^2 &:= \int_\sO\left(y|Du|^2 + (1+y)u^2\right)\fw\,dx\,dy,
\\
\label{eq:H2NormHeston}
\|u\|_{H^2(\sO,\fw)}^2 &:= \int_\sO\left( y^2|D^2u|^2 + (1+y)^2|Du|^2 + (1+y)u^2\right)\,\fw\,dx\,dy,
\end{align}
\end{subequations}
%PF 5.28.2013 Added footnote
with weight function $\fw:\HH\to(0,\infty)$ given by
%CP 1.23.2016: We should remove this when we consider the generalization to variable coefficients
\begin{comment}
\footnote{The factor $y^{\beta-1}$ is the important one in \eqref{eq:HestonWeight}; inclusion of the factor $e^{-\mu y}$ simplifies the structure of the bilinear form \eqref{eq:BilinearForm} slightly and, together with the factor $e^{-\gamma|x|}$, ensures that $\HH$ has finite volume with respect to the weight, $\fw$.}
\end{comment}
\begin{equation}
\label{eq:HestonWeight}
%CP 1.23.2016: I removed the exponential factors from the definition of \fw
%CP 3.7.2016: Added the exponential terms
%PF 3-7-2016: OK, changed gamma to tau to avoid conflict
\fw(x,y) := y^{\beta-1}e^{-\tau|x|-\mu y}, \quad\forall\, (x,y) \in \HH,
%e^{-\gamma|x|-\mu y}
\end{equation}
where $\tau$ and $\mu$ are nonnegative constants.
It will be convenient in our analysis to write $A$ from \eqref{eq:Operator_simplified}, for all $v\in C^{\infty}(\sO)$, in the equivalent form,\footnote{We employ the Einstein summation convention with $1 \leq i,j \leq n-1$.}
\begin{equation}
\label{eq:Operator}
\begin{aligned}
Av(z) &= -y\left(a^{ij}(z)v_{x_ix_j}(z) + 2a^{in}(z)v_{x_iy}(z) + a^{nn}(z)v_{yy}(z)\right)
\\
&\qquad - b^i(z)v_{x_i}(z) - b^n(z)v_y(z) + c(z)v(z),
\quad\text{a.e. } z \in \sO.
\end{aligned}
\end{equation}
%CP 1.23.2016: Introduced the bilinear form here
%PF 2-25-2016: Not consistent -- sometimes C_0, sometimes C_c. Changed to C_0.
%For all $u,v\in C^{\infty}_c(\HH)$, we define
%CP 3.7.2016: Changed the bilinear form
%PF 3-7-2016: OK
For all $u,v\in C^{\infty}_0(\HH)$, we define
\begin{equation}
\label{eq:Operator_A_bilinear_form}
\begin{aligned}
\fa(u,v):=(Au,v)_{L^2(\sO,\fw)}
&=
\int_{\sO} \left(a^{ij}u_{x_i}v_{x_j}+a^{in}(u_{x_i}v_y+u_yv_{x_i})+a^{nn}u_yv_y\right) y\fw \,dxdy
\\
&\quad +\int_{\sO} \left(\partial_{x_j}a^{ij}+\partial_y a^{in}+\hat b_i-\tau a^{ij}\frac{x_j}{|x|}-\mu a^{in}\right)u_{x_i}v y\fw \,dxdy
\\
&\quad +\int_{\sO} \left(\partial_{x_i}a^{in}+\partial_y a^{nn}+\hat b_n-\tau a^{in}\frac{x_i}{|x|}-\mu a^{nn}\right)u_yv y\fw \,dxdy
\\
&\quad +\int_{\sO} cuv\fw \,dxdy,
\end{aligned}
\end{equation}
and we call $\fa$ the \emph{bilinear form associated with the operator $A$}.
%Notice that assumptions \eqref{eq:Operator_A_ellipticity} and \eqref{eq:Operator_A_boundedness} ensure that the bilinear form $\fa:H^1(\sO,\fw)\times H^1(\sO,\fw)\rightarrow\RR$ is coercive and continuous.
%PF 2-25-2016: It's not coercive. That is one of the main issues in proofs of existence of solutions to these VEs or VIs by Hilbert space methods
%CP 2.25.2016: Replaced with Garding inequality
%PF 2-25-2016: Added back ellipticity
The assumptions \eqref{eq:Operator_A_ellipticity} and \eqref{eq:Operator_A_boundedness} ensure that the bilinear form $\fa:H^1(\sO,\fw)\times H^1(\sO,\fw)\rightarrow\RR$ is continuous and satisfies the G\r{a}rding inequality
%PF 2-27-2016: Added motivation for norm definition
and this motivates the definition \eqref{eq:H1NormHeston} of the weighted Sobolev space, $H^1(\sO,\fw)$.
%CP 3.7.2016: Added
%PF 3-7-2016: OK
In definition \eqref{eq:HestonWeight} of the weight $\fw$, the power term $y^{\beta-1}$ is required in order to obtain a bilinear form $\fa$ as in \eqref{eq:Operator_A_bilinear_form} that is continuous and satisfies the G\r{a}rding inequality. The role of the exponential term $e^{-\tau|x|-\mu y}$ is mainly to ensure that the measure of subsets $\sO\subseteqq\HH$ is finite, when $\tau$ and $\mu$ are positive constants. Even though this property is used extensively in the results obtained in \cite{Daskalopoulos_Feehan_statvarineqheston} and \cite{Feehan_maximumprinciple}, it does not play any role in the proofs of the purely local results given in \S \ref{sec:SobolevPoincare}, \S \ref{sec:JohnNirenberg}, \S \ref{sec:SupremumEstimates}, \S \ref{sec:HolderContinuityVariationalEquation} and \S \ref{sec:Harnack}, but we include the exponential term in the definition of the weight $\fw$ for
%PF 3-7-2016: Changed
%completeness.
the sake of consistency with \cite{Daskalopoulos_Feehan_statvarineqheston} and because positivity of $\tau$ is used in the proof of Claim \ref{claim:EssBoundedPenaltyTerm} in Section \ref{sec:HolderContinuityVariationalInequality}.

%CP 1.23.2016: Wrote Heston as an example
%PF 2-27-2016: OK
%CP 3.7.2016: Changed this example
%PF 3-7-2016: OK
\begin{exmp}[Heston operator]
\label{exmp:Heston}
A particular example of a degenerate operator as in \eqref{eq:Operator} is the generator of the two-dimensional Heston stochastic volatility process with killing \cite{Heston1993},
\begin{equation}
\label{eq:OperatorHestonIntro}
Av := -\frac{y}{2}\left(v_{xx} + 2\varrho\sigma v_{xy} + \sigma^2 v_{yy}\right) - \left(r-q-\frac{y}{2}\right)v_x - \kappa(\theta-y)v_y + rv, \quad v\in C^\infty(\HH),
\end{equation}
where $\kappa>0$, $\theta>0$, $r\geq 0$, and $q \in \RR$. We express the Heston operator $A$ in \eqref{eq:OperatorHestonIntro} in divergence form as in \eqref{eq:Operator_A_bilinear_form} by choosing the weight $\fw$ with
\begin{equation}
\label{eq:DefnBetaMu}
\beta := \frac{2\kappa\theta}{\sigma^2} \quad\hbox{and}\quad \mu := \frac{2\kappa}{\sigma^2},
\end{equation}
and $\tau$ is a positive constant; see \cite[\S 1.1]{Daskalopoulos_Feehan_statvarineqheston}. To ensure that the strict ellipticity condition \eqref{eq:Operator_A_ellipticity} is satisfied, we assume that
\begin{equation}
\label{eq:EllipticHeston}
\sigma \neq 0 \quad\hbox{and}\quad -1< \varrho < 1.
\end{equation}
We notice that the condition \eqref{eq:Operator_A_b} is satisfied only if
\begin{equation}
\label{eq:Condition_b_1_Heston}
r-q-\frac{\kappa\theta\varrho}{\sigma} = 0,
\end{equation}
and this can be accomplished by using a simple affine change of variables on $\RR^2$ which maps $(\HH,\partial\HH)$ onto $(\HH,\partial\HH)$, as described in \cite[Lemma 2.2]{Daskalopoulos_Feehan_statvarineqheston}. Then the bilinear form associated with the Heston operator, $A$, in \eqref{eq:OperatorHestonIntro} is given by
%PF 3-7-2016: \rho --> \varrho
\begin{equation}
\label{eq:HestonWithKillingBilinearForm}
\begin{aligned}
\fa(u,v) &:= \frac{1}{2}\int_\sO\left(u_xv_x + \varrho\sigma u_yv_x + \varrho\sigma u_xv_y + \sigma^2u_yv_y\right)y\,\fw\,dx\,dy
\\
&\quad - \frac{1}{2}\int_\sO\left(\tau\sign(x)+\mu\varrho\sigma-1\right)u_x v y\,\fw\,dx\,dy
\\
&\quad - \frac{1}{2}\int_\sO\tau\varrho\sigma\sign(x)u_y v y\,\fw\,dx\,dy
+ \int_\sO ruv\,\fw\,dx\,dy, \quad\forall\, u, v \in H^1(\sO,\fw).
\end{aligned}
\end{equation}
%PF 2-29-2016 added sentence
This completes our discussion of this example.
\end{exmp}

%PF 2-18-2016 added sentence
%CP 2.24.2016: Ok
We now return to the general setting described prior to Example \ref{exmp:Heston}.
Given a subset $T\subset\partial\sO$ we let $H^1_0(\sO\cup T,\fw)$ be the closure in $H^1(\sO,\fw)$ of $C^\infty_0(\sO\cup T)$. Given a source function $f\in L^2(\sO,\fw)$, we call a function
$u\in H^1(\sO,\fw)$ a \emph{solution} to the variational equation
%CP 1.23.2016: Changed the operator
%PF 2-27-2016: OK
%for the Heston operator
defined by the operator $A$ in \eqref{eq:Operator}, if
\begin{equation}
\label{eq:IntroHestonWeakMixedProblemHomogeneous}
\fa(u,v) = (f,v)_{L^2(\sO,\fw)}, \quad \forall\, v \in H^1_0(\sO\cup\Gamma_0,\fw).
\end{equation}
We call $u$ a \emph{subsolution} to \eqref{eq:IntroHestonWeakMixedProblemHomogeneous} if $\fa(u,v) \leq (f,v)_{L^2(\sO,\fw)}$ for all nonnegative test functions, $v$, and call $u$ a \emph{supersolution} to \eqref{eq:IntroHestonWeakMixedProblemHomogeneous} if $-u$ is a subsolution.

Given $g\in H^1(\sO,\fw)$, we say that $u$ obeys an (inhomogeneous) Dirichlet boundary condition $u=g$ on $\Gamma_1$ in the sense of $H^1$ if
$$
u-g \in H^1_0(\sO\cup\Gamma_0,\fw),
$$
and, of course, a homogeneous Dirichlet boundary condition on $\Gamma_1$ if $g=0$.

If $u \in H^2(\sO,\fw)$, we recall from \cite{Daskalopoulos_Feehan_statvarineqheston} that $u$ is a solution to \eqref{eq:IntroBoundaryValueProblem} if and only if $u\in H^1_0(\sO\cup\Gamma_0,\fw)$ and $u$ is a solution to
\eqref{eq:IntroHestonWeakMixedProblemHomogeneous}.

\begin{defn}[Balls with respect to the Euclidean metric]
\label{defn:Balls_with_respect_to_Euclidean_metric}
We let
\begin{align}
\label{eq:Euclidean_balls_relative_to_the_half_space}
\EE_R(z_0) &:= \{z\in\HH: |z-z_0|<R\},
\\
\label{eq:Euclidean_balls_relative_to_a_subdomain}
E_R(z_0) &:= \{z\in\sO: |z-z_0|<R\},
\end{align}
for any given $z_0 \in \bar\HH$ and $R>0$.
\end{defn}

We say that an open subset, $U\subset\HH$, obeys an \emph{exterior cone condition relative to $\HH$ at a point $z_0\in \partial U$} if there exists a finite, right circular cone $K = K_{z_0}\subset \bar\HH$ with vertex $z_0$ such that $\bar U\cap K_{z_0} = \{z_0\}$ (compare \cite[p. 203]{GilbargTrudinger}).  An open subset, $U\subset\HH$, obeys a \emph{uniform exterior cone condition relative to $\HH$ on $T\subset\partial U$} if $U$ satisfies an exterior cone condition relative to $\HH$ at every point $z_0\in T$ and the cones $K_{z_0}$ are all congruent to some fixed finite cone, $K$ (compare \cite[p. 205]{GilbargTrudinger}). Recall that $\Gamma_0$ is the interior of the portion, $\bar\sO\cap\partial\HH$, of the boundary, $\partial\sO$, of the open subset $\sO\subseteqq\HH$.

\begin{defn}[Interior and exterior cone conditions]
\label{defn:RegularDomain}
Let $K$ be a finite, right circular cone. We say that $\sO$ obeys \emph{interior and exterior cone conditions at $z_0\in \bar\Gamma_0\cap \bar \Gamma_1$ with cone $K$} if the open subsets $\sO$ and $\HH\less\bar\sO$ obey exterior cone conditions relative to $\HH$ at $z_0$ with cones congruent to $K$. We say that $\sO$ obeys \emph{uniform interior and exterior cone conditions on $\bar\Gamma_0\cap \bar \Gamma_1$ with cone $K$} if the open subsets $\sO$ and $\HH\less\bar\sO$ obey exterior cone conditions relative to $\HH$ at each point $z_0\in \bar\Gamma_0\cap \bar \Gamma_1$ with cones congruent to $K$.
\end{defn}

\subsubsection{Boundary local supremum bounds}
\label{subsubsec:LocalSupremumBounds}
%CP 1.23.2016: I reformulated this paragraph for the more general operator A
%PF 2-27-2016: OK
\begin{comment}
Recall that $\kappa\theta$ is the coefficient of $-v_y$ in the definition \eqref{eq:OperatorHestonIntro} of $A$ when $y=0$.
As in \cite[Theorem I.1.1]{DaskalHamilton1998}, the assumption that $\kappa\theta$ is positive is of crucial importance. We notice from \eqref{eq:DefnBetaMu} that $\beta = 2\kappa\theta/\sigma^2$ must then be positive and so the weight $\fw$ belongs to $L^1(\HH)$. Therefore, the volumes of bounded subsets in $\HH$ are finite with respect to the weights $y^{\beta-1} \,dx\,dy$, and $\fw \,dx\,dy$, a fact which we repeatedly use in this article. Clearly, if $\beta$ were negative, then $\fw$ would belong to $L^1_{\loc}(\HH)$, but not to $L^1(\HH)$. We rely on the assumption that $\beta>0$ in the statements and proofs of the local supremum estimates.
\end{comment}
The volumes of bounded subsets in $\HH$ are finite with respect to the weight $y^{\beta-1} \,dx\,dy$, when $\beta>0$, a fact which we repeatedly use in this article. We rely on the assumption that $\beta>0$ in the statements and proofs of the local supremum estimates.

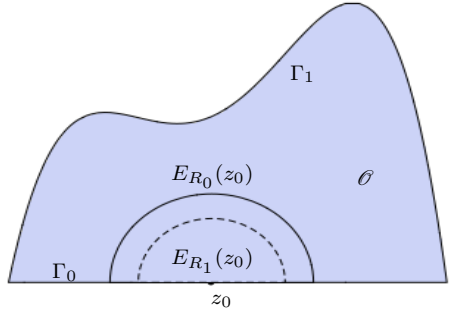
\begin{figure}
\centering
\begin{picture}(200,150)(0,0)
\put(0,0){\includegraphics[width=200pt,height=150pt,clip=true,trim=0pt 0pt 0pt 50pt]{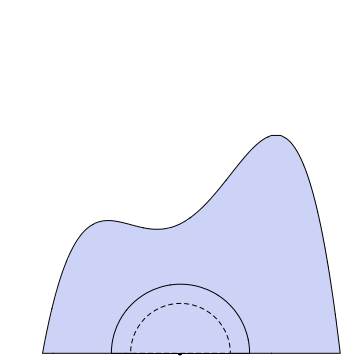}}
\put(40,5){$\scriptstyle \Gamma_0$}
\put(130,80){$\scriptstyle \Gamma_1$}
\put(155,40){$\scriptstyle \sO$}
\put(85,42){$\scriptstyle E_{R_0}(z_0)$}
\put(85,10){$\scriptstyle E_{R_1}(z_0)$}
\put(100,-5){$\scriptstyle z_0$}
\end{picture}
\caption{Concentric half-balls centered at a `degenerate boundary interior' point, $z_0\in\Gamma_0$.}
\label{fig:domain_and_middlehalfballs}
\end{figure}

We have the following analogues of \cite[Proposition 4.5.1]{Koch} and \cite[Theorem 8.15]{GilbargTrudinger}, but now for the cases of a `degenerate-boundary interior' point, $z_0\in\Gamma_0$, and a `degenerate boundary corner point', $z_0\in\bar\Gamma_0\cap\bar\Gamma_0$; see Figures \ref{fig:domain_and_middlehalfballs} and \ref{fig:domain_and_righthalfballs}, respectively. Though Koch allows for points in the interior of $\Gamma_0$, there is no analogue in \cite{Koch} of our Theorem \ref{thm:MainSupremumEstimatesBoundary}, which allows for corner points, while Gilbarg and Trudinger \cite{GilbargTrudinger} only allow for boundary points where the elliptic partial differential operator is strictly elliptic.

\begin{thm}[Supremum estimates near points in $\Gamma_0$]
\label{thm:MainSupremumEstimatesInterior}
Let $s>n+\beta$ and let $R_0$ be a positive constant. Then there are positive constants, $C=C(\Lambda,n,\nu_0,R_0,s)$ and $R_1=R_1(R_0)<R_0$, such that the following holds. Let $\sO\subseteqq\HH$ be an open subset. If $u \in H^1(\sO,\fw)$
is a subsolution (respectively, supersolution) to the variational equation \eqref{eq:IntroHestonWeakMixedProblemHomogeneous} with source function $f \in L^2(\sO,\fw)$, and $z_0 \in \Gamma_0$ is such that
$\EE_{R_0}(z_0) \subset \sO$, and $f$ obeys
\begin{equation}
\label{eq:Lsfcondition_ball}
f \in L^s(E_{R_0}(z_0),y^{\beta-1}),
\end{equation}
then $u \in L^\infty(E_{R_1}(z_0))$, and
\begin{equation}
\label{eq:MainSupremumEstimates2}
\esssup_{E_{R_1}(z_0)}u(-u)
\leq  C\left( \|f\|_{L^s(E_{R_0}(z_0),y^{\beta-1})} + \|u^+(u^-)\|_{L^2(E_{R_0}(z_0),y^{\beta-1})} \right).
\end{equation}
\end{thm}

\begin{thm}[Supremum estimates near points in $\overline{\Gamma}_0\cap\overline{\Gamma}_1$]
\label{thm:MainSupremumEstimatesBoundary}
Let $K$ be a finite right circular cone, let $s>n+\beta$, and let $R_0>0$ be a positive constant. Then there are positive constants, $C=C(K,\Lambda,n,\nu_0,R_0,s)$ and $R_1=R_1(K,\Lambda,n,\nu_0, R_0)<R_0$, such that the following holds. Let $\sO\subsetneqq\HH$ be an open subset. If $u \in H^1(\sO,\fw)$
is a subsolution (respectively, supersolution) of equation \eqref{eq:IntroHestonWeakMixedProblemHomogeneous} with source function $f \in L^2(\sO,\fw)$ and $z_0 \in \overline{\Gamma}_0\cap\overline{\Gamma}_1$ is such that $\sO$ obeys an interior cone condition at $z_0$ with cone $K$, and
$$
u = 0 \hbox{ on } \Gamma_1\cap\bar E_{R_0}(z_0) \quad\hbox{(in the sense of $H^1$)},
$$
and $f$ obeys \eqref{eq:Lsfcondition_ball}, then $\esssup_{E_{R_1}(z_0)} u (-u) < \infty$ and the estimate \eqref{eq:MainSupremumEstimates2} holds.
\end{thm}

\begin{figure}
\centering
\begin{picture}(200,150)(0,0)
\put(0,0){\includegraphics[width=200pt,height=150pt,clip=true,trim=0pt 0pt 0pt 50pt]{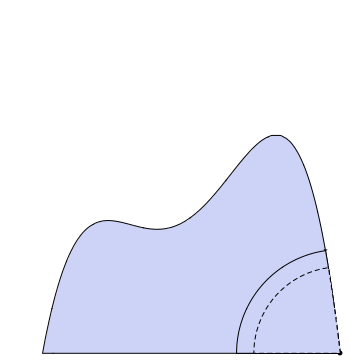}}
\put(50,5){$\scriptstyle \Gamma_0$}
\put(130,80){$\scriptstyle \Gamma_1$}
\put(80,30){$\scriptstyle \sO$}
\put(130,50){$\scriptstyle E_{R_0}(z_0)$}
\put(150,15){$\scriptstyle E_{R_1}(z_0)$}
\put(195,0){$\scriptstyle z_0$}
\end{picture}
\caption{Concentric half-balls centered at a `degenerate boundary corner point', $z_0\in\bar\Gamma_0\cap\bar\Gamma_0$.}
\label{fig:domain_and_righthalfballs}
\end{figure}
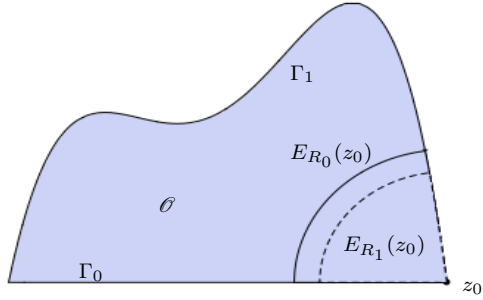

\begin{rmk}[Use of the weight $y^{\beta-1}$ versus $\fw$ in Theorems \ref{thm:MainSupremumEstimatesInterior} and \ref{thm:MainSupremumEstimatesBoundary}]
Notice that on the right-hand-side of estimate \eqref{eq:MainSupremumEstimates2} we have $\|f\|_{L^s(E_{R_0}(z_0),y^{\beta-1})}$ instead of $\|f\|_{L^s(E_{R_0}(z_0),\fw)}$. This allows us to conclude that the constant $C$ appearing in \eqref{eq:MainSupremumEstimates2} is independent of the point $z_0\in\bar\Gamma_0$. By \eqref{eq:HestonWeight}, the weight $\fw$ contains the term
$e^{-\tau|x|}$, which means that the constant $C$ will depend on the $x$-coordinate of the point $z_0\in\bar\Gamma_0$, if we replace $\|f\|_{L^s(E_{R_0}(z_0),y^{\beta-1})}$ by $\|f\|_{L^s(E_{R_0}(z_0),\fw)}$ on the right-hand-side of \eqref{eq:MainSupremumEstimates2}.
\end{rmk}

For $g \in L^{\infty}_{\loc}(\bar\Gamma_1)$ and $z_0 \in\bar\Gamma_0\cap\bar\Gamma_1$ and $R_0>0$, we set
\begin{align*}
M := \esssup_{\Gamma_1\cap B_{R_0}(z_0)} g,
\end{align*}
and define
$$
u^M(z):=(u(z) \vee M)^+\quad\hbox{for a.e. } z \in B_{R_0}(z_0).
$$
We then have the following analogue of \cite[Theorem 8.25]{GilbargTrudinger} which applies to a variational equation defined by strictly elliptic operator and an inhomogeneous Dirichlet boundary condition.

\begin{cor} [Supremum estimates near points in $\overline{\Gamma}_0\cap\overline{\Gamma}_1$ for variational subsolutions with inhomogeneous Dirichlet boundary condition]
\label{cor:MainSupremumEstimatesBoundary}
Let $s>n+\beta$ and let $R_0$ be a positive constant. Then there are positive constants, $C=C(K,\Lambda,n,\nu_0,R_0,s)$  and $R_1=R_1(K,\Lambda,n,\nu_0, R_0)<R_0$, such that the following holds. Let $z_0 \in\bar\Gamma_0\cap\bar\Gamma_1$. If $u \in H^1(\sO,\fw)$
is a subsolution of equation \eqref{eq:IntroHestonWeakMixedProblemHomogeneous} with source function $f \in L^2(\sO,\fw)$ satisfying \eqref{eq:Lsfcondition_ball}, and $g \in H^1(\sO, \fw) \cap L^{\infty}_{\loc}(\bar\Gamma_1)$, in the sense that
\begin{equation}
\label{eq:uBoundaryCondition}
u-g \in H^1_0(\sO\cup\Gamma_0,\fw),
\end{equation}
then $u^M \in L^\infty(E_{R_1}(z_0))$, and
\begin{equation}
\label{eq:MainSupremumEstimates3}
\esssup_{E_{R_1}(z_0)}u^M
\leq  C\left( \|f\|_{L^s(E_{R_0}(z_0),y^{\beta-1})} + \|u\|_{L^2(E_{R_0}(z_0),y^{\beta-1})} + \|g\|_{L^{\infty}(\bar \Gamma_1\cap\bar E_{R_0}(z_0))} \right).
\end{equation}
\end{cor}

\begin{rmk}[Supremum estimates near points in $\overline{\Gamma}_0\cap\overline{\Gamma}_1$ for supersolutions with inhomogeneous Dirichlet boundary condition]
\label{rmk:MainSupremumEstimatesBoundary}
Corollary \ref{cor:MainSupremumEstimatesBoundary} holds for supersolutions to equation \eqref{eq:IntroHestonWeakMixedProblemHomogeneous} with the observation that in the estimate \eqref{eq:MainSupremumEstimates3} we need to replace $u^M$ with $u^m$ where $u^m$ is defined as follows. Let
\begin{align*}
m :=\essinf_{\Gamma_1\cap B_{R_0}(z_0)} g,
\end{align*}
and set
$$
u^m(z):=(u(z) \wedge m)^-\quad\hbox{for a.e. } z \in B_{R_0}(z_0).
$$
\end{rmk}

\begin{rmk}[Inhomogeneous Dirichlet boundary conditions and variational equations]
\label{rmk:Inhomogeneous_Dirichlet_variational_equation}
Given a (non-zero) boundary-data function $g\in H^1(\sO,\fw)$ then, as an alternative to our proofs of Corollaries \ref{cor:MainSupremumEstimatesBoundary}, \ref{cor:MainHolderContinuityBoundary}, and \ref{cor:MainContinuity}, we could replace $u$ and $(f,v)_{L^2(\sO,\fw)}$ in \eqref{eq:IntroHestonWeakMixedProblemHomogeneous} by $\tilde u := u-g \in H^1_0(\sO\cup\Gamma_0,\fw)$ and the functional $F\in H^{-1}(\sO,\fw) := (H^1_0(\sO\cup\Gamma_0,\fw))'$, where
\begin{equation}
\label{eq:SourceFunctional}
F(v) := (f,v)_{L^2(\sO,\fw)} - \fa(g,v), \quad\forall\, v\in H^1_0(\sO\cup\Gamma_0,\fw),
\end{equation}
and instead of \eqref{eq:IntroHestonWeakMixedProblemHomogeneous}, consider the variational equation,
\begin{equation}
\label{eq:Variational_equation_reduction}
\fa(\tilde u, v) = F(v), \quad\forall\, v\in H^1_0(\sO\cup\Gamma_0,\fw).
\end{equation}
This reduction would bring our arguments into closer alignment with those of Gilbarg and Trudinger \cite[Chapter 8]{GilbargTrudinger}, but at the cost of a slightly more complicated proofs than those we employ in this article and little gain.
\end{rmk}

\subsubsection{H\"older continuity up to the boundary for solutions to the variational equation}
We recall the definition of the \emph{Koch distance function}, $d(\cdot,\cdot)$, on $\HH$ introduced by Koch in \cite[p.~11]{Koch},
\begin{equation}
\label{eq:IntroKochDistance}
\begin{aligned}
d(z,z_0) := \frac{|z-z_0|}{\sqrt{y + y_0 + |z-z_0|}}, \quad \forall\, z=(x,y),\ z_0=(x_0,y_0)\in \bar\HH,
\end{aligned}
\end{equation}
where $|z-z_0|^2 = (x-x_0)^2 + (y-y_0)^2$. The Koch distance function is equivalent to the \emph{cycloidal distance function} introduced by Daskalopoulos and Hamilton in \cite[p.~901]{DaskalHamilton1998} for the study of the porous medium equation.

Following \cite[\S 1.26]{Adams_1975}, for an open subset $U\subset\HH$, we let $C(U)$ denote the vector space of continuous functions on $U$ and let $C(\bar U)$ denote the Banach space of functions in $C(U)$ which are bounded and
uniformly continuous on $U$, and thus have unique bounded, continuous extensions to $\bar U$, with norm
$$
\|u\|_{C(\bar U)} := \sup_{U}|u|.
$$
Given $\alpha \in (0,1)$, we say that $u\in C^\alpha_s(\bar U)$ if $u\in C(\bar U)$ and
$$
\|u\|_{C^\alpha_s(\bar U)} < \infty,
$$
where
\begin{equation}
\label{eq:CalphasNorm}
\|u\|_{C^\alpha_s(\bar U)} := [u]_{C^\alpha_s(\bar U)} + \|u\|_{C(\bar U)},
\end{equation}
and
\begin{equation}
\label{eq:CalphasSeminorm}
[u]_{C^\alpha_s(\bar U)} := \sup_{\stackrel{z_1,z_2\in U}{z_1\neq z_2}}\frac{|u(z_1)-u(z_2)|}{d^\alpha(z_1,z_2)}.
\end{equation}
Moreover, $C^\alpha_s(\bar U)$ is a Banach space \cite[\S I.1]{DaskalHamilton1998} with respect to the norm \eqref{eq:CalphasNorm}. We say that $u\in C^\alpha_s(U)$ if $u\in C^\alpha_s(\bar V)$ for all precompact open subsets $V\Subset U\cup\Gamma_0$.

When $U$ may be \emph{unbounded}, we let $C_{\loc}(\bar U)$ denote the linear subspace of functions $u\in C(U)$ such that $u\in C(\bar V)$ for every precompact open subset $V\Subset \bar U$; similarly, we let $C^\alpha_{s,\loc}(\bar U)$ denote the linear subspace of functions $u \in C^\alpha_s(U)$ such that $u\in C^\alpha_s(\bar V)$ for every precompact open subset $V\Subset \bar U$.

%CP 1.11.2016: We don't use these spaces in the paper. We can remove their definition.
%PF 2-25-2016: OK
%For any non-negative integer $k$, we let $C^k_0(U\cup\Gamma_0)$ denote the linear subspace of functions $u\in C^k(U)$ such that $u\in C^k(\bar V)$ for every precompact open subset $V\Subset U\cup\Gamma_0$ and similarly define $C^{\infty}_0(U\cup\Gamma_0)$.

We have the following analogues of \cite[Theorem 8.27 and 8.29]{GilbargTrudinger} and \cite[Theorem 4.5.5 and 4.5.6]{Koch}, but again for the cases of a `degenerate-boundary interior' point, $z_0\in\Gamma_0$, and a `degenerate boundary corner point', $z_0\in\bar\Gamma_0\cap\bar\Gamma_0$; see Figures \ref{fig:domain_and_middlehalfballs} and \ref{fig:domain_and_righthalfballs}, respectively. Though Koch allows for points in the interior of $\Gamma_0$, there is no analogue in \cite{Koch} of our Theorem \ref{thm:MainContinuityBoundary}, which allows for corner points; as before, Gilbarg and Trudinger \cite{GilbargTrudinger} only allow for boundary points where the elliptic partial differential operator is strictly elliptic.

\begin{thm} [H\"older continuity near points in $\Gamma_0$ for solutions to the variational equation]
\label{thm:MainContinuityInterior}
Let $s > \max\{2n,n+\beta\}$ and let $R_0$ be a positive constant. Then there are positive constants, $R_1 = R_1(R_0) < R_0$, and $C=C(\Lambda,n,\nu_0,R_0,s)$, and
$\alpha = \alpha(\Lambda,n,\nu_0,R_0,s) \in (0,1)$ such that the following holds. Let $\sO\subseteqq\HH$ be an open subset. If $u \in H^1(\sO,\fw)$ satisfies the variational equation \eqref{eq:IntroHestonWeakMixedProblemHomogeneous} with source function $f \in L^2(\sO,\fw)$ and $z_0 \in \Gamma_0$ is such that $\EE_{R_0}(z_0) \subset \sO$, and $f$ obeys \eqref{eq:Lsfcondition_ball}, then $u \in C^\alpha_s(\bar E_{R_1}(z_0))$, and
\begin{equation}
\label{eq:MainContinuity3}
\|u\|_{C^\alpha_s(\bar E_{R_1}(z_0))} \leq C\left(\|f\|_{L^s(E_{R_0}(z_0),y^{\beta-1})} + \|u\|_{L^2(E_{R_0}(z_0), y^{\beta-1})}\right).
\end{equation}
\end{thm}

\begin{rmk}[H\"older continuity up to $\Gamma_0$ and Sobolev embeddings]
H\"older continuity of solutions does not follow by an embedding theorem for Sobolev weighted spaces, analogous to \cite[Corollary 7.11]{GilbargTrudinger}, not even for functions $u \in H^2(\sO, \fw)$. For example,
for any $\beta >2$, let $q \in (0,(\beta-2)/2)$ and
\[
u(x,y) = y^{-q}, \quad \forall\, (x,y) \in \sO.
\]
Then, $u\in H^2(\sO,\fw)$, but $u \notin C^{\alpha}_s(\sO)$, for any $\alpha \in [0,1]$, since, a fortiori, $u \notin C(\sO\cup\Gamma_0)$.
\end{rmk}

\begin{thm} [H\"older continuity near points in $\overline{\Gamma}_0\cap\overline{\Gamma}_1$ for solutions to the variational equation]
\label{thm:MainContinuityBoundary}
Let $K$ be a finite, right circular cone, let $s > \max\{2n,n+\beta\}$, and let $R_0$ be a positive constant. Then there are positive constants, $R_1 = R_1(K,\Lambda,n,\nu_0,R_0) < R_0$, and $C=C(K,\Lambda,n,\nu_0,R_0,s)$, and $\alpha = \alpha(K,\Lambda,n,\nu_0,R_0,s) \in (0,1)$, such that the following holds. Let $\sO\subsetneqq\HH$ be an open subset. If $u \in H^1(\sO,\fw)$ satisfies the variational equation \eqref{eq:IntroHestonWeakMixedProblemHomogeneous} with source function $f \in L^2(\sO,\fw)$ and $z_0 \in \overline{\Gamma}_0\cap\overline{\Gamma}_1$ is such that $f$ obeys \eqref{eq:Lsfcondition_ball}, and
$$
u = 0 \hbox{ on }\Gamma_1\cap \bar E_{R_0}(z_0) \quad\hbox{(in the sense of $H^1$)},
$$
and $\sO$ obeys an interior and exterior cone condition with cone $K$ at $z_0$ and a uniform exterior cone condition with cone $K$ along $\overline{\Gamma}_1 \cap \bar E_{R_0}(z_0)$, then $u \in C^\alpha_s(\bar E_{R_1}(z_0))$ and satisfies \eqref{eq:MainContinuity3}.
\end{thm}

\begin{rmk}[Comparison with analysis near the non-degenerate boundary]
The term $\sigma(\sqrt{RR_0})$, where $\sigma(R) := \hbox{osc}_{\partial\sO\cap \bar B_R(z_0)}u$, which appears in  \cite[Equation (8.72)]{GilbargTrudinger} in the statement of \cite[Theorem 8.27]{GilbargTrudinger} does not appear in the statement of our Theorem \ref{thm:MainContinuityBoundary}. The reason is that unlike in \cite[Equation (8.71)]{GilbargTrudinger}, the test functions defined in the proof of Theorem \ref{thm:MainContinuityBoundary} do not need to involve $\esssup_{\partial\sO\cap\bar B_R(z_0)}u$ or $\essinf_{\partial\sO\cap\bar B_R(z_0)}u$ since no boundary condition is imposed on $v$ along $\Gamma_0$, in contrast to the Dirichlet boundary condition assumed for $v$ in the proofs of \cite[Theorem 8.18 and 8.26]{GilbargTrudinger}.
\end{rmk}

%CP 1.11.2016: The reviewer had a comment about this remark. My thoughts are the following. %PF 2-27-2016: Uncommented your remark, as it seems appropriate
By constructing suitable weighted Sobolev spaces adapted both to the degeneracy of the operator and the geometry of the corners, we may be able to obtain improved regularity estimates in a neighborhood of the points in $\bar \Gamma_0\cap\bar \Gamma_1$, similar to the ideas used for the study of strictly elliptic operators on polygonal domains described by Grisvard \cite[\S 4.4.1]{Grisvard}. We believe that this problem requires careful consideration and is best considered in a separate article.

\begin{rmk}[Counter-examples to higher-order regularity near corners for solutions to elliptic boundary value problems]
It is worth recalling \cite[\S 7.5]{Krylov_LecturesHolder} that the unique solution $u\in C^2(\sO)\cap C(\bar\sO)$ to the Dirichlet problem, $\Delta u = 1$ on $\sO:=(0,\pi)\times(0,\pi)$ and $u=0$ on $\partial\sO$, belongs to $C^1(\bar\sO)$ but not $C^2(\bar\sO)$. (Following our customary sign convention, we denote $\Delta u = -\sum_{i=1}^n u_{x_ix_i}$.) This example illustrates that the question of regularity near corner points is delicate even for boundary value problems defined by strictly elliptic operators and thus can be expected to be even more so in the case of degenerate-elliptic operators.
\end{rmk}

We have the following analogue of \cite[Theorem 8.27]{GilbargTrudinger} which applies to a variational equation defined by a strictly elliptic operator on an open subset satisfying an exterior cone condition and an inhomogeneous Dirichlet boundary condition.

\begin{cor}[H\"older continuity near points in $\overline{\Gamma}_0\cap\overline{\Gamma}_1$ for variational solutions with inhomogeneous Dirichlet boundary condition]
\label{cor:MainHolderContinuityBoundary}
Let $K$ be a finite, right circular cone, let $s > \max\{2n,n+\beta\}$ and let $R_0$ be a positive constant. Assume  $g \in H^1(\sO, \fw) \cap C^{\gamma}_{s,\loc}(\bar\Gamma_1)$, where $\gamma \in (0,1]$. Then there are positive constants, $R_1 = R_1(K, \Lambda, n, \nu_0, R_0) < R_0$, and $C=C(K,\Lambda,n, \nu_0,R_0,s)$, and $\alpha = \alpha(\gamma, \Lambda,n,\nu_0,R_0,s) \in (0,1)$ such that the following holds. Let $z_0\in\bar\Gamma_0\cap\bar\Gamma_1$. Assume that $\sO$ obeys an interior and exterior cone condition with cone $K$ at $z_0$ and a uniform exterior cone condition with cone $K$ along $\overline{\Gamma}_1 \cap \bar E_{R_0}(z_0)$. If $u \in H^1(\sO,\fw)$ satisfies the variational equation \eqref{eq:IntroHestonWeakMixedProblemHomogeneous} and \eqref{eq:uBoundaryCondition}, and the source function $f \in L^2(\sO,\fw)$ obeys \eqref{eq:Lsfcondition_ball}, then $u \in C^\alpha_s(\bar E_{R_1}(z_0))$, and
\begin{equation}
\label{eq:MainContinuity4}
\|u\|_{C^\alpha_s(\bar E_{R_1}(z_0))} \leq C\left(\|f\|_{L^s(E_{R_0}(z_0),y^{\beta-1})} + \|u\|_{L^2(E_{R_0}(z_0), y^{\beta-1})}
+\|g\|_{C^{\gamma}_s(\bar\Gamma_1\cap\bar E_{R_0}(z_0))}\right).
\end{equation}
When $\gamma=0$, that is $g \in H^1(\sO, \fw) \cap C_{\loc}(\bar\Gamma_1)$, then $u \in C(\bar E_{R_1}(z_0))$ and $u$ satisfies
\begin{equation}
\|u\|_{C(E_{R_1}(z_0))}
\leq  C\left( \|f\|_{L^s(E_{R_0}(z_0),y^{\beta-1})} + \|u\|_{L^2(E_{R_0}(z_0),y^{\beta-1})} + \|g\|_{L^{\infty}(\bar \Gamma_1\cap\bar E_{R_0}(z_0))} \right).
\end{equation}
\end{cor}

For any $\delta>0$, we let
\begin{equation}
\label{eq:Domain_of_finite_height}
\sO_{\delta} : = \sO \cap \left(\RR\times(0,\delta)\right).
\end{equation}
We then have the

\begin{cor} [H\"older continuity up to $\bar\Gamma_0$ for solutions to the variational equation]
\label{cor:MainContinuity}
Let $K$ be a finite, right circular cone, let $s>\max\{2n, n+\beta\}$, $\delta>0$, and $\gamma \in [0,1)$. Then there are constants $C=C(\delta,K,\Lambda,n,\nu_0,s )>0$ and $\alpha_1=\alpha_1(\delta, \gamma, K, \Lambda,n,\nu_0,s)\in [0,1)$ such that the following hold. Assume that $\sO$ obeys a uniform interior and exterior cone condition with cone $K$ on $\bar\Gamma_0\cap\bar\Gamma_1$ and a uniform exterior cone condition with cone $K$ along $\Gamma_1\cap\partial\sO_\delta$. Let $f \in L^2(\sO,\fw)$, $g \in  H^1(\sO,\fw) \cap C^\gamma_s(\bar\Gamma_1\cap\bar\sO_{\delta})$, and $u \in H^1(\sO,\fw)$ obey
\eqref{eq:IntroHestonWeakMixedProblemHomogeneous} and \eqref{eq:uBoundaryCondition}, and assume that $f$ and $u$ satisfy
\begin{equation}
\label{eq:LsfL2ucondition_strip}
\sup_{z_0\in\Gamma_0}\|f\|_{ L^s(E_{\delta}(z_0),y^{\beta-1})} < \infty \quad\hbox{and}\quad
\sup_{z_0\in\Gamma_0}\|u\|_{ L^2(E_{\delta}(z_0),y^{\beta-1})} < \infty.
\end{equation}
Then $u \in C^{\alpha_1}_{s}(\bar\sO_{\delta/2})$ and satisfies
\begin{equation}
\label{eq:MainContinuity5}
\|u\|_{C^{\alpha_1}_{s}(\bar\sO_{\delta/2})} \leq C\left(\sup_{z_0\in\Gamma_0}\|f\|_{ L^s(E_{\delta}(z_0),y^{\beta-1})} + \sup_{z_0\in\Gamma_0}\|u\|_{ L^2(E_{\delta}(z_0),y^{\beta-1})} +
\|g\|_{C^{\gamma}_s(\bar\Gamma_1\cap\bar\sO_{\delta})} \right).
\end{equation}
When $\gamma \in (0,1)$, then $\alpha_1 \in (0,1)$, and when $\gamma=0$, then $\alpha_1=0$.
\end{cor}

Condition \eqref{eq:LsfL2ucondition_strip} on $u$ is satisfied when $u \in L^2(\sO,\fw)$ and the open subset, $\sO$, is bounded in the $x$-direction, as we can see from the definition \eqref{eq:HestonWeight} of the weight $\fw$.

\subsubsection{Strong maximum principle}
We also have the following analogue of \cite[Theorem 8.19]{GilbargTrudinger}. It is important to note that Theorem \ref{thm:StrongMaximumPrinciple} is an analogue of the classical strong maximum principle, except that points in the degenerate-boundary portion, $\Gamma_0$, play the same role as points in $\sO$. We now assume that $\sO\subseteqq\HH$ is domain, that is, a \emph{connected}, open subset.

\begin{thm} [Strong maximum principle]
\label{thm:StrongMaximumPrinciple}
Let $\sO \subseteqq \HH$ be a domain. Let $z_0\in\sO\cup\Gamma_0$, $R_0$ be a positive constant, and $u \in H^1(\sO,\fw)$ be a subsolution to equation \eqref{eq:IntroHestonWeakMixedProblemHomogeneous} with $f=0$. If the ball $E_{R_0}(z_0)$ as in \eqref{eq:Euclidean_balls_relative_to_a_subdomain} obeys $E_{R_0}(z_0) \Subset \sO\cup\Gamma_0$ and
$$
\esssup_{E_{R_0}(z_0)} u = \esssup_{\sO} u,
$$
then $u$ is constant on $\sO$.
\end{thm}

Note that $\esssup_{E_{R_0}(z_0)} u < \infty$ by Theorem \ref{thm:MainSupremumEstimatesInterior} when $z_0\in\Gamma_0$, while \cite[Theorem 8.17]{GilbargTrudinger} yields this local boundedness result when $E_{R_0}(z_0) \Subset \sO$.

\subsubsection{H\"older continuity up to the boundary for solutions to the variational inequality}
Given a source function $f\in L^2(\sO,\fw)$, an (inhomogeneous) Dirichlet boundary condition $g\in H^1(\sO,\fw)$ on $\Gamma_1$, and an obstacle function $\psi\in H^1(\sO,\fw)$ obeying \eqref{eq:ObstacleFunctionLessThanZero} in the sense that
\begin{equation}
\label{eq:ObstacleLessThanZeroBoundary}
(\psi-g)^+ \in H^1_0(\sO\cup \Gamma_0,\fw),
\end{equation}
we call $u\in H^1(\sO,\fw)$ a solution to the variational inequality for the
%CP 1.23.2016: Removed Heston operator and added the more general operator
%Heston operator
%PF 2-27-2016: OK
operator $A$ defined in \eqref{eq:Operator} with Dirichlet boundary condition along $\Gamma_1$ if
\begin{equation}
\label{eq:VIProblemHeston}
\begin{gathered}
u-g\in H^1_0(\sO\cup\Gamma_0,\fw),  \quad u\geq\psi \hbox{ a.e. on }\sO,
\\
\fa(u,v-u) \geq (f,v-u)_{L^2(\sO,\fw)}
\\
\quad \forall\, v\in H^1(\sO,\fw), \quad v-g\in H^1_0(\sO\cup\Gamma_0,\fw), \quad v\geq\psi \hbox{ a.e. on }\sO.
\end{gathered}
\end{equation}
Given additional mild conditions on $f$ and $\psi$, it is proved in \cite{Daskalopoulos_Feehan_statvarineqheston} that there is a unique solution, $u \in H^1(\sO,\fw)$, to \eqref{eq:VIProblemHeston},
%CP 1.23.2016: Added below
%PF 2-27-2016: OK
when $A$ is the Heston operator defined in \eqref{eq:OperatorHestonIntro}. For Theorem \ref{thm:MainHolderContinuityVariationalInequality}, we require

%CP 3.7.2016: Removed my proof of the Holder continuity for the variational inequality
%PF 3-7-2016: I moved this proof to the end of the subsubsection as it is still a nice idea

\begin{hyp}[Conditions on the source and obstacle functions]
\label{hyp:SourceObstacleFunctionBoundaryRegularity}
For some $\delta>0$,
%PF 2-26-2015: This is one but not the only option
%CP 2.25.2016: Added in hypotheses
%Assume that there is a positive constant, $c_0$, such that $c\geq c_0$ on $\sO$. For some $\delta \leq \delta_0$, where the positive constant $\delta_0$ is chosen as in \eqref{eq:Delta_0}, we assume that
%PF 2-27-2016: Added instead to theorem statement as one option
\begin{align}
\label{eq:HolderSourceFunctionUpToBdry}
f &\in L^2(\sO,\fw)\cap L^\infty(\sO_\delta),
\\
\label{eq:HolderObstacleFunctionUpToBdry}
\psi &\in H^2(\sO_\delta,\fw)\cap L^{\infty}(\sO_\delta),
\end{align}
where $\sO_\delta$ is defined in \eqref{eq:Domain_of_finite_height}.
\end{hyp}

We then have

\begin{thm}[H\"older continuity up to $\bar\Gamma_0$ for solutions to the variational inequality with homogeneous boundary condition]
\label{thm:MainHolderContinuityVariationalInequality}
Require that $\sO$ obeys a uniform interior and exterior cone condition on $\bar\Gamma_0\cap\bar\Gamma_1$ with cone $K$
and a uniform exterior cone condition with cone $K$ along $\Gamma_1\cap\partial\sO_\delta$,
%PF 2-27-2016: added
%PF 3-7-2016: Added
for some $\delta > 0$. Assume that $f$ obeys \eqref{eq:HolderSourceFunctionUpToBdry} and $g = 0$ and $\psi$ obeys \eqref{eq:ObstacleLessThanZeroBoundary} (with $g=0$) and
\eqref{eq:HolderObstacleFunctionUpToBdry}, and that
\begin{equation}
\label{eq:ApsifSupBound}
\esssup_{\sO_\delta}(A\psi-f)^+ < \infty.
\end{equation}
%PF 2-27-2016 Added pair of alternative hypotheses
%CP 3.7.2016: Left only the assumption that c >= 0 and added that \gamma, \mu >0.
%Require in addition that $\sO_\delta$ is bounded and $c \geq 0$ a.e. on $\sO_\delta$ \emph{or} $c \geq c_0$ a.e. on $\sO_\delta$ for a positive constant $c_0$ and $\delta \leq \delta_0$ with $\delta_0$ as in \eqref{eq:Delta_0}.
%PF 3-7-2016: We need the following to apply the WMP to the VE and VI
If $\sO_\delta$ is bounded, require that $c \geq 0$ a.e. on $\sO_\delta$; if $\sO_\delta$ is unbounded, require in addition that $c \geq c_0 > 0$ a.e. on $\sO_\delta$ for a positive constant $c_0$.
%PF 3-9-2016: Not needed by Camelia's method
%and that $\tau > 0$ in \eqref{eq:HestonWeight}.
If $u\in H^1_0(\sO\cup\Gamma_0,\fw)$ is a solution to \eqref{eq:VIProblemHeston} such that at least \emph{one} of the following conditions holds,
\begin{equation}
\label{eq:VISolutionW1infinity}
\height(\sO) < \infty \quad\hbox{or}\quad u\in W^{1,\infty}(\sO_\delta\less\sO_{\delta/2}),
\end{equation}
%PF 2-27-2016: delta introduced earlier
%where $\delta$ is as in Hypothesis \ref{hyp:SourceObstacleFunctionBoundaryRegularity}.
then
$$
u\in C_s^{\alpha_1}(\bar\sO_{\delta/2}),
$$
where $\alpha_1=\alpha_1(\delta,K, \Lambda,n,\nu_0,s)\in (0,1)$.
\end{thm}

\begin{cor}[H\"older continuity up to $\bar\Gamma_0$ for solutions to the variational inequality with inhomogeneous Dirichlet boundary condition]
\label{cor:MainHolderContinuityVariationalInequality}
Assume the hypotheses of Theorem \ref{thm:MainHolderContinuityVariationalInequality} and $g \in H^2(\sO,\fw)\cap C^{\gamma}_s(\bar\Gamma_1\cap\partial \sO_{\delta/2})$, with $\gamma\in (0,1]$. Let $u\in H^1(\sO,\fw)$ be a solution to \eqref{eq:VIProblemHeston}
such that
\begin{equation*}
\height(\sO) < \infty \quad\hbox{or}\quad u-g\in W^{1,\infty}(\sO_\delta\less\sO_{\delta/2}).
\end{equation*}
Then $u\in C_s^{\alpha_2}(\bar\sO_{\delta/2})$, where $\alpha_2=\alpha_1\wedge\gamma$ and the constant $\alpha_1$ is as in the conclusion of Theorem \ref{thm:MainHolderContinuityVariationalInequality}. If $g \in H^2(\sO,\fw)\cap C(\bar\Gamma_1\cap\partial \sO_{\delta/2})$, then $u\in C(\bar\sO_{\delta/2})$.
\end{cor}

\begin{rmk}[Hypotheses on the solution to the variational inequality]
\label{rmk:VISolutionW1infinity}
The second condition in \eqref{eq:VISolutionW1infinity} in Theorem \ref{thm:MainHolderContinuityVariationalInequality} is implied by the $W^{2,p}_{\loc}(\sO)$ regularity result \cite[Theorem 6.18]{Daskalopoulos_Feehan_statvarineqheston}
for $p>2$ and corresponding $W^{2,p}(U)$ \apriori estimates using the conditions \eqref{eq:HolderSourceFunctionUpToBdry} and \eqref{eq:HolderObstacleFunctionUpToBdry}, and the Sobolev embedding $W^{2,p}(U) \hookrightarrow
C_b^1(U)$ for open subsets $U\Subset\HH$ with the interior cone property \cite[Theorem 5.4 (C)]{Adams_1975}.
\end{rmk}

\begin{rmk}[Inhomogeneous Dirichlet boundary conditions and variational inequalities]
\label{rmk:Inhomogeneous_Dirichlet_variational_inequality}
Given a (non-zero) boundary-data function $g\in H^1(\sO,\fw)$ then, as an alternative to our proof of Corollary \ref{cor:MainHolderContinuityVariationalInequality}, we could replace $u, v, \psi\in H^1(\sO,\fw)$ and $(f,v-u)_{L^2(\sO,\fw)}$ in \eqref{eq:VIProblemHeston} by $\tilde u := u-g$, $\tilde v := v-g$, $\tilde\psi := \psi-g$ in $H^1_0(\sO\cup\Gamma_0,\fw)$ and the functional $F\in H^{-1}(\sO,\fw)$ in \eqref{eq:SourceFunctional} and, instead of \eqref{eq:VIProblemHeston}, consider the variational inequality,
\begin{equation}
\label{eq:Variational_inequality_reduction}
\fa(\tilde u, w-\tilde u) \geq F(w-\tilde u), \quad\forall\, w\in H^1_0(\sO\cup\Gamma_0,\fw), \quad w \geq \tilde\psi \quad\hbox{a.e. on }\sO.
\end{equation}
This reduction would bring our arguments into closer alignment with those of Gilbarg and Trudinger \cite[Chapter 8]{GilbargTrudinger} and Troianiello \cite[Chapter 4]{Troianiello}, but at the cost of a slightly more complicated proofs than those we employ in this article and little gain.
\end{rmk}

%CP 2.25.2016: Added paragraph below
%PF 2-27-2016: Slight rephrasing of paragraph below
%PF 3-7-2016: Coercivity proof moved here
%PF 3-9-2016: Removed, since we do use this in remark
%Although we shall not use this observation in this article,
It is interesting to note that the bilinear form $\fa$ given by \eqref{eq:Operator_A_bilinear_form} is coercive if the height of the domain $\sO$ is sufficiently small. Indeed, from the expression \eqref{eq:Operator_A_bilinear_form} for the bilinear form, we can write $\fa(u,u)$ as a sum of four terms $I_1+I_2+I_3+I_4$. If $\sO\subseteqq \RR^{n-1}\times (0,\delta)$, for a constant $\delta\in (0, 1]$, the expression \eqref{eq:Operator_A_bilinear_form} yields a positive constant, $C=C(\Lambda, n)$, such that
$$
|I_2|+|I_3| \leq \sqrt{\delta} C \|u\|^2_{H^1(\sO,\fw)}.
$$
If in addition there is a positive constant, $c_0$, such that $c\geq c_0$ on $\sO$, the preceding inequality and the strict ellipticity condition \eqref{eq:Operator_A_ellipticity} gives the inequality,
$$
\fa(u,u) \geq (\nu_0-\sqrt{\delta} C) \|\sqrt{y}Du\|^2_{L^2(\sO,\fw)} + (c_0-\delta C) \|u\|_{L^2(\sO,\fw)},\quad\forall\, u\in H^1(\sO,\fw).
$$
Hence, there are positive constants
\begin{equation}
\label{eq:Delta_0}
\delta_0=\delta_0(c_0,\Lambda,n,\nu_0)
\quad\hbox{ and } \quad
C_0=C_0(c_0,\Lambda,n,\nu_0),
\end{equation}
such that
\begin{equation}
\label{eq:Coercive_A_delta}
\fa(u,u) \geq C_0 \|u\|_{H^1(\sO,\fw)},\quad\forall\, u\in H^1(\sO,\fw),
\end{equation}
for all subdomains $\sO\subseteqq \RR^{n-1}\times(0,\delta_0)$. Therefore, when $\sO\subseteqq \RR^{n-1}\times(0,\delta_0)$ and $c \geq c_0 > 0$ a.e. on $\sO$, the bilinear form $\fa: H^1(\sO,\fw) \times H^1(\sO,\fw) \to \RR$ is \emph{coercive}.
%PF 3-9-2016: Added
We use this observation in Remark \ref{rmk:ClaimEssBoundedPenaltyTerm_alternative_proof}.

\subsubsection{Harnack inequality for non-negative solutions to the variational equation}
We also have the following analogue of \cite[Theorem 8.20 and Corollary 8.21]{GilbargTrudinger} and \cite[Theorem 4.5.3]{Koch}; it is important to note that Theorem \ref{thm:MainHarnack} is a direct analogue of the classical \emph{interior} Harnack inequality --- with points in the degenerate-boundary portion, $\Gamma_0$, playing the same role as points in $\sO$ --- and \emph{not} a `boundary Harnack inequality' (compare, for example, \cite[Theorem 1.1]{Caffarelli_Fabes_Mortola_Salsa_1981}).

\begin{thm} [Harnack inequality near $\Gamma_0$]
\label{thm:MainHarnack}
Let $\sO'\subset \sO \subseteqq \HH$ be open subsets such that $\sO'\Subset \sO\cup\Gamma_0$. Then there is a positive constant $C$, depending at most on $\diam(\sO')$, $\dist(\partial\sO\cap\HH, \partial\sO'\cap\HH)$, $\Lambda$, $\nu_0$ and $n$, such that for any non-negative $u \in H^1(\sO,\fw)$ obeying \eqref{eq:IntroHestonWeakMixedProblemHomogeneous} with $f=0$ on $\sO$, we have
\begin{equation}
\label{eq:MainHarnack1}
\esssup_{\sO'}u \leq C\essinf_{\sO'}u.
\end{equation}
\end{thm}

\begin{rmk}[Applications to the proof of optimal regularity for variational solutions to the obstacle problem]
Continuity up the `degenerate boundary' (Theorem \ref{thm:MainHolderContinuityVariationalInequality}) and the Harnack inequality (Theorem \ref{thm:MainHarnack}) are among the results of this article which Daskalopoulos and Feehan apply in \cite{Daskalopoulos_Feehan_optimalregstatheston} to prove that a solution $u\in H^1(\sO,\fw)$ to \eqref{eq:IntroObstacleProblem} actually belongs to $C^{1,1}_{s,\loc}(\sO\cup\Gamma_0)$.
\end{rmk}
%COMMENT At some point, we would like to get rid of the ``loc'' where it may be understood as a convention in Holder spaces such as the above

\subsection{Connections with previous research}
\label{subsec:Survey}
As noted in \S \ref{subsec:Overview}, there is a long history of research on local $L^\infty$ and $C^\alpha$ estimates and H\"older regularity and Harnack inequalities for weak solutions to degenerate-elliptic equations, so a reader may reasonably ask what is new in this article. Because our article builds most directly on work of Koch, we begin with a comparison of our methods and results with those in \cite{Koch}. We then contrast our work with that of S. Chanillo and R. L. Wheeden \cite{Chanillo_Wheeden_1986}, E. B. Fabes, C. E. Kenig and R. P. Serapioni \cite{Fabes_Kenig_Serapioni_1982a}, J. J. Kohn and L. Nirenberg \cite{Kohn_Nirenberg_1967}, and M. K. V. Murthy and G. Stampacchia \cite{Murthy_Stampacchia_1968}, as well as a selection of later articles which further develop their ideas.

The arguments in our article are not straightforward adaptations of the proofs of the analogous classical results described by Gilbarg and Trudinger \cite[Theorems 8.15, 8.20, 8.22 and 8.27]{GilbargTrudinger}, due to the fact that our Sobolev spaces are weighted, so the standard Sobolev, Poincar\'e, and John-Nirenberg inequalities do not apply. We rely on
%CP 1.23.2016: removed extra 'the'
%PF 2-27-2016: OK
the Moser iteration technique and the most difficult step in making this technique work involves the selection of a suitable John-Nirenberg inequality. For this purpose, we use the so-called \emph{abstract John-Nirenberg inequality}, due to Bombieri and Giusti \cite[Theorem 4]{Bombieri_Giusti_1972}, which can be applied to any topological space endowed with a regular Borel measure satisfying some natural requirements. In order to verify the hypotheses of the abstract John-Nirenberg inequality in our weighted Sobolev space setting (Proposition \ref{prop:ApplicationAbstractJohnNirenberg}), we prove a local version of the Poincar\'e inequality, Corollary \ref{cor:PoincareInequality}, suitable for our weighted Sobolev spaces.

\subsubsection{Connections with work of Koch}
In \cite{Koch}, Koch considers weak solutions to a certain linear parabolic partial differential equation in divergence form and which arises in the study of the porous medium equation. He takes the spatial domain to be the whole upper half space, $\HH=\RR^{n-1}\times\RR_+$, assumes a degeneracy similar to that in the
%CP 1.23.2016: Adapted to the generalization for variable coefficients
%PF 2-27-2016: OK
operator \eqref{eq:Operator}
%CP 1.23.2016: I removed this parenthesis ...
%(namely, $\vartheta(t,x)=x_n$),
%PF 2-27-2016: OK
and obtains a local $L^\infty$ bound \cite[Proposition 4.5.1]{Koch}, a Harnack inequality \cite[Theorem 4.5.3]{Koch}, and a $C^\alpha$ estimate and H\"older continuity \cite[Theorem 4.5.5]{Koch} up to the degenerate boundary
%CP 1.23.2016: Changed x_n by y
%($x_n=0$)
%PF 2-27-2016: OK
($y=0$) for weak solutions.
%CP 1.23.2016: Adapted to the generalization for variable coefficients
%PF 2-27-2016: OK
% Koch uses Sobolev weights which are comparable to ours ($\fw \simeq x_n^{\beta-1}$),
Koch uses the same Sobolev weights as ours, but whereas he uses potential theory and pointwise estimates for fundamental solutions to prove the Harnack inequality and H\"older continuity, our method of proof is based on Moser iteration and avoids any need for potential theory or pointwise estimates of fundamental solutions.
%PF 2-25-2016 rephrased because now we do have variable coefficient operators
%CP 2.25.2016: Ok
%We believe that this is an important distinction because we can therefore expect the methods and results of our article to extend to the broader class of degenerate elliptic (and parabolic) operators discussed in \S \ref{subsec:Class_noncoercive_bilinear_maps} --- this would be difficult to achieve using potential theory.
We believe that this is an important distinction: in this article we establish results, for a broader class of degenerate elliptic operators, that would be difficult to achieve using potential theory.

While Koch takes the spatial domain to be the whole upper half-space, $\sO=\HH$, we consider the
%CP 1.23.2016: Added below for the more general operator
%Heston
%PF 2-27-2016: OK
variational equation \eqref{eq:IntroHestonWeakMixedProblemHomogeneous}
on subdomains of the half-space, $\sO\subsetneqq\HH$, with Dirichlet boundary condition along the non-degenerate boundary, $\Gamma_1$. In \cite{Koch}, Koch does not need to analyze the regularity of solutions at the `corner points' ($\bar\Gamma_0\cap\bar\Gamma_1$), but in our article we establish local supremum bounds for weak subsolutions and $C^\alpha$ estimates and H\"older continuity up to $\bar\Gamma_0$ for weak solutions on neighborhoods of points in $\bar\Gamma_0\cap\bar\Gamma_1$ (see our Theorems \ref{thm:MainSupremumEstimatesBoundary} and \ref{thm:MainContinuityBoundary}, and Corollaries \ref{cor:MainSupremumEstimatesBoundary} and \ref{cor:MainHolderContinuityBoundary}) --- results which appear difficult to obtain using pointwise estimates of the fundamental solution.

In \cite{Koch}, Koch uses Moser iteration but only to obtain the local $L^\infty$ bound for a weak solution \cite[Proposition 4.5.1]{Koch}. In order to prove H\"older regularity of solutions along the boundary $\bar\Gamma_0$, we need the version of the Poincar\'e inequality for weighted Sobolev spaces that we prove in Corollary \ref{cor:PoincareInequality}. Koch also obtains a version of the Poincar\'e inequality for weighted Sobolev spaces \cite[Lemma 4.4.4]{Koch} that applies to functions defined on the whole half-space. The H\"older regularity results we establish in this article are local and they are most easily proved using a local version of the Poincar\'e inequality, such as our Corollary \ref{cor:PoincareInequality}.
%CP 1.23.2016: Replaced 'Our' by 'The'
%Our
%PF 2-27-2016: OK
The proof of our Poincar\'e inequality --- relying only on integration by parts and the Poincar\'e inequality for standard Sobolev spaces --- appears simpler to us than the proof of \cite[Lemma 4.4.4]{Koch}.

Our local version of the Poincar\'e inequality (Corollary \ref{cor:PoincareInequality}) allows us to appeal to the `abstract John-Nirenberg inequality' \cite[Theorem 4]{Bombieri_Giusti_1972} and employ Moser iteration to obtain, as we noted above, H\"older regularity for weak solutions up to the `corner points'  ($\bar\Gamma_0\cap\Gamma_1$) and a Harnack inequality (on neighborhoods of points in $\Gamma_0$) for non-negative weak solutions without relying on pointwise estimates of fundamental solutions. In particular, Koch does not use a John-Nirenberg inequality for weighted Sobolev spaces to obtain the results we cited in \cite{Koch}.

Finally, Koch does not consider applications to H\"older regularity of solutions to variational inequalities as we do in our article.

%CP 1.11.2016: I think that this survey of the literature is quite extensive. It includes also the work of Koch and Lierl-Saloff-Coste which intersect with ours and explains the differences.
%I checked also the literature on Dirichlet forms. This is mostly done for symmetric Dirichlet forms (and so it does not include the Heston operator). For non-symmetric Dirichlet forms, the work of Lierl-Saloff-Coste is closest to what we do and we explain it in this survey.
%I checked also the references to Daskalopoulos-Hamilton. The paper of Song-Wang would look related to ours, but I think that it is best that we do not mention it. The paper of Fornaro_et_al_2015 studies a similar operator but they discuss only the semigroup properties and derive no supremum estimates, Holder estimates, or Harnack inequality. I don't think that we need to mention it because is not directly related to what we do.
%PF 2-25-2016: OK

\subsubsection{Connections with other closely related work}
%COMMENT: Work of Kohn and Nirenberg
%PF 5.28.2013: Slightly edited
Kohn and Nirenberg prove an \apriori estimate, existence, and uniqueness of a solution in a certain weighted Sobolev space \cite[Equation (1.6)]{Kohn_Nirenberg_1967} to a variational equation defined by a boundary-degenerate, linear, second-order
%PF 5.28.2013: Their main result only concerns elliptic operators
elliptic operator \cite[Theorem 1]{Kohn_Nirenberg_1967}. They assume that the domain boundary is smooth, while we allow the domain to have singularities (at points in $\bar\Gamma_0\cap\bar\Gamma_1$). Rather than exploit the regularity of the solution implied by a suitable choice of weighted Sobolev space, they use the sign of the Fichera
%PF 2-29-2016 Updated footnote notation
function\footnote{Namely, $(b^\mu-a^{\mu\nu}_{z_\nu})\eta_\mu$, where $(\eta_1,\ldots,\eta_n)$ is the inward-pointing unit normal vector field along $\partial\sO$ \cite[Equation (1.1.3)]{Radkevich_2009a}.}
to determine when to impose Dirichlet boundary condition on portions of the domain boundary. In the case of Heston operator $A$ in \eqref{eq:OperatorHestonIntro}, this implies a dichotomy, $0<\beta<1$ and $\beta\geq 1$, when applying a Dirichlet boundary condition along $\Gamma_0$, whereas our choice of rather different weighted Sobolev spaces removes this undesirable dichotomy entirely and we never need to prescribe a Dirichlet boundary condition along $\Gamma_0$; see
%PF 2-18-2016 Reference to the older preprint version is correct here since Appendices omitted in published version
%CP 2.24.2016: Ok
Appendix B in the earlier preprint version \cite{Feehan_maximumprinciple_v5} of \cite{Feehan_maximumprinciple} for a detailed discussion. When $0<\beta<1$ (recall that $\beta=2\kappa\theta/\sigma^2$ from \eqref{eq:DefnBetaMu}), Kohn and Nirenberg would require a homogeneous Dirichlet condition along the full boundary, $\partial\sO$, in their main \cite[Theorem 1]{Kohn_Nirenberg_1967}: while this is in accordance with the Fichera sign condition \cite[pp. 798--801]{Kohn_Nirenberg_1967}, a boundary condition along $\Gamma_0$
%CP 1.23.2016: I removed the part below because such problems can have physical motivation when we study the hitting time of the boundary \Gamma_0
%lacks any physical motivation and
%PF 2-27-2016: OK
limits the regularity of the solution, $u$, to being at most continuous up to $\Gamma_0$.

%PF 5.28.2013: Rephrased per your comments added below; to be fair, these are global, not local contraints and might not alone preclude local regularity
Even when $\beta\geq 1$, their additional technical conditions \cite[(a)--(d), pp. 799--800]{Kohn_Nirenberg_1967} mean that their main result does not apply to the problem we consider in this article. For example, they use the Fichera condition to partition the boundary as $\partial\sO = \Sigma_1 \cup \Sigma_2 \cup \Sigma_3$ and, when $\beta\geq 1$, $\Sigma_1=\bar\Gamma_0$, $\Sigma_2 =  \emptyset$, and $\Sigma_3 = \Gamma_1$. They require that $\Sigma_2\cup\Sigma_3$ be relatively closed, which means that $\Gamma_1$ should be relatively closed, which is not true in our problem. Moreover, the closures of the portions of the boundary with a Dirichlet condition, $\Sigma_2\cup\Sigma_3$, and of the portion without any boundary condition, $\Sigma_1$, are disjoint, while in our problem, they are allowed to intersect.

While \cite[Theorem 1]{Kohn_Nirenberg_1967} provides a global \apriori estimate (see \cite[Inequality (1.7)]{Kohn_Nirenberg_1967}), along with existence and uniqueness of a solution in a certain weighted Sobolev space, it bounds the weighted Sobolev norm \cite[Equation (1.6)]{Kohn_Nirenberg_1967} of $u \in W^{2,k}_{\loc}(\sO)$ in terms of the same weighted Sobolev norm of $f \in W^{2,k}_{\loc}(\sO)$, for any $k\geq 1$, and this regularity requirement on $f$ is unusually strong. While we might try to extract global regularity for $u$ (in terms of H\"older norms) up to $\partial\sO$, that would require a suitable embedding theorem for weighted Sobolev spaces and, as far as we can tell (see, for example, \cite {Kufner}), such an embedding theorem is not available for the weighted Sobolev space defined in \cite[Equation (1.6)]{Kohn_Nirenberg_1967}. Simple localization procedures, using cutoff functions, usually require appropriate interpolation inequalities and these are not developed in \cite{Kohn_Nirenberg_1967} and may not be straightforward. On the other hand, more advanced methods of developing local supremum or H\"older estimates usually require Sobolev, Poincar\'e, and John-Nirenberg inequalities and these are not developed in \cite{Kohn_Nirenberg_1967} and, again, may not be straightforward for the choices of weights selected in \cite{Kohn_Nirenberg_1967}. Indeed, the Sobolev weights appearing in \cite[Theorem 1]{Kohn_Nirenberg_1967} appear to have a technical motivation, while the weights used in our article are directly motivated by the discussion in
%\S \ref{subsec:Class_noncoercive_bilinear_maps}.
\cite[Section 8]{Feehan_maximumprinciple}.

%PF 5.28.2013 Commented out the below. While the below is true, it doesn't explain why our choice of weights works whereas the above discussion does. The exponential factors in our weights are useful because they make \HH finite volume, but they are otherwise irrelevant.
%directly related to quantities with probabilistic interpretation \cite{McKean_1956} as the scale function, $\fs(y)=y^{-\beta}e^{\mu y}$, and speed measure, $\fm(y)=2/\sigma^2 y^{\beta-1}e^{-\mu y}$, for the Feller square-root process \cite{Feller1951}.

%COMMENT: Work of Murthy and Stampacchia
Murthy and Stampacchia \cite{Murthy_Stampacchia_1968, Murthy_Stampacchia_1968corr} establish local supremum estimates, H\"older regularity, and global $L^p$ estimates for solutions in weighted Sobolev spaces to a variational equation defined by a boundary-degenerate, linear, second-order elliptic operator. They assume that the (Lipschitz) coefficients
%PF 2-29-2016 Updated notation
$\bar a^{\mu\nu}$ in \eqref{eq:A_nondivergence_form} obey
$$
\langle \bar a(z)\xi, \xi\rangle \geq m(z) |\xi|^2, \quad\forall\, \xi\in\RR^n \hbox{ and a.e. } z\in \sO,
$$
where the weight, $m\geq 0$ a.e. on a bounded domain $\sO$, is required to obey \cite[p. 1]{Murthy_Stampacchia_1968}
$$
m \in L^s(\sO) \quad\hbox{and}\quad m^{-1}\in L^t(\sO),
$$
for some $s, t \geq 1$ such that $1/s+1/t<2/n$.
%CP 1.23.2016: Adapted to variable coefficients
%The Heston operator, $A$ in \eqref{eq:OperatorHestonIntro},
%PF 2-27-2016: OK
The operator $A$ defined in \eqref{eq:Operator} does not satisfy the Murthy-Stampacchia condition since we would need to choose $m(x,y)=\nu_0y$ and clearly $m^{-1} \notin L^t(\sO)$ for any $t\geq 1$ whenever $\Gamma_0$ is non-empty (as we allow throughout our article).

%COMMENT: Work of Fabes, Kenig and Serapioni
Fabes, Kenig and Serapioni \cite{Fabes_Kenig_Serapioni_1982a} consider operators of the form $Au = (a^{\mu\nu}u_{x_\mu})_{x_\nu}$, and Lipschitz coefficients $a^{\mu\nu}$ obeying \cite[p. 78]{Fabes_Kenig_Serapioni_1982a}
%PF 2-29-2016 Updated notation
$$
C^{-1}w(z) |\xi|^2 \leq \langle \bar a(z)\xi, \xi\rangle \leq C w(z) |\xi|^2, \quad\forall\, \xi\in\RR^n \hbox{ and a.e. } z\in \sO,
$$
where $C$ is a positive constant, and $w$ is a weight that belongs to the Muckenhoupt class, $A_2$. They use Moser iteration to establish local supremum estimates and H\"older continuity for solutions and a Harnack inequality for non-negative solutions, $u \in H^1_0(\sO,w)$, to the variational equation \cite[p. 94]{Fabes_Kenig_Serapioni_1982a}
%PF 2-29-2016 Updated notation
$$
\int_\sO \bar a^{\mu\nu}u_{z_\mu}v_{z_\nu}\,dz = \int_\sO fv\,dz, \quad \forall\, v \in C^\infty_0(\sO),
$$
given $f \in L^2(\sO)$ (by \cite[p. 81]{Fabes_Kenig_Serapioni_1982a}) and where they define \cite[p. 91]{Fabes_Kenig_Serapioni_1982a} (note the contrast with our definition
%PF 2-25-2016: Changed reference
%CP 2.25.2016: Ok
\eqref{eq:H1WeightedSobolevSpace} of $H^1(\sO,\fw)$)
$$
%CP 1.23.2016: Moved the weight w after the parenthesis
%PF 2-25-2016: OK
%\|u\|_{H^1(\sO,w)} := \left(\int_\sO \left(|Du|^2 + wu^2\right)\,dx\right)^{1/2},
\|u\|_{H^1(\sO,w)} := \left(\int_\sO \left(|Du|^2 + u^2\right)w\,dx\right)^{1/2},
$$
and $H^1(\sO,w)$ is the completion of $C^\infty_0(\sO)$ in $H^1(\sO,w)$. The Poincar\'e inequality holds in the case of $A_2$ weights \cite[p. 95, Item (4)]{Fabes_Kenig_Serapioni_1982a} and the Sobolev inequality holds in the case of $A_p$ weights \cite[Theorems 1.2, 1.3, 1.5 and 1.6]{Fabes_Kenig_Serapioni_1982a}. A calculation
%COMMENT: When we integrate this weight to verify the A_p condition, the integral is divergent
shows that our choice of weight,
%CP 1.23.2016: Removed the exponential term from the weight
%$\fw(x,y)=y^{\beta-1}e^{-\mu y -\gamma |x|}$
%PF 2-27-2016: Need some comment as to whether results proved with exponential factor included carry over to when not and vice versa
%CP 3.3.2016: The exponential factor changes nothing about whether the weight with or without exponential factor belongs to a A_p class
%PF 3-7-2016: Added back exponential factor
$\fw(x,y)=y^{\beta-1}e^{\tau|x|-\mu y}$ in \eqref{eq:HestonWeight} --- or any of its variants which keep the important factor $y^{\beta-1}$ --- is not contained in the $A_p$ class when $\beta\geq p$, and therefore the crucial Sobolev and Poincar\'e inequalities established in \cite{Fabes_Kenig_Serapioni_1982a} do not apply. Even if we restrict to the case $\beta<2$, Fabes, Kenig and Serapioni only obtain results for solutions obeying a homogeneous Dirichlet boundary condition along the full boundary, $\partial\sO$, whereas the essential feature of our article is that we impose no boundary condition along $\Gamma_0$. Finally, the absence of the lower-order terms in \eqref{eq:A_nondivergence_form} considerably simplifies the problem since, in a degenerate-elliptic operator, the term
%PF 2-29-2016 Updated notation
$b^\mu u_{z_\mu}$ may be as significant as
%PF 2-29-2016 Updated notation
$a^{\mu\nu}u_{z_\mu z_\nu}$.
%CP 1.23.2016: Added the following previous works of Serapioni and his co-authors
%PF 2-27-2016: OK

The method of Moser iteration has also been extended to degenerate operators in divergence form in articles such as \cite{Chiarenza_Serapioni_1984a, Chiarenza_Serapioni_1984b, Chiarenza_Serapioni_1985, Chiarenza_Rustichini_Serapioni_1989, Franchi_Serapioni_1987}, where the properties of $A_2$ and $A_{n/2+1}$ weights are used to derive the Harnack inequality and H\"older regularity properties of solutions. We remark that the weight
%PF 2-29-2016 Updated notation
$y^{\beta-1}$ considered in our article does not belong to these classes of functions, when $\beta \geq 2$ and $\beta\geq n/2$, respectively. Moreover such a restriction would not be natural in the present context.

%COMMENT: Work of Chanillo and Wheeden, Mohammed, Pingen, etc
Chanillo and Wheeden \cite{Chanillo_Wheeden_1986} prove a Harnack inequality, extending that of \cite{Fabes_Kenig_Serapioni_1982a} by allowing unequal weights,
$$
%PF 2-29-2016 Updated notation
w(z) |\xi|^2 \leq \langle \bar a(z)\xi, \xi\rangle \leq v(z) |\xi|^2, \quad\forall\, \xi\in\RR^n \hbox{ and a.e. } z\in \sO.
$$
While they also relax the condition that $w\in A_2$, they require that $v,w$ obey a doubling condition\footnote{This is also true for our weight, $\fw$, by Lemma \ref{lem:MeasureBalls}.}
and Poincar\'e and Sobolev inequalities \cite[\S 1.2]{Chanillo_Wheeden_1986}. However, their Harnack inequality has the traditional, interior form  (compare
\cite[Theorem 8.21]{GilbargTrudinger} for the case of a strictly elliptic operator) for a subdomain $\sO'\Subset\sO$.
%CP 1.23.2016: I removed this line because I added above more about the work of Serapioni and his co-authors and this paper is also included
%Similar results are obtained by Franchi and Serapioni \cite{Franchi_Serapioni_1987}, assuming $A_p$ weights.
%PF 2-25-2016: OK
Mohammed \cite{Mohammed_2002} extends the work of Chanillo and Wheeden by allowing general, non-zero coefficients $b^\mu$ and $c$ for $A$ in \eqref{eq:A_nondivergence_form}. Pingen also extends the work of Chanillo and Wheeden, but rather by considering quasilinear elliptic system in pure divergence form and no lower-order terms. He obtains an interior Harnack inequality and interior H\"older continuity, under suitable conditions on the structure of the quasilinearity and doubling conditions on the weights $w$ and $z := v^2/w$. Di Fazio, Fanciullo, and Zamboni \cite{DiFazio_Fanciullo_Zamboni_2010, DiFazio_Zamboni_2006, Zamboni_2002} and Stredulinsky \cite{Stredulinsky} also obtain an interior Harnack inequality and interior H\"older continuity for quasi-linear degenerate elliptic equations in divergence form under related hypotheses.

%COMMENT: Work of Lierl-Saloff-Coste
%PF 2-18-2016 Any reason why we're not citing the most recent version of their paper? I know it is not published yet.
%CP 2.24.2016: Updated the references to the latest version
%PF 2-25-2016: OK
Lierl and Saloff-Coste use Moser iteration to establish a parabolic Harnack inequality for time-dependent, non-symmetric, local Dirichlet forms \cite[Theorem 3.14]{Lierl_Saloff-Coste_2012v6_arxiv}. Their hypotheses,
\cite[Assumptions 0,1,2 and 4]{Lierl_Saloff-Coste_2012v6_arxiv}, are satisfied by the bilinear form
%CP 1.23.2016: Changed reference to the bilinear form
%\eqref{eq:HestonWithKillingBilinearForm}
%PF 2-27-2016: OK
\eqref{eq:Operator_A_bilinear_form} defined by the
%CP 1.23.2016: Changed reference to the more general operator
%Heston
%PF 2-27-2016: OK
the operator $A$ in \eqref{eq:Operator} on domains of finite height, for example,
%PF 2-25-2016: This is in 2D still?
%CP 2.25.2016: Changed
$\sO\subseteqq\RR^{n-1}\times(0,y_0)$, where $y_0$ is a positive constant.
%CP 2.24.2016: Lierl and Saloff-Coste now prove a version of the Poincare inequality in Theorem 3.11 under their hypotheses (they actually reference it in a paper of Sturm). I updated this and rephrased the below.
%PF 2-25-2016: OK
%The Poincar\'e inequality is a crucial ingredient in the proof of the Harnack inequality and whereas we prove the required Poincar\'e inequality (Corollary \ref{eq:CorPoincareInequality}) --- adapted to our choice of Sobolev weighted spaces and cycloidal distance function --- used in our proof of the Harnack inequality (Theorem \ref{thm:MainHarnack}), the Poincar\'e inequality is assumed as a hypothesis \cite[Assumption 4]{Lierl_Saloff-Coste_2012v6_arxiv} in the proof due to Lierl and Saloff-Coste.
The Poincar\'e inequality is a crucial ingredient in the proof of the Harnack inequality, which we prove in Corollary \ref{eq:CorPoincareInequality} by elementary methods. Lierl and Saloff-Coste state in \cite[Theorem 3.11]{Lierl_Saloff-Coste_2012v6_arxiv} a different version of the Poincar\'e inequality that involves the distance to the boundary of the ball, which in turn is proved in \cite[Corollary 2.5]{Sturm_1996}. In our Poincar\'e inequality, Corollary \ref{eq:CorPoincareInequality}, we do not need to use the distance to the boundary of the ball.

Lierl and Saloff-Coste also prove H\"older continuity of solutions \cite[Corollary 3.17]{Lierl_Saloff-Coste_2012v6_arxiv} with zero source function. To prove the H\"older continuity of solutions with non-zero source function, $f$, we need the stronger \emph{weak Harnack inequality} (compare \cite[Theorem 8.18]{GilbargTrudinger} for the case of a strictly elliptic operator), which is embedded in our proof of Theorem \ref{thm:MainContinuityInterior} in estimate \eqref{eq:Eq9}. Since the weak Harnack inequality allows non-zero source functions (unlike the Harnack inequality), it enables us to establish H\"older continuity of solutions with non-zero source function. Because the Harnack inequality is an `interior estimate' (recall that $\Gamma_0$ essentially plays the same role as the interior of $\sO$ in our article), it cannot be used to obtain H\"older continuity of solutions to the variational equation at corner points ($\bar\Gamma_0\cap\bar\Gamma_1$), as we do in our Theorem \ref{thm:MainContinuityBoundary}.

For variational inequalities defined by degenerate elliptic or parabolic operators, there has been little previous research. Vitanza and Zamboni \cite{Vitanza_Zamboni_2005, Vitanza_Zamboni_2007} describe existence and uniqueness results for solutions in certain weighted Sobolev spaces, but do not consider boundary regularity of solutions or partial Dirichlet boundary conditions.

\subsection{Mathematical highlights and guide to the article}
\label{subsec:Guide}
For the convenience of the reader, we provide a brief outline of the article. We begin in \S \ref{sec:SobolevPoincare} by describing a Sobolev inequality due to H. Koch \cite{Koch} and prove a Poincar\'e inequality for our weighted Sobolev spaces. In \S \ref{sec:JohnNirenberg}, we recall the abstract
John-Nirenberg inequality (Theorem \ref{thm:AbstractJohnNirenberg}) due to E. Bombieri and E. Giusti \cite{Bombieri_Giusti_1972} and justify its application (via Proposition \ref{prop:ApplicationAbstractJohnNirenberg}) in the setting of our weighted Sobolev spaces. The supremum estimate near $\bar\Gamma_0$ for solutions to the variational equation \eqref{eq:IntroHestonWeakMixedProblemHomogeneous} (Theorems \ref{thm:MainSupremumEstimatesInterior} and \ref{thm:MainSupremumEstimatesBoundary}) is proved in \S \ref{sec:SupremumEstimates} by adapting the Moser iteration technique employed in the proof of \cite[Theorem 8.15]{GilbargTrudinger} to the setting of our degenerate elliptic operators and weighted Sobolev spaces. Section \ref{sec:HolderContinuityVariationalEquation} contains our proof of local H\"older continuity along $\bar\Gamma_0$ of solutions to the variational equation \eqref{eq:IntroHestonWeakMixedProblemHomogeneous} (Theorems \ref{thm:MainContinuityInterior} and \ref{thm:MainContinuityBoundary}). The essential difference between the proofs of Theorems \ref{thm:MainContinuityInterior} and \ref{thm:MainContinuityBoundary} and the proof of their classical analogue for variational solutions to non-degenerate elliptic equations \cite[Theorems 8.27 and 8.29]{GilbargTrudinger} consists in a modification of the methods of \cite[\S 8.6, \S 8.9, and \S 8.10]{GilbargTrudinger} when deriving our energy estimates \eqref{eq:EnergyEstimate}, where we adapt the application of the
John-Nirenberg inequality and Poincar\'e inequality to our framework of weighted Sobolev spaces. In this section we also prove the Strong Maximum Principle (Theorem \ref{thm:StrongMaximumPrinciple}). In \S
\ref{sec:HolderContinuityVariationalInequality}, we apply the penalization method and techniques of \cite{Daskalopoulos_Feehan_statvarineqheston}, together with Theorems \ref{thm:MainContinuityInterior} and \ref{thm:MainContinuityBoundary}, to prove local H\"older continuity along $\bar\Gamma_0$ of solutions to the variational inequality \eqref{eq:IntroObstacleProblem}
(Theorem \ref{thm:MainHolderContinuityVariationalInequality}). Finally, in \S \ref{sec:Harnack} we prove the Harnack inequality (Theorem \ref{thm:MainHarnack}) for solutions to the variational equation \eqref{eq:IntroHestonWeakMixedProblemHomogeneous}. Appendix \ref{sec:Auxiliary} contains the proofs of auxiliary results employed throughout the article whose proofs are sufficiently technical that they would have otherwise interrupted the logical flow of our article.

%COMMENT TODO Keep for the journal version; comment out for the arXiv version
A longer, unpublished version of this article appeared as \cite{Feehan_Pop_regularityweaksoln_v3} and additional details for some lengthy but routine calculations are available there.

\subsection{Notation and conventions} In the definition and naming of function spaces, including spaces of continuous functions, H\"older spaces, or Sobolev spaces, we follow Adams \cite{Adams_1975} and alert the reader to occasional differences in definitions between \cite{Adams_1975} and standard references such as Gilbarg and Trudinger \cite{GilbargTrudinger} or Krylov \cite{Krylov_LecturesHolder, Krylov_LecturesSobolev}. We denote $\RR_+ := (0,\infty)$, $\bar\RR_+ := [0,\infty)$,
%CP 1.23.2016: Added the variable dimension in the definition of \HH
%PF 2-25-2016: OK
$\HH := \RR^{n-1}\times\RR_+$, and $\bar\HH := \RR^{n-1}\times\bar\RR_+$, where $n\geq 2$. We let $\NN:=\left\{1,2,3,\ldots\right\}$ denote the set of positive integers. For $x,y\in\RR$, we denote $x\wedge y : =\min\{x,y\}$, $x \vee y := \max\{x,y\}$. Moreover, $x^+:=x\vee 0$ and $x^- := -(x\wedge 0)$, so $x=x^+-x^-$ and $|x| = x^+ + x^-$, a convention which differs from that of \cite[\S 7.4]{GilbargTrudinger}. If $V\subset S$ is an open subset of a subset $S\subset \RR^n$, we write $U\Subset S$ when $\bar U$ is compact and $\bar U\subset S$.

%CP 1.23.2016: This paragraph is no longer needed
%PF 2-25-2016: :-)
\begin{comment}
Throughout our article, we fix $n = 2$. We keep track of the dependency of many of our estimates on the dimension, $n$, of $\HH = \RR^{n-1}\times(0,\infty)$ in our analysis, even though $n=2$ in this article, as this will make it easier to extend our results to partial differential equations on open subsets of $\HH$ when $n\geq 2$ but which preserve the key features of \eqref{eq:IntroBoundaryValueProblem}.
\end{comment}

When we label a condition an \emph{Assumption}, then it is considered to be universal and in effect throughout this article and so not referenced explicitly in theorem and similar statements; when we label a condition a \emph{Hypothesis}, then it is only considered to be in effect when explicitly referenced.

%PF 2-23-2016 Uncommented. Anyone else to thank?
%CP 2.24.2016: Added
\subsection{Acknowledgments}
We would like to thank Panagiota Daskalopoulos for many useful discussions on degenerate partial differential equations and for proposing some of the questions considered in this article. In addition we want to thank Sagun Chanillo and Richard Wheeden for many helpful references concerning the method of Moser iteration.

\section{Sobolev and Poincar\'e inequalities for weighted Sobolev spaces}
\label{sec:SobolevPoincare}
The main result of this subsection is a Poincar\'e inequality (Lemma \ref{lem:PoincareInequality}) for weighted Sobolev spaces. In addition, we review a Sobolev inequality (Lemma \ref{lem:AnalogSobolev}) due to H. Koch \cite{Koch}. Recall from \cite[Corollary 4.3.4]{Koch} that the weight $y^{\beta-1}$ defines a doubling measure, $y^{\beta-1}\,dx\,dy$ on $\HH$ for any $\beta>0$ (see, for example, \cite[Definition 1.2.6]{Turesson_2000}), where $\,dx\,dy$ is Lebesgue measure on $\HH$. In the following Lemma \ref{lem:AnalogSobolev} and the sequel, we will need the following

\begin{defn}
\label{defn:Defnp}
Throughout our article, we fix
\begin{equation}
\label{eq:Defnp}
p := \frac{2(n+\beta)}{n+\beta-1},
\end{equation}
for any $\beta>0$.
\end{defn}

We recall the

\begin{lem} [Weighted Sobolev inequality] \cite[Lemma 4.2.4]{Koch}
\label{lem:AnalogSobolev}
Let $p$ be as in \eqref{eq:Defnp}. Then there is a positive constant $C=C(n,p)$ such that
\begin{equation}
\label{eq:AnalogSobolev}
\begin{aligned}
\int_{\HH} |u|^p y^{\beta-1}\,dx\,dy
&\leq c \left(\int_{\HH} |u|^2 y^{\beta-1}\,dx\,dy \right)^{\frac{p-2}{2}}
\int_{\HH} |\nabla u|^2 y^{\beta}\,dx\,dy,
\end{aligned}
\end{equation}
for any $u \in L^2\left(\HH, y^{\beta-1}\right)$ such that $\nabla u \in L^2\left(\HH,y^{\beta}\right)$.
\end{lem}

For $R>0$ and $z_0 \in \bar\sO$, we denote
\begin{align}
\label{eq:Balls_relative_to_a_subdomain}
B_R(z_0) &= \left\{z \in \sO: d(z,z_0) < R \right\},
\\
\label{eq:Balls_relative_to_the_half_space}
\BB_R(z_0) &= \left\{z \in \HH: d(z,z_0) < R \right\},
\end{align}
while
\begin{align*}
\bar B_R(z_0) = \left\{z \in \bar\sO: d(z,z_0) \leq R \right\} \quad\hbox{and}\quad
\bar \BB_R(z_0) = \left\{z \in \bar\HH: d(z,z_0) \leq R \right\},
\end{align*}
are the usual closures of $B_R(z_0)$ in $\bar\sO$ and of $\BB_R(z_0)$ in $\bar\HH$. Using definition \eqref{eq:IntroKochDistance} of the cycloidal distance, we obtain the following inclusions. For all $R>0$, we have
\begin{align}
\label{eq:EuclidBallInsideCycloid}
E_{R^2}(z_0) &\subset B_R(z_0),\quad\forall\,z_0\in\bar\HH, \\
\label{eq:SimpleCycloidBallInsideEuclid}
%CP 1.23.2016: Added here that the inclusion is true only for points on the boundary. This is ok for where we apply the inclusion.
%PF 2-27-2016: OK
B_R(z_0) &\subset E_{2R^2}(z_0),\quad\forall\, z_0\in\partial\HH.
\end{align}

Throughout the article we also use the following

\begin{defn}[Volume of sets]
If $S \subset \bar{\HH}$ is a Borel measurable subset, we let $|S|_{\beta}$ denote the volume of $S$ with respect to the measure $y^{\beta} \,dx\,dy$, and $|S|_{\fw}$ denote the volume of $S$ with respect to the measure $\fw \,dx\,dy$.
\end{defn}

We now recall
\begin{lem} \cite[Lemma 4.3.3]{Koch}
\label{lem:MeasureBalls}
There is a positive constant
%PF 2-23-2016: added
%CP 2.24.2016: Ok
$c \geq 1$, depending only on $n$ and $\beta$, such that, for any $R>0$ and $z_0 \in \bar{\HH}$,
\begin{equation}
\label{eq:MeasureBalls}
\begin{aligned}
c^{-1} R^n(R+\sqrt{y_0})^{n+2\beta}
\leq |\BB_R(z_0)|_{\beta}
\leq c R^n(R+\sqrt{y_0})^{n+2\beta}.
\end{aligned}
\end{equation}
Moreover, the following inclusions hold,
\begin{equation}
\label{eq:InclusionBalls}
\begin{aligned}
\EE_{R_1}(z_0) \subseteqq \BB_R(z_0) \subseteqq \EE_{R_2}(z_0),
\end{aligned}
\end{equation}
where $R_1 = R\left(R+\sqrt{y_0}\right)/2000$ and $R_2 = R\left(R+2\sqrt{y_0}\right)$.
\end{lem}

We have the following Poincar\'e inequalities, adapted to our weighted Sobolev spaces.

%CP 1.11.2016: Comment about the proof of the Poincare inequality. In our proof, the use of the scaling in the y-direction has the purpose to reduce the proof of the Poincare inequality to the application of the classical Poincare inequality, away from the degenerate boundary. We can see this in the penultimate inequality in the proof of the Lemma to which we apply [26, Equation (7.45)]. For this reason, we only need to scale in the degenerate direction y.
%PF 2-27-2016: OK
\begin{lem}[Poincar\'e inequality]
\label{lem:PoincareInequality}
Let $z_0\in\partial\HH$ and $R>0$. Then there is a positive constant $C$, depending on $\beta$, $n$ and $R$, such that for any $u \in H^1(\BB_R(z_0), \fw)$, we have
\begin{equation}
\label{eq:PoincareInequalityAnyBeta}
\begin{aligned}
\inf_{c \in \RR} \left(\int_{\BB_R(z_0)} |u(z)-c|^2 y^{\beta-1} \,dx\,dy\right)^{1/2}
\leq
C \left(\int_{\BB_R(z_0)} |\nabla u(z)|^2 y^{\beta} \,dx\,dy\right)^{1/2}.
\end{aligned}
\end{equation}
\end{lem}

As a consequence of Lemma \ref{lem:PoincareInequality}, we obtain

\begin{cor}[Poincar\'e inequality with scaling]
\label{cor:PoincareInequality}
%PF 2.21.2016: Changed $C$ to $C_0$ to avoid confusion
%CP 2.24.2016: Ok
There is a positive constant $C_0$, depending only on $\beta$ and $n$, such that for any $z_0\in\partial\HH$, $R>0$ and $u \in H^1(\BB_{R}(z_0), \fw)$  we have
\begin{equation}
\label{eq:CorPoincareInequality}
\begin{aligned}
&\inf_{c \in \RR} \left(\frac{1}{|\BB_R(z_0)|_{\beta-1}}\int_{\BB_R(z_0)} |u(z)-c|^2 y^{\beta-1} \,dx\,dy\right)^{1/2}\\
&\qquad \qquad \qquad \leq
C_0 R^2 \left(\frac{1}{|\BB_R(z_0)|_{\beta}}\int_{\BB_R(z_0)} |\nabla u(z)|^2 y^{\beta} \,dx\,dy\right)^{1/2}.
\end{aligned}
\end{equation}
\end{cor}

To prove Lemma \ref{lem:PoincareInequality} and Corollary \ref{cor:PoincareInequality}, we make use of the following extension property.

\begin{lem}[Extension operator]
\label{lem:Extension}
Let $z_0\in\partial\HH$ and $R>0$.
%CP 1.23.2016: Changed to a rectangle in n dimensions
%PF 2-27-2016: OK
Let $a_i, b_i\in\RR$, $a_i<b_i$, for all $1\leq i\leq n$, be such that
$D = \prod_{i=1}^n(a_i,b_i)$ is a rectangle with the property that $\BB_{R}(z_0) \subseteqq D$. Then, there is a continuous extension
\[
E : H^1(\BB_R(z_0),\fw) \rightarrow H^1(D,\fw),
\]
and there exists a positive constant $C$, depending on $D$, $R$, $n$ and $\beta$, such that for any $u \in H^1(\BB_R(z_0),\fw)$ we have
\begin{equation}
\label{eq:ExtensionContinuity}
\begin{aligned}
\|Eu\|_{L^2(D,y^{\beta-1})} &\leq C \|u\|_{L^2(\BB_R(z_0),y^{\beta-1})}, \\
\|\nabla Eu\|_{L^2(D,y^{\beta})} &\leq C \|\nabla u\|_{L^2(\BB_R(z_0),y^{\beta})}.
\end{aligned}
\end{equation}
\end{lem}

\begin{rmk}
\label{rmk:ReferenceSetting}
Without loss of generality, in the proofs of Lemmas \ref{lem:PoincareInequality} and  \ref{lem:Extension} and Corollary \ref{cor:PoincareInequality} we may assume $z_0=(0,0)$.
\end{rmk}

\begin{proof}[Proof of Lemma \ref{lem:PoincareInequality}]
%CP 1.23.2016: Changed in n-dimensional rectangle
Let
%PF 2.21.2016: Quantities were already fixed in lemma hypotheses
%CP 2.24.2016: This was in Lemma 2.7, not Lemma 2.5. I think we should
%leave them.
%PF 2-25-2016: Good point; left them.
$a_i<b_i$, for all $1\leq i\leq n-1$, and let
$\delta>0$ be such that $\BB_R(z_0) \subseteqq D_0\times(0,\delta)$, where we denote $D_0:=\prod_{i=1}^{n-1} (a_i,b_i)$. Let $k>1$ be such that
\begin{equation}
\label{eq:DefinitionKPoincare}
2 k^{-\beta}=\frac{1}{2}.
\end{equation}
Let $\hat u=Eu$ be the extension of $u$ to $D$ given by Lemma \ref{lem:Extension}. Assuming that \eqref{eq:PoincareInequalityAnyBeta} holds for $\hat u$, we obtain that it holds for $u$
also in the following way,
\begin{equation*}
\begin{aligned}
\inf_{c \in \RR} \left(\int_{\BB_R(z_0)} |u(z)-c|^2 y^{\beta-1} \,dx\,dy\right)^{1/2}
&\leq \inf_{c \in \RR} \left(\int_{D} |\hat u(z)-c|^2 y^{\beta-1} \,dx\,dy\right)^{1/2}\\
&\leq
C \left(\int_{D} |\nabla \hat u(z)|^2 y^{\beta} \,dx\,dy\right)^{1/2}\\
&\leq
C \left(\int_{\BB_R(z_0)} |\nabla u(z)|^2 y^{\beta} \,dx\,dy\right)^{1/2}.
\end{aligned}
\end{equation*}
%CP 1.23.2016: We only apply inequality \eqref{eq:ExtensionContinuity} in the last line
%PF 2-27-2016: OK
In the last inequality above, we made use of \eqref{eq:ExtensionContinuity}.

Therefore, we may assume $u \in H^1(D,\fw)$. Our goal is to prove that \eqref{eq:PoincareInequalityAnyBeta} holds for $u \in H^1(D,\fw)$. By \cite[Corollary A.14]{Daskalopoulos_Feehan_statvarineqheston}, we may assume without loss
of generality that $u \in C^1(\bar D)$. Let $c \in \RR$ and let $v = u-c$. Then, by the mean value theorem, we have for any $y \in (0,\delta)$ and
%CP 1.23.2016: Changed (a,b) to D_0
%PF 2-27-2016: OK
$x \in D_0$,
\[
v(x,y) = v(x,ky) + \int_{ky}^{y} v_y(x,t) dt.
\]
Squaring both sides of the preceding equation and integrating in $y$ with respect to $y^{\beta-1} \,dy$, we obtain
\begin{equation}
\label{eq:PoincareEq1}
\begin{aligned}
\int_{0}^{\delta} |v(x,y)|^2 y^{\beta-1} \,dy
\leq 2 \int_{0}^{\delta} |v(x,ky)|^2 y^{\beta-1} \,dy
+ 2\int_{0}^{\delta}\left|\int_{ky}^{y} v_y(x,t) dt\right|^2 y^{\beta-1} \,dy.
\end{aligned}
\end{equation}
By applying the change of variable $y'=ky$, we see that
\begin{equation}
\label{eq:PoincareEq2}
\begin{aligned}
\int_{0}^{\delta} |v(x,ky)|^2 y^{\beta-1} \,dy
= k^{-\beta}\int_{0}^{k\delta} |v(x,y')|^2 y'^{\beta-1} \,dy'.
\end{aligned}
\end{equation}
Also, we have for $\beta \neq 1$,
\begin{equation}
\label{eq:PoincareEq3}
\begin{aligned}
\int_{0}^{\delta}\left|\int_{ky}^{y} v_y(x,t) dt\right|^2 y^{\beta-1} \,dy
&=    \int_{0}^{\delta}\left|\int_{ky}^{y} v_y(x,t) t ^{\beta/2} t ^{-\beta/2}dt\right|^2 y^{\beta-1} \,dy\\
&\leq \frac{1}{|1-\beta|} \int_{0}^{\delta}\int_{y}^{ky} |v_y(x,t)|^2 t ^{\beta} dt  \left| y^{-\beta+1} - (ky)^{-\beta+1} \right| y^{\beta-1} \,dy\\
&\leq \delta \frac{1+k^{-\beta+1}}{|1-\beta|} \int_{0}^{k\delta} |v_y(x,y)|^2 y ^{\beta} \,dy.
\end{aligned}
\end{equation}
For $\beta=1$, a similar calculation gives us
\begin{equation}
\label{eq:PoincareEq3'}
\begin{aligned}
\int_{0}^{\delta}\left|\int_{ky}^{y} v_y(x,t) dt\right|^2  \,dy
&\leq \delta \log k  \int_{0}^{k\delta} |v_y(x,y)|^2 y \,dy.
\end{aligned}
\end{equation}
Define a positive constant $C_0\equiv C_0(\beta, \delta)$ by $C_0=2\delta (1+k^{-\beta+1})/|1-\beta|$ when $\beta \neq 1$, and $C_0=2\delta \log k$ when $\beta=1$. By combining equations \eqref{eq:PoincareEq1},
\eqref{eq:PoincareEq2}, \eqref{eq:PoincareEq3} and \eqref{eq:PoincareEq3'}, we obtain
\begin{equation*}
\begin{aligned}
\int_{0}^{\delta} |v(x,y)|^2 y^{\beta-1} \,dy
&\leq 2 k^{-\beta}\int_{0}^{k\delta} |v(x,y)|^2 y^{\beta-1} \,dy
+ C_0 \int_{0}^{k\delta} |v_y(x,y)|^2 y ^{\beta} \,dy  \\
&\leq 2 k^{-\beta}\int_{0}^{\delta} |v(x,y)|^2 y^{\beta-1} \,dy
+2 k^{-\beta}\int_{\delta}^{k\delta} |v(x,y)|^2 y^{\beta-1} \,dy \\
& \quad + C_0 \int_{0}^{k\delta} |v_y(x,y)|^2 y ^{\beta} \,dy.
\end{aligned}
\end{equation*}
Recall that $k>1$ was chosen such that \eqref{eq:DefinitionKPoincare} is satisfied. Therefore, by integrating also in $x$, there exists $C=C(\beta,\delta)$ such that
%CP 1.23.2016: Changed (a,b) to D_0
%PF 2-27-2016: OK
\begin{equation*}
\begin{aligned}
\int_{D_0}\int_{0}^{k\delta} |v(x,y)|^2 y^{\beta-1} \,dy  \,dx
&\leq C \int_{D_0} \int_{\delta}^{k\delta} |v(x,y)|^2 y^{\beta-1} \,dy \,dx
+ C \int_{D_0} \int_{0}^{k\delta} |v_y(x,y)|^2 y ^{\beta} \,dy \,dx.
\end{aligned}
\end{equation*}
Since $v=u-c$, we have
%CP 1.23.2016: Changed (a,b) to D_0
%PF 2-27-2016: OK
\begin{equation}
%CP 2.21.2016: Added reference
%PF 2-27-2016: OK
\label{eq:Poincare_ineq_shifted_up}
\begin{aligned}
&\inf_{c \in \RR} \int_D |u(x,y)-c|^2 y^{\beta-1} \,dy  \,dx\\
& \qquad
\leq C \inf_{c \in \RR} \int_{D_0} \int_{\delta}^{k\delta} |u(x,y)-c|^2 y^{\beta-1} \,dy \,dx  + C \int_D |u_y(x,y)|^2 y ^{\beta} \,dy \,dx.
\end{aligned}
\end{equation}
%CP 1.23.2016: Changed (a,b) to D_0
%CP 2.21.2016: Changed the rest of the proof
%PF 2-27-2016: OK
The rectangle $D':=D_0\times(\delta,k\delta)$ is a convex domain and so we may apply the classical Poincar\'e inequality \cite[Equation (7.45)]{GilbargTrudinger} to give
\begin{equation*}
\begin{aligned}
\inf_{c \in \RR} \int_{D_0} \int_{\delta}^{k\delta} |u(x,y)-c|^2 \,dy \,dx
\leq C \int_{D_0} \int_{\delta}^{k\delta} |\nabla u(x,y)|^2 \,dy \,dx.
\end{aligned}
\end{equation*}
Let $C':=(k\delta)^{\beta-1}$ if $\beta\geq 1$, and $C'=\delta^{\beta-1}$ if $\beta<1$. Then we see that $y^{\beta-1} \leq C'$, for all $y\in (\delta, k\delta)$, which gives
$$
\int_{D_0} \int_{\delta}^{k\delta} |u(x,y)-c|^2 y^{\beta-1}\,dy \,dx  \leq C'\int_{D_0} \int_{\delta}^{k\delta} |u(x,y)-c|^2\,dy \,dx, \quad\forall\, c\in\RR.
$$
Using in addition the inequality,
$$
\int_{D_0} \int_{\delta}^{k\delta} |\nabla u(x,y)|^2 \,dy \,dx \leq \delta^{-\beta}\int_{D_0} \int_{\delta}^{k\delta} |\nabla u(x,y)|^2 y^{\beta}\,dy \,dx,
$$
and combining it with the preceding two inequalities, we obtain that
\begin{equation*}
\begin{aligned}
%CP 1.23.2016: Changed (a,b) to D_0
%PF 2-27-2016: OK
\inf_{c \in \RR} \int_{D_0} \int_{\delta}^{k\delta} |u(x,y)-c|^2 y^{\beta-1} \,dy \,dx  \leq C \int_{D_0} \int_{\delta}^{k\delta} |\nabla u(x,y)|^2 y^{\beta} \,dy \,dx,
\end{aligned}
\end{equation*}
where $C=C(\beta, D_0, \delta, k)$ is a positive constant. Because the domain $D_0$ and $\delta$ depend only on $R$, and $k$ depends on $\beta$, the constant $C$ in the preceding inequality depends only on $\beta, n$ and $R$. Combining the preceding inequality with \eqref{eq:Poincare_ineq_shifted_up} yields \eqref{eq:PoincareInequalityAnyBeta}.
\end{proof}

\begin{rmk}
Koch states a weighted Poincar\'e inequality on the half-space \cite[Lemma 4.4.4]{Koch}, with weight $y^{\beta-1} e^{-\kappa\rho(z,z_0)}$, where $\kappa$ is a positive constant, $z_0$ is a fixed point in $\bar{\HH}$, and $\rho(z,z_0)$ is equivalent to $d^2(z,z_0)$, in the sense that there exists a constant $c>0$ such that
\[
%CP 1.11.2016: Added a \quad
%PF 2-27-2016: OK
c d^2(z,z_0) \leq \rho(z,z_0) \leq \frac{1}{c} d^2(z,z_0), \quad\forall\, z \in \HH.
\]
The proof of this result is long and technical. So, rather than use this result to prove a weighted Poincar\'e inequality on a ball using an extension principle, we give a much simpler proof for balls and weights $y^{\beta-1}$ and
$y^{\beta}$.
\end{rmk}

%CP 1.11.2016: We don't use this remark anywhere in the paper and I think that we can remove it.
%PF 2-27-2016: OK
\begin{comment}
\begin{rmk}
When $\beta \geq 1$, from \cite[Lemma A.1 and A.4]{Daskalopoulos_Feehan_statvarineqheston} we have that $H^1_0(\sO,\fw)=H^1_0(\sO\cup\Gamma_0,\fw)$. Then, as in the case of the Poincar\'e inequality for finite-width domains
\cite[\S 6.26]{Adams_1975}, it might be true that the stronger version of \eqref{eq:PoincareInequalityAnyBeta} holds
\begin{equation}
\label{eq:PoincareInequalityAnyBetaGE1}
\begin{aligned}
\left(\int_{\BB_R(z_0)} |u(z)|^2 y^{\beta-1} \,dx\,dy\right)^{1/2}
&\leq
C \left(\int_{\BB_R(z_0)} |\nabla u(z)|^2 y^{\beta} \,dx\,dy\right)^{1/2}.
\end{aligned}
\end{equation}
\end{rmk}
\end{comment}

\begin{rmk}[Scaling under Koch metric]
\label{rmk:ScalingPropertiesKochMetric}
Using the definitions \eqref{eq:IntroKochDistance} for the cycloidal distance
%PF 2.21.2016: added
%CP 2.24.2016: Ok
and \eqref{eq:Balls_relative_to_the_half_space} for the ball $\BB_R(z_0)$,
we obtain the following scaling property
\begin{equation}
\label{eq:ScalingPropertyKochMetric}
%PF 2.21.2016: added z_0 \in \partial\HH explicitly
%CP 2.24.2016: Ok
\BB_{R_1}(z_0) = \left(\frac{R_1}{R_2}\right)^2 \BB_{R_2}(z_0), \quad \forall\, R_1, R_2 >0 \text{ and } z_0 \in \partial\HH,
\end{equation}
%PF 2.21.2016: added
%CP 2.24.2016: Corrected
%PF 2-25-2016: OK
since
%$d(t^2z, z_0) = td(z, z_0)$
$d(z_0+t^2(z-z_0), z_0) = td(z, z_0)$ for all $z \in \HH$, $z_0 \in \partial\HH$, and $t > 0$.
Notice that \eqref{eq:ScalingPropertyKochMetric} does not hold if $z_0=(x_0, y_0)$ with $y_0>0$.
\end{rmk}

\begin{proof} [Proof of Corollary
%PF 2.21.2016: Wrong reference - \ref{eq:CorPoincareInequality}
%CP 2.24.2016: Ok
\ref{cor:PoincareInequality}]
Let $R>0$ and $\bar R>0$ and define $v$ by rescaling
\[
%CP 2.24.2016: Corrected
%PF 2-25-2016: OK
%u(z) = v \left(\left(\frac{\bar R}{R}\right)^2 z\right), \quad \forall\, z\in  \BB_{R}(z_0).
u(z) = v \left(z_0+\left(\frac{\bar R}{R}\right)^2 (z-z_0)\right), \quad \forall\, z\in  \BB_{R}(z_0).
\]
The rescaling map defined by
%CP 2.24.2016: Corrected
%PF 2-25-2016: OK
%$z \mapsto (\bar R/R)^2 z$
$z \mapsto z_0+(\bar R/R)^2 (z-z_0)$
maps $\BB_{R}(z_0)$ into $\BB_{\bar R}(z_0)$ by Remark \ref{rmk:ScalingPropertiesKochMetric}. By applying Lemma \ref{lem:PoincareInequality} to $v$ on $\BB_{\bar R}(z_0)$, there is a positive constant
%PF 2.21.2016: C -> C_0
%CP 2.24.2016: Ok
$C_0$, depending only on $\bar R$, $n$ and $\beta$, such that \eqref{eq:PoincareInequalityAnyBeta} holds. By changing variables, we obtain
\begin{equation}
\label{eq:RescaledPoincare}
\begin{aligned}
\inf_{c \in \RR} \left(\frac{\bar R}{R}\right)^{2(\beta-1)} \int_{\BB_R(z_0)} |u-c|^2 y^{\beta-1} \,dx\,dy
& \leq
%PF 2.21.2016: Added C_0
%CP 2.24.2016: Ok
C_0\left(\frac{R}{\bar R}\right)^4  \left(\frac{\bar R}{R}\right)^{2\beta} \int_{\BB_R(z_0)} |\nabla u|^2 y^{\beta} \,dx\,dy.
\end{aligned}
\end{equation}
Using Lemma \ref{lem:MeasureBalls}, we rewrite \eqref{eq:RescaledPoincare} in the following form
\begin{equation*}
\begin{aligned}
\inf_{c \in \RR} \frac{|\BB_{\bar R}(z_0)|_{\beta-1}}{|\BB_R(z_0)|_{\beta-1}} \int_{\BB_R(z_0)} |u-c|^2 y^{\beta-1} \,dx\,dy
&\leq
%PF 2.21.2016: Added C_0
%CP 2.24.2016: Ok
C_0\left(\frac{R}{\bar R}\right)^4\frac{|\BB_{\bar R}(z_0)|_{\beta}}{|\BB_R(z_0)|_{\beta}} \int_{\BB_R(z_0)} |\nabla u|^2 y^{\beta} \,dx\,dy,
\end{aligned}
\end{equation*}
from which \eqref{eq:CorPoincareInequality} follows immediately
%PF 2.21.2016: Added
%CP 2.24.2016: Ok
by taking $\bar R=1$.
\end{proof}

\section{John-Nirenberg inequality}
\label{sec:JohnNirenberg}
In this section we recall the abstract John-Nirenberg inequality (Theorem \ref{thm:AbstractJohnNirenberg}) due to E. Bombieri and E. Giusti \cite{Bombieri_Giusti_1972} and, in particular, provide a justification --- via
Proposition \ref{prop:ApplicationAbstractJohnNirenberg} --- that its hypotheses hold in the setting of the problems described in \S \ref{sec:Introduction}.

We restrict the statement of \cite[Theorem 4]{Bombieri_Giusti_1972} to the framework of our problems, so in \cite[Theorem 4]{Bombieri_Giusti_1972} we choose $\HH$ to be the topological space and $d\mu = y^{\beta-1}\,dx\,dy$ to be
the regular positive Borel measure on $\HH$. Let $S_r$, $0\leq r\leq 1$ be a family of non-empty open sets in $\HH$ such that
\begin{equation}
\begin{aligned}
& S_s \subseteqq S_r, & 0\leq s\leq r\leq 1,
\\
& 0< |S_r|_{\beta-1} < \infty, &\forall\, r \in [0, 1].
\end{aligned}
\end{equation}
Let $w$ be a measurable positive function on $S_1$. For $t\neq 0$ and $0\leq r \leq 1$, we denote by
\begin{equation*}
\begin{aligned}
|w|_{t,r} &= \left(\frac{1}{|S_r|_{\beta-1}} \int_{S_r} |w|^t y^{\beta-1}\,dx\,dy\right)^{1/t},\\
|w|_{\infty,r} &= \esssup_{S_r} w,\\
|w|_{-\infty,r} &= \essinf_{S_r} w.
\end{aligned}
\end{equation*}
We now recall the

\begin{thm}[Abstract John-Nirenberg Inequality]
\label{thm:AbstractJohnNirenberg}
\cite[Theorem 4]{Bombieri_Giusti_1972}
Let $0< \theta_0, \theta_1\leq \infty$ and $w$ be a measurable positive function on $S_1$ such that
\[
|w|_{\theta_0,1} < \infty \text{    and     }  |w|_{\theta_1,1} > 0.
\]
Suppose there exist constants $\gamma>0$, $0 < t^* \leq \frac{1}{2} \min\{\theta_0,\theta_1\}$ and $Q>0$ such that for all $0\leq s<r\leq 1$ and $0<t \leq t^*$,
\begin{equation}
\label{eq:AbstractJohnNirenberg1}
\begin{aligned}
|w|_{\theta_0,s} \leq \left(Q(r-s)^{\gamma}\right)^{1/\theta_0-1/t} |w|_{t,r},\\
|w|_{-\theta_1,s} \geq \left(Q(r-s)^{\gamma}\right)^{1/t-1/\theta_1} |w|_{-t,r}.
\end{aligned}
\end{equation}
Assume further that
\begin{equation}
\label{eq:AbstractJohnNirenberg2}
\begin{aligned}
A := \sup_{0\leq r\leq 1} \inf_{c \in \RR} \frac{1}{|S_r|_{\beta-1}} \int_{S_r} |\log w - c | y^{\beta-1}\,dx\,dy < \infty.
\end{aligned}
\end{equation}
Then, we have
\begin{equation}
\label{eq:AbstractJohnNirenberg}
\begin{aligned}
|w|_{\theta_0,0}
\leq
\left(\frac{|S_1|_{\beta-1}}{|S_0|_{\beta-1}}\right)^{1/\theta_0+1/\theta_1}
\exp\left\{c_2Q^{-2}\left(A+1/t^*\right)\right\} |w|_{-\theta_1,0},
\end{aligned}
\end{equation}
where $c_2$ is a constant depending only on $\gamma$, but not on $Q, \theta_0, \theta_1, t^*,A$ and $\beta$.
\end{thm}

In many of our proofs, we will make use of a sequence of cutoff functions, $\{\eta_N\}_{N\in\NN}$. Let $\varphi:\mathbb{R}\rightarrow [0,1]$ be a smooth function such that $\varphi(x)\equiv 1$ for $x<0$, and $\varphi \equiv 0$ for $x>1$. Let $z_0 \in\HH$ and let $\{R_N\}_{N\geq 0}$ be an non-increasing sequence of positive numbers. We define
\begin{equation}
\label{eq:GlobalDefinitionCutOffFunction}
\eta_N(z) := \varphi\left(\frac{1}{R^2_{N-1}-R^2_N} (d^2(z_0,z)-R^2_N)\right), \quad \forall\, z\in\bar\HH, \quad\forall\, N\in\NN.
\end{equation}
Then, the sequence $\{\eta_N\}_{N\geq 1}$ has the following properties,
\begin{align}
\label{eq:RangeCutOffFunction}
&\eta_N|_{B_{R_N}(z_0)}\equiv 1, \quad \eta_N|_{B^c_{R_{N-1}(z_0)}}\equiv 0,
\\
\label{eq:GlobalBoundCutOffFunction}
&\left|\nabla \eta_N\right| \leq  \frac{C}{R^2_{N-1}-R^2_N},
\end{align}
where $B^c_{R_{N-1}}(z_0) := \HH\less\bar B_{R_{N-1}}(z_0)$ and $C$ is a positive constant independent of $N$ and the sequence
$\{R_N\}_{N\geq 0}$. The bound in \eqref{eq:GlobalBoundCutOffFunction} can be deduced from the calculation,
\[
\nabla \eta_N = \varphi'\left(\frac{1}{R^2_{N-1}-R^2_N} (d^2(z_0,z)-R^2_N)\right) \frac{1}{R^2_{N-1}-R^2_N} \nabla d^2(z_0,z).
\]
Also, we have that $|\nabla d^2(z_0,z)| \leq 5$, for all $z_0, z \in \HH$. Since $\varphi'$ is also uniformly bounded on $\RR$, we obtain
%PF 2.21.2016: Added
%CP 2.24.2016: Ok
the forthcoming inequality
\eqref{eq:BoundCutOffFunction}.

Similarly, we can construct a sequence of cutoff functions, $\{\eta_N\}_{N\in\NN}$, when
$\{R_N\}_{N \geq 0}$ is a non-decreasing sequence of positive numbers.

We now provide a justification that the hypotheses of Theorem \ref{thm:AbstractJohnNirenberg} hold in the setting of the problems discussed in this article.

\begin{prop} [Application of Theorem \ref{thm:AbstractJohnNirenberg}]
\label{prop:ApplicationAbstractJohnNirenberg}
Let $z_0 \in \partial\HH$ and $0<4R\leq 1$. Let $S_r=\BB_{(2+r)R}(z_0)$, for all $0\leq r \leq 1$. Let $\theta_0,\theta_1$ be as in Theorem \ref{thm:AbstractJohnNirenberg} and set $t^* = \frac{1}{2} \min\{\theta_0,\theta_1\}$. Then, there exist positive constants $Q$ and $\gamma$, independent of $R$ and $z_0$, such that \eqref{eq:AbstractJohnNirenberg} holds for any bounded positive function $w$ on $S_1$ which satisfies
%PF 2.21.2016: Added
%CP 2.24.2016: Ok
the forthcoming
energy estimates \eqref{eq:EnergyEstimate} or \eqref{eq:HarnackEnergyEstimates},
%CP 2.24.2016: Added where p is defined
%PF 2-27-2016: OK
where we recall that $p$ is defined in \eqref{defn:Defnp}.
\end{prop}

\begin{proof}
We begin by proving the first inequality in \eqref{eq:AbstractJohnNirenberg1} by applying Moser iteration finitely many times. The second inequality in \eqref{eq:AbstractJohnNirenberg1} can be proved by a similar technique. We
give the proof when $w$ satisfies the energy estimate \eqref{eq:EnergyEstimate}, but the proof applies as well to positive bounded functions $w$ satisfying the energy estimate \eqref{eq:HarnackEnergyEstimates}.

%CP 2.24.2016: Added here
%PF 2-27-2016: OK
As in the hypotheses of Theorem \ref{thm:AbstractJohnNirenberg}, we let $t\in (0, t^*]$. First, we consider the \emph{special case} when $\theta_0$ and
%PF 2.23.2016: t does not appear in the hypothesis of Proposition \ref{prop:ApplicationAbstractJohnNirenberg}
%CP 2.24.2016: Added at the beginning of this paragraph.
%PF 2-27-2016: OK
$t$ satisfy the requirement: There exists an integer $N^* \geq 1$ such that $\theta_0$ can be written as
\begin{equation}
\label{eq:SpecialCase}
\begin{aligned}
\theta_0 = t \left(\frac{p}{2}\right)^{N^*}.
\end{aligned}
\end{equation}
%PF 2-23-2016: What is p? While p appears in \eqref{eq:EnergyEstimate} and \eqref{eq:HarnackEnergyEstimates}, its range needs to be recalled in the hypotheses in this proposition.
%CP 2.24.2016: p is defined in Definition 2.1. It never changes in the paper, it is always this value. I added this to the hypotheses of the proposition.
%PF 2-27-2016: OK
Let $0 \leq s < r \leq 1$ and set $R_0=(2+r)R$. We denote
$$
c := \sum_{k=1}^{\infty} \frac{1}{k^2}
$$
and we let
\begin{equation}
\label{eq:JNRadius}
R^2_N := \left((2+r)^2-(r-s)^2\sum_{k=1}^N \frac{1}{ck^2} \right)R^2, \quad  N=1,\ldots,N^*.
\end{equation}
We observe that $(2+s)R < R_N < R_{N-1} \leq (2+r)R$. Let $\{\eta_N\}_{N\in\NN}$ be a sequence of non-negative, smooth cutoff functions as constructed in \eqref{eq:GlobalDefinitionCutOffFunction}, by choosing $R_N$ as in \eqref{eq:JNRadius}. Then,
\eqref{eq:GlobalBoundCutOffFunction} becomes
\begin{align}
\label{eq:BoundCutOffFunction}
\left|\nabla \eta_N\right| \leq \frac{CN^2}{R^2(r-s)^2}.
\end{align}

Let $P_N:=t \left(p/2\right)^N$, for $N=1,\ldots, N^*$, and $\alpha_N=p_N-1$, for all $N=0,\ldots, N^*-1$. We set
\begin{equation}
\label{eq:DefinitionINJohnNirenberg}
\begin{aligned}
I(N) := \left(\int_{\BB_{R_N}(z_0)} |w|^{p_{N}} y^{\beta-1}\,dx \,dy\right)^{1/p_{N}},
\end{aligned}
\end{equation}
From our hypothesis, $w$ satisfies \eqref{eq:EnergyEstimate}, that is,
\begin{equation}
\label{eq:EnergyEstimateJohnNirenberg}
\begin{aligned}
\| \eta w^{(\alpha+1)/2} \|_{L^{p}\left(\HH,y^{\beta-1}\right)}
&\leq
C_0(R,\alpha) \|w^{(\alpha+1)/2}\|_{L^2\left(\supp\eta, y^{\beta-1}\right)},
\end{aligned}
\end{equation}
where
\begin{equation}
\label{eq:EnergyEstimateConstantJohnNirenberg}
\begin{aligned}
C_0(R,\alpha) &:= \left(C|1+\alpha|\right)^{(\xi+1)/p} \left(1+ \|\sqrt{y}\nabla \eta\|_{L^{\infty}(\HH)}^2\right)^{1/p},
\end{aligned}
\end{equation}
and $\xi$ and $C$ are positive constants, independent of $w$, $\alpha$ and $\eta$. We choose $\alpha=\alpha_{N-1}$ and $\eta=\eta_N$ in \eqref{eq:EnergyEstimateJohnNirenberg}, so the definition \eqref{eq:DefinitionINJohnNirenberg}
gives us, for all $N\geq 1$,
\begin{equation}
\label{eq:NewEquation}
\begin{aligned}
I(N) \leq C_1(R,r,s,N) I(N-1),
\end{aligned}
\end{equation}
where
\begin{equation*}
\begin{aligned}
C_1(R,r,s,N):=\left(C|p_{N-1}|\right)^{(\xi+1)/p_{N}}\left(  1+\|\sqrt{y}\nabla\eta_N\|^2_{L^{\infty}(\HH)}\right)^{1/p_{N}}.
\end{aligned}
\end{equation*}
From Lemma \ref{lem:MeasureBalls}, we have $y \leq C R^2$ on $\BB_{R_N}(z_0)$, where $C$ is a positive constant independent of $R$ and $N$. Using the bound
\eqref{eq:BoundCutOffFunction}, we obtain
\begin{equation*}
\begin{aligned}
C_1(R,r,s,N):=\left(C|p_{N-1}|\right)^{(\xi+1)/p_{N}}\left( \frac{CN^4}{R^2(r-s)^4}\right)^{1/p_{N}}.
\end{aligned}
\end{equation*}
By iterating inequality \eqref{eq:NewEquation}, we obtain
\begin{equation}
\label{eq:Eq10}
\begin{aligned}
I(N^*) &\leq C_2(R,r,s) I(0),
\end{aligned}
\end{equation}
where
\begin{equation}
\label{eq:DefnCRrs}
C_2(R,r,s) := \prod_{N=1}^{N^*} \left(C p_{N-1}^{\xi+1} N^4 R^{-2}(r-s)^{-4}\right)^{1/p_N}.
\end{equation}
Next, we prove the

\begin{claim}
\label{claim:CoeffIteration}
There are positive constants $Q$ and $\gamma$, independent of $N^*,R,r$ and $s$, such that
\begin{equation}
\label{eq:CoeffIteration}
\begin{aligned}
C_2(R,r,s) \leq \left(Q(r-s)^{\gamma}\right)^{1/\theta_0-1/t} R^{\frac{4}{p-2}\left(1/\theta_0-1/t\right)}.
\end{aligned}
\end{equation}
\end{claim}

\begin{proof}[Proof of Claim \ref{claim:CoeffIteration}]
We can rewrite the expression \eqref{eq:DefnCRrs} for $C_2(R,r,s)$ to obtain
\begin{align}
\label{eq:CoeffIteration1}
C_2(R,r,s)
&\leq \left(C t^{\xi+1} R^{-2}(r-s)^{-4}\right)^{\sum_{N=1}^{N^*} 1/p_N}  \left(C \frac{p}{2}\right)^{\sum_{N=1}^{N^*}N/p_N},
\end{align}
where we used in the last line that $N^4 \leq C (p/2)^N$, for some positive constant $C=C(p)$. Equation \eqref{eq:SpecialCase} leads to the identities
\[
\sum_{N=1}^{N^*} \frac{1}{p_N} = \frac{2}{p-2}\left(\frac{1}{t} - \frac{1}{\theta_0}\right)
\quad\hbox{and}\quad
 \sum_{N=1}^{N^*}  \frac{N}{p_N} = \frac{4}{p(p-2)}\left(\frac{1}{t} - \frac{1}{\theta_0}\right).
\]
Therefore, inequality \eqref{eq:CoeffIteration} becomes
\begin{equation}
\label{eq:CoeffIteration2}
\begin{aligned}
C_2(R,r,s)
&\leq \left( R^{-2}(r-s)^{-4}\right)^{\frac{2}{p-2}\left(\frac{1}{t} - \frac{1}{\theta_0}\right)}
 \left(C \theta_0^{\xi+1} \frac{p}{2}\right)^{\frac{4}{p(p-2)}\left(\frac{1}{t} - \frac{1}{\theta_0}\right)},
\end{aligned}
\end{equation}
which is equivalent to \eqref{eq:CoeffIteration} with the choice of the constants
$Q =\left(C \theta_0^{\xi+1}p/2\right)^{-1}$ and $\gamma = 8/(p-2)$. This completes the proof of Claim \ref{claim:CoeffIteration}.
\end{proof}
%PF 2-23-2016: What fact? Reference? We assumed p obeys 2 < p < \infty?
%CP 2.24.2016: Added reference
%PF 2-27-2016: OK
From identity \eqref{eq:Defnp}, we have that $4/(p-2)=2(n+\beta-1)$, and so Lemma \ref{lem:MeasureBalls}
%PF 2.23.2016 added
%CP 2.24.2016: Ok
(with constant $c_0 = c_0(n,\beta) > 1$) yields
\begin{align*}
\frac{|\BB_{(2+s)R}(z_0)|_{\beta-1}^{1/\theta_0}}{|\BB_{(2+r)R}(z_0)|_{\beta-1}^{1/t}}
%PF 2.23.2016 added
%CP 2.24.2016: Ok
&\geq
\frac{c_0^{-1/\theta_0} ((2+s)R)^{2(n+\beta-1)/\theta_0} } {c_0^{1/t} ((2+r)R)^{2(n+\beta-1)/t} }
\\
%PF 2.23.2016 added
%CP 2.24.2016: Ok
&=
c_0^{-1/\theta_0-1/t} \left(\frac{2+s}{2+r}\right)^{2(n+\beta-1)(1/\theta_0-1/t)} R^{2(n+\beta-1)(1/\theta_0-1/t)}
\\
%PF 2.23.2016 corrected 4/(p-2) --> (4/(p-2))
%CP 2.24.2016: Ok
&\geq C^{1/\theta_0+1/t} R^{(4/(p-2))(1/\theta_0-1/t)},
\end{align*}
for a positive constant $C = 1/c_0 < 1$ and recalling that $0 \leq s < r \leq 1$ and $\theta_0 > t$ by \eqref{eq:SpecialCase}.
%PF 2.23.2016 maybe "by \eqref{eq:SpecialCase}" should be replaced by
%"by hypothesis of Proposition \ref{prop:ApplicationAbstractJohnNirenberg}" after adding t
%CP 2.24.2015: Both are correct, but \eqref{eq:SpecialCase} says
%exactly what from the hypotheses of the proposition is used. Please
%choose whichever you prefer.
%PF 2-25-2016: It is OK now
%CP 2.25.2016: Ok
Therefore, inequality
%PF 2.23.2016 corrected reference
%CP 2.24.2016: Ok
%\eqref{eq:CoeffIteration2}
\eqref{eq:CoeffIteration}
becomes
\begin{equation}
\label{eq:CoeffIteration3}
\begin{aligned}
C_2(R,r,s)
&\leq
C^{-1/\theta_0-1/t} \left(Q(r-s)^{\gamma}\right)^{1/\theta_0 -1/t}  \frac{|\BB_{(2+s)R}(z_0)|_{\beta-1}^{1/\theta_0}}{|\BB_{(2+r)R}(z_0)|_{\beta-1}^{1/t}}.
\end{aligned}
\end{equation}
From
%PF 2.23.2016 changed
%CP 2.24.2016: Ok
the hypothesis of Proposition \ref{prop:ApplicationAbstractJohnNirenberg} that
%our hypothesis,
$t \leq t^*\leq\theta_0/2$, we have
\[
3(1/\theta_0-1/t) \leq -1/\theta_0-1/t \leq 1/\theta_0-1/t,
\]
and so, for a new positive constant $Q$, the inequality \eqref{eq:CoeffIteration3} leads to
\begin{equation}
\label{eq:CoeffIteration4}
\begin{aligned}
C_2(R,r,s)
&\leq
\left(Q(r-s)^{\gamma}\right)^{1/\theta_0 -1/t}  \frac{|\BB_{(2+s)R}(z_0)|_{\beta-1}^{1/\theta_0}}{|\BB_{(2+r)R}(z_0)|_{\beta-1}^{1/t}}.
\end{aligned}
\end{equation}
By employing the inequalities \eqref{eq:CoeffIteration4} and \eqref{eq:Eq10} and the definition \eqref{eq:DefinitionINJohnNirenberg} of $I(N)$, we obtain
\begin{equation*}
\begin{aligned}
&\left(\int_{\BB_{(2+s)R}(z_0)} |w|^{\theta_0} y^{\beta-1} \,dx\,dy\right)^{1/\theta_0}\\
&\qquad\qquad \leq
  \left(Q(r-s)^{\gamma}\right)^{1/\theta_0 -1/t}  \frac{|\BB_{(2+s)R}(z_0)|_{\beta-1}^{1/\theta_0}}{|\BB_{(2+r)R}(z_0)|_{\beta-1}^{1/t}}
  \left(\int_{\BB_{(2+r)R}(z_0)} |w|^{t} y^{\beta-1} \,dx\,dy\right)^{1/t},
\end{aligned}
\end{equation*}
from which we readily obtain the first inequality in \eqref{eq:AbstractJohnNirenberg1}, in the \emph{special case} where $t$ and $\theta_0$ satisfy \eqref{eq:SpecialCase} for some integer $N^*\geq 1$.

Next, we show that the first inequality in \eqref{eq:AbstractJohnNirenberg1} holds for \emph{any} $t \in (0,t^*)$. For this purpose, we choose an integer $N^*\geq 1$ such that
\[
t \left(\frac{p}{2}\right)^{N^*-1} < \theta_0 < t \left(\frac{p}{2}\right)^{N^*}.
\]
We denote $\theta^*_0 = t \left(p/2\right)^{N^*}$ and we apply the previous analysis to $t$ and $\theta_0^*$, which now satisfy \eqref{eq:SpecialCase}, to give
\begin{equation*}
\begin{aligned}
|w|_{\theta^*_0,s} &\leq \left(Q(r-s)^{\gamma}\right)^{1/\theta^*_0-1/t} |w|_{t,r}.
\end{aligned}
\end{equation*}
Using H\"older's inequality with $p=\theta^*_0/\theta_0>1$, we find that
\[
|w|_{\theta_0,s} \leq |w|_{\theta^*_0,s},
\]
and so
\begin{equation*}
\begin{aligned}
|w|_{\theta_0,s}
&\leq \left(Q(r-s)^{\gamma}\right)^{1/\theta^*_0-1/t} |w|_{t,r}\\
&\leq \left(Q(r-s)^{\gamma}\right)^{\frac{1/\theta^*_0-1/t}{1/\theta_0-1/t} \left(1/\theta_0-1/t\right)} |w|_{t,r}.
\end{aligned}
\end{equation*}
Notice that $2\theta^*_0/p \leq \theta_0 \leq \theta^*_0$ and $0 < t < \theta_0/2$. Then,
\begin{align*}
1 \leq \frac{1/\theta^*_0-1/t}{1/\theta_0-1/t}
  \leq \frac{1/\theta^*_0 - 1/t}{p/2\theta^*_0-1/t}
  \leq \frac{(2/p)^{N^*}-1}{(2/p)^{N^*+1}-1}
  \leq \frac{p}{p-2}.
\end{align*}
Consequently, we define $\widetilde{Q}$ to be $Q^{p/(p-2)}$ if $Q<1$, and we leave $Q$ unchanged if $Q \geq 1$ and, setting $\widetilde{\gamma}:=\gamma p/(p-2)$, the preceding estimate for $|w|_{\theta_0,s}$ becomes
\begin{equation*}
\begin{aligned}
|w|_{\theta_0,s} &\leq \left(\widetilde{Q}(r-s)^{\widetilde{\gamma}}\right)^{1/\theta_0-1/t} |w|_{t,r},
\end{aligned}
\end{equation*}
which is precisely the first inequality in \eqref{eq:AbstractJohnNirenberg1}.
\end{proof}

\section{Supremum estimates near the degenerate boundary}
\label{sec:SupremumEstimates}
In this section, we prove Theorems \ref{thm:MainSupremumEstimatesInterior} and \ref{thm:MainSupremumEstimatesBoundary}
and Corollary \ref{cor:MainSupremumEstimatesBoundary}, that is, local boundedness up to $\bar\Gamma_0$ for
subsolutions (respectively, supersolutions), $u$, to the variational equation \eqref{eq:IntroHestonWeakMixedProblemHomogeneous}. Our choice of test functions when applying Moser iteration follows that employed in the proof of \cite[Theorem 8.15]{GilbargTrudinger}. However, the choice of test functions used in the proof of the classical local supremum estimates \cite[Theorem 8.17]{GilbargTrudinger} is not suitable in our case because the test functions in \eqref{eq:IntroHestonWeakMixedProblemHomogeneous} are not required to satisfy a homogeneous Dirichlet boundary condition along $\bar\Gamma_0$. In addition, the method of deriving the energy estimate \eqref{eq:SupEstEnergyEst} is slightly different from \cite[Theorem 8.18]{GilbargTrudinger} because, instead of using the classical Sobolev inequalities \cite[Theorem 7.10]{GilbargTrudinger}, we use Lemma \ref{lem:AnalogSobolev}.

We begin with
%CP 1.11.2016: There should be no "the", is this right?
%PF 2-27-2016: Maybe better included I think :-)
the

\begin{lem}
\label{lem:RegularDomain}
Let $K$ be a finite, right circular cone and $\sO$ be an open subset which obeys the uniform interior and exterior cone condition on $\bar \Gamma_0 \cap \bar \Gamma_1 $ with cone $K$. Then, there are positive constants $\bar R$ and $c$ depending on $K$, $n$ and $\beta$ such that, for all $R \in (0,\bar R]$, we have
\begin{equation}
\label{eq:DomainCondition_interior}
c^{-1} |\BB_R(z_0)|_{\beta-1} \leq |B_R(z_0)|_{\beta-1} \leq c |\BB_R(z_0)|_{\beta-1},
\quad\forall\, z_0 \in \bar \Gamma_0,
\end{equation}
and also
\begin{equation}
\label{eq:DomainCondition_exterior}
c^{-1} |\BB_R(z_0)|_{\beta-1} \leq |\BB_R(z_0)\backslash B_R(z_0)|_{\beta-1} \leq c |\BB_R(z_0)|_{\beta-1},
\quad \forall\, z_0 \in \bar \Gamma_0 \cap \bar \Gamma_1.
\end{equation}
\end{lem}

An open subset, $\sO$, which does not satisfy condition \eqref{eq:DomainCondition_interior} can be created along the lines of \cite[Example 4.2.17]{KaratzasShreve1991} (Lebesgue's thorn); see
%CP 1.11.2016: Changed v1 to v2
%PF 2-27-2016: Can we use just one version preprint reference? We're citing v2 and v3 here
%CP 2.28.2016: This example is not in v3, only in v2. This is the only place where we cite v2 and in all other places we cite v3
%PF 3-5-2016: I suggest we just add this example back so we can just reference v3.
\cite[Example A.1]{Feehan_Pop_regularityweaksoln_v2}.

\begin{proof}[Proof of Lemma \ref{lem:RegularDomain}]
The proof of the lemma can be obtained just as in the case of the Euclidean distance function with the aid of Lemma \ref{lem:MeasureBalls}. Complete details are provided in the proof of \cite[Lemma 4.1]{Feehan_Pop_regularityweaksoln_v3}.
\end{proof}

We can now proceed to the

\begin{proof}[Proof of Theorems \ref{thm:MainSupremumEstimatesInterior} and \ref{thm:MainSupremumEstimatesBoundary}]
For the proof of Theorem \ref{thm:MainSupremumEstimatesInterior}, we choose $\bar R < \sqrt{R_0/2}$. For the proof of Theorem \ref{thm:MainSupremumEstimatesBoundary}, we choose $\bar R$ smaller than $\sqrt{R_0/2}$ and than the constant $\bar R$ appearing in the conclusion of
Lemma \ref{lem:RegularDomain}. Notice that  \eqref{eq:SimpleCycloidBallInsideEuclid} shows that $B_{\bar R}(z_0) \subset E_{R_0}(z_0)$.

\begin{step} [Energy estimates]
\label{step:EnergyEstimates}
Let $\alpha \geq 1$ and let $\eta \in C^1_0(\bar{\HH})$ be a non-negative cutoff function with support in $\bar\BB_{2R}(z_0)$,
where $R$ is chosen such that $0<2R<\bar R$. We define
\begin{equation}
\label{eq:DefAandW}
\begin{aligned}
A &:= \|f \|_{L^s(\supp\eta, y^{\beta-1})}.
\end{aligned}
\end{equation}
We will apply the calculations in Steps \ref{step:EnergyEstimates} and \ref{step:MoserIterations} to $w$ defined by
\begin{equation}
\label{eq:Defnw}
w := u^+(u^-)+A.
\end{equation}
For concreteness, we will illustrate our calculations with the choice $w=u^+ + A$ (when u is a subsolution), but they apply equally well to the choice $w=u^- + A$
(when u is a supersolution). Our goal in Step \ref{step:EnergyEstimates} is to prove the following

\begin{claim}[Energy estimate]
\label{claim:SupEstEnergyEst}
There are positive constants $C=C(\Lambda, \nu_0, n, s, \bar R)$, and $\xi=\xi(n, \beta, s)$, such that
\begin{equation}
\label{eq:SupEstEnergyEst}
\begin{aligned}
&\left(\int_{\sO} |\eta w^{\alpha}|^p y^{\beta-1}\,dx\,dy \right)^{1/p}\\
&\qquad
  \leq  (C\alpha)^{\xi+1} \left(\|\sqrt{y}\nabla\eta\|_{L^{\infty}(\HH)}^{2/p} + |\supp\eta|_{\beta-1} ^{1/p-1/2}\right)
  \left(\int_{\supp\eta} w^{2\alpha} y^{\beta-1} \,dx\,dy\right)^{1/2}.
\end{aligned}
\end{equation}
\end{claim}

\begin{proof}[Proof of Claim \ref{claim:SupEstEnergyEst}]
We fix $k\in \NN$. As in the proof of \cite[Theorem 8.15]{GilbargTrudinger}, we consider the functions $H_k :\RR\rightarrow[0,\infty)$,
\begin{equation}
\label{eq:DefinitionOfH}
\begin{aligned}
H_k(t) &:=
\begin{cases}
0, & t <A,\\
t^{\alpha}-A^{\alpha}, & A\leq t\leq k,\\
\alpha k^{\alpha-1}(t-k)+H_k(k), & t > k.
\end{cases}
\end{aligned}
\end{equation}
and
\begin{equation}
\label{eq:DefinitionOfG}
G_k(t) = \int_0^{t} |H'_k(s)|^2 ds.
\end{equation}
Then,
\begin{equation}
\label{eq:TestFunctionSupEst}
v=G_k(w)\eta^2
\end{equation}
is a valid test function in $H^1_0(\sO\cup\Gamma_0,\fw)$ in
%CP 1.23.2016: Replaced with the more general bilinear form
%\eqref{eq:HestonWithKillingBilinearForm}
%PF 2-27-2016: OK
\eqref{eq:Operator_A_bilinear_form}
by \cite[Lemma A.1]{Feehan_Pop_regularityweaksoln_v3}. Using the strict ellipticity of the operator $y^{-1}A$, together with the fact that $\nabla v = G'_k(w)\eta^2 \nabla w + 2 G_k(w) \eta \nabla \eta$ and $G_k(w)=0$ when $w \leq A$, we obtain as in the proof of \cite[Theorem 8.15]{GilbargTrudinger} that there is a positive constant, $C=C(\Lambda, n, \nu_0, \bar R)$, such that
\begin{equation}
\label{eq:SupEst1}
\begin{aligned}
\int_{\sO} |\nabla w|^2 \eta^2 G'_k(w) y^{\beta}\,dx\,dy
&\leq C \left[\int_{\sO}\eta^2\frac{|f|}{A} w^2 G'_k(w) y^{\beta-1}\,dx\,dy \right.\\
&\qquad  \left. + \int_{\sO}\left(\eta^2 +y|\nabla\eta|^2\right) w^2 G'_k(w) y^{\beta-1}\,dx\,dy  \right].
\end{aligned}
\end{equation}
H\"older's inequality applied to the conjugate pair $(s,s^*)$ gives
\begin{align*}
&\int_{\sO}\eta^2\frac{|f|}{A} w^2 G'_k(w) y^{\beta-1}\,dx\,dy \\
&\quad\leq \left(\int_{\hbox{supp }\eta}\frac{|f|^s}{A^s} y^{\beta-1}\,dx\,dy\right)^{1/s}
      \left(\int_{\sO} |\eta^2w^2 G'_k(w)|^{s^*} y^{\beta-1}\,dx\,dy\right)^{1/s^*},
\end{align*}
and thus, by definition \eqref{eq:DefAandW} of $A$,
\begin{equation}
\label{eq:SupEst2}
\begin{aligned}
\int_{\sO}\eta^2\frac{|f|}{A} w^2 G'_k(w) y^{\beta-1}\,dx\,dy
\leq \left(\int_{\sO}|\eta^2 w^2 G'_k(w)|^{s^*} y^{\beta-1}\,dx\,dy\right)^{1/s^*}.
\end{aligned}
\end{equation}
We need to justify first that the right-hand side in \eqref{eq:SupEst2} is finite. First, we notice that the following identities hold
\begin{equation}
\label{eq:SupEstIdentities}
\begin{aligned}
 |\nabla H_k(w)|^2 &=  |\nabla w|^2 |H'_k(w)|^2= |\nabla w|^2 G'_k(w),\\
 |wH'_k(w)|^2      &=  |w|^2G'_k(w),
\end{aligned}
\end{equation}
From the hypothesis $s>n+\beta$ in Theorems \ref{thm:MainSupremumEstimatesInterior} and  \ref{thm:MainSupremumEstimatesBoundary}, we observe that $2<2s^*<p$, so we may apply the interpolation inequality \cite[Inequality (7.10)]{GilbargTrudinger}. For any $\eps\in(0,1)$, we have
\begin{equation}
\label{eq:SupEstInterpolation}
\|\eta w H'_k(w)\|_{L^{2s^*}(\HH,y^{\beta-1})}
\leq
\eps \|\eta w H'_k(w)\|_{L^{p}(\HH,y^{\beta-1})}
+ \eps^{-\xi}\|\eta w H'_k(w)\|_{L^{2}(\HH,y^{\beta-1})},
\end{equation}
where
\begin{equation}
\label{eq:SupEstInterpolationXi}
\xi \equiv \xi(p,s) := \frac{p(s^*-1)}{p-2s^*}.
\end{equation}
We notice that $|H'_k(w)| \leq \alpha k^{\alpha-1}$ and $\eta w \in H^1(\sO, \fw)$ has compact support in $\bar B_{2R}(z_0)$. Therefore, we may apply Lemma \ref{lem:Extension} to build an extension $\hat w$ of $\eta w$ to a
rectangle $D$ containing $\bar B_{2R}(z_0)$. Lemma \ref{lem:AnalogSobolev}, shows that $\hat w \in L^p(D, y^{\beta-1})$, which implies that
\[
\|\eta w H'_k(w)\|_{L^{p}(\HH,y^{\beta-1})}  < \infty,
\]
and so, the right-hand side of \eqref{eq:SupEst2} is finite.

Inequalities \eqref{eq:SupEst1} and \eqref{eq:SupEst2}, together with the identities \eqref{eq:SupEstIdentities} yield
\begin{equation}
\label{eq:SupEst3}
\begin{aligned}
\int_{\sO} \eta^2 |\nabla H_k(w)|^2 y^{\beta}\,dx\,dy
&\leq
C  \left[\left(\int_{\sO}|\eta w H'_k(w)|^{2s^*} y^{\beta-1}\,dx\,dy\right)^{1/s^*} \right.\\
&\qquad \left. +\int_{\sO}\left(\eta^2+y|\nabla\eta|^2\right)|w H'_k(w)|^2  y^{\beta-1}\,dx\,dy \right].
\end{aligned}
\end{equation}
From Lemma \ref{lem:AnalogSobolev}, we obtain
\begin{equation}
\label{eq:SupEst4}
\begin{aligned}
\int_{\sO} |\eta H_k(w)|^p y^{\beta-1}\,dx\,dy
&\leq \left( \int_{\sO} \eta^2|H_k(w)|^2 y^{\beta-1}\,dx\,dy\right)^{(p-2)/2}
 \int_{\sO} |\nabla(\eta H_k(w))|^2 y^{\beta}\,dx\,dy\\
& \leq 2\left( \int_{\sO} \eta^2|H_k(w)|^2 y^{\beta-1}\,dx\,dy\right)^{(p-2)/2} \\
&\quad \times\left(\int_{\sO} |\nabla\eta|^2 |H_k(w)|^2 y^{\beta}\,dx\,dy +\eta^2 |\nabla H_k(w)|^2 y^{\beta}\,dx\,dy\right).
\end{aligned}
\end{equation}
Using $H_k(w) \leq w H'_k(w)$ and inequality \eqref{eq:SupEst3} in \eqref{eq:SupEst4}, we see that
\begin{equation}
\label{eq:SupEst5}
\begin{aligned}
&\int_{\sO} |\eta H_k(w)|^p y^{\beta-1}\,dx\,dy
\leq C \left[\left(1+\|\sqrt{y}\nabla\eta\|^2_{L^{\infty}(\HH)}\right)\left( \int_{\supp\eta}  |w H'_k(w)|^2 y^{\beta-1}\,dx\,dy\right)^{p/2} \right.\\
&\qquad \left. + \left( \int_{\sO}  |\eta w H'_k(w)|^2 y^{\beta-1}\,dx\,dy\right)^{(p-2)/2}
\left( \int_{\sO} |\eta w H'_k(w)|^{2s^*}  y^{\beta-1}\,dx\,dy\right)^{1/s^*} \right],
\end{aligned}
\end{equation}
where $C=C(\Lambda, n, \nu_0, \bar R)>0$. We rewrite the estimate for $\eta w H'_k(w)$ in \eqref{eq:SupEstInterpolation} in the form
\begin{equation*}
\begin{aligned}
\left(\int_{\sO} |\eta w H'_k(w)|^{2s^*}y^{\beta-1}\,dx\,dy\right)^{1/s^*}
&= \|\eta w H'_k(w)\|^2_{L^{2s^*}(\HH,y^{\beta-1})}\\
&\leq 2\eps^2\|\eta w H'_k(w)\|^2_{L^{p}(\HH,y^{\beta-1})} + 2\eps^{-2\xi}\|\eta w H'_k(w)\|^2_{L^{2}(\HH,y^{\beta-1})}.
\end{aligned}
\end{equation*}
Applying the preceding inequality in \eqref{eq:SupEst5}, we obtain
\begin{equation*}
\begin{aligned}
\|\eta H_k(w)\|^p_{L^p(\HH, y^{\beta-1})}
&\leq
C (1+\eps^{-2\xi})\left(1+\|\sqrt{y}\nabla\eta\|^2_{L^{\infty}(\HH)}\right) \|w H'_k(w)\|^p_{L^2(\supp\eta, y^{\beta-1})}\\
&\quad+ C \eps^2\| \eta w H'_k(w)\|^{p-2}_{L^2(\HH, y^{\beta-1})} \|\eta w H'_k(w)\|^2_{L^{p}(\HH,y^{\beta-1})}.
\end{aligned}
\end{equation*}
To estimate the last term in the preceding inequality, we apply Young's inequality with the conjugate pair of exponents, $\left(p/2, p/(p-2)\right)$, to give
\begin{equation}
\label{eq:SupEst6}
\begin{aligned}
\|\eta H_k(w)\|^p_{L^p(\HH, y^{\beta-1})}
& \leq
C\left(1+(\eps^2+\eps^{-2\xi})\right) \left(1+\|\sqrt{y}\nabla\eta\|^2_{L^{\infty}(\HH)}\right) \| w H'_k(w)\|^p_{L^2(\supp\eta, y^{\beta-1})}\\
& \quad+
C \eps^2  \| \eta w H'_k(w)\|^{p}_{L^p(\HH, y^{\beta-1})},
\end{aligned}
\end{equation}
Employing the definition \eqref{eq:DefinitionOfH} of $H_k(w)$ gives $ 0 \leq w H'_k(w) \leq \alpha H_k(w) + \alpha A^{\alpha}$, and so
\begin{align*}
\int_{\sO} |\eta w H'_k(w)|^p y^{\beta-1}\,dx\,dy
&\leq |2\alpha|^p \left[\int_{\sO} |\eta H_k(w)|^p y^{\beta-1}\,dx\,dy
   + |\supp \eta|_{\beta-1} A^{\alpha p}\right],
\end{align*}
and thus, applying inequality \eqref{eq:SupEst6} yields
\begin{equation*}
\begin{aligned}
\int_{\sO} |\eta  H_k(w)|^p y^{\beta-1}\,dx\,dy
&\leq
C\left(1+\left(\eps^2+\eps^{-2\xi}\right)\right) \left(1+\|\sqrt{y}\nabla\eta\|^2_{L^{\infty}(\HH)}\right) \| wH'_k(w)\|^p_{L^2(\supp\eta, y^{\beta-1})}\\
& \quad+ C|2\alpha|^p \eps^2 \left( \| \eta H_k(w)\|^{p}_{L^p( y^{\HH,\beta-1})} + |\supp \eta|_{\beta-1} A^{\alpha p}\right).
\end{aligned}
\end{equation*}
By choosing $\eps =1/(2\sqrt{C(2\alpha)^p})$ and taking $p$-th order roots, we obtain
\begin{align*}
&\left(\int_{\sO} |\eta H_k(w)|^p y^{\beta-1}\,dx\,dy \right)^{1/p}\\
&\quad\leq  (C\alpha)^{\xi} \left(\left(1+\|\sqrt{y}\nabla\eta\|^2_{L^{\infty}(\HH)}\right)^{1/p}
  \left(\int_{\supp\eta} |wH'_k(w)|^2 y^{\beta-1} dx\,dy\right)^{1/2}
 +|\supp\eta|_{\beta-1}^{1/p} A^{\alpha}\right).
\end{align*}
Because the positive constants $C$ and $\xi$ are independent of $k$, we may take limit as $k$ goes to $\infty$, in the preceding inequality, and we obtain
\begin{equation*}
\begin{aligned}
\left(\int_{\sO} |\eta w^{\alpha}|^p y^{\beta-1}\,dx\,dy \right)^{1/p}
&\leq  (C\alpha)^{\xi+1} \left(\left(1+\|\sqrt{y}\nabla\eta\|^2_{L^{\infty}(\HH)}\right)^{1/p}
  \left(\int_{\supp\eta} |w|^{2\alpha} y^{\beta-1} \,dx\,dy\right)^{1/2}\right.\\
&\quad  +\left.|\supp\eta|_{\beta-1}^{1/p} A^{\alpha}\right),
\end{aligned}
\end{equation*}
We also have
\begin{equation*}
\begin{aligned}
A^{\alpha}
&\leq \left(\frac{1}{|\supp\eta|_{\beta-1}} \int_{\supp\eta} w^{2\alpha} y^{\beta-1}\,dx\,dy\right)^{1/2}.
\end{aligned}
\end{equation*}
Combining the last two inequalities gives \eqref{eq:SupEstEnergyEst}. This completes the proof of Claim \ref{claim:SupEstEnergyEst}.
\end{proof}
This completes Step \ref{step:EnergyEstimates}.
\end{step}

\begin{step} [Moser iteration]
\label{step:MoserIterations}
The purpose of this step is to apply the Moser iteration technique to $w$ in \eqref{eq:Defnw} with a suitable choice of $\alpha \geq 1$ and of a sequence of non-negative cutoff functions, $\{\eta_N\}_{N\geq 1} \subset C^1_0(\bar\HH)$, with support in $\bar\BB_{2R}(z_0)$. We choose $\{\eta_N\}_{N\in\NN}$ as in \eqref{eq:GlobalDefinitionCutOffFunction} with $R_N:=R\left(1+1/(N+1)\right)$. Then, \eqref{eq:RangeCutOffFunction} and \eqref{eq:GlobalBoundCutOffFunction} become
\begin{equation}
\label{eq:SupEstCutOffFunction}
\begin{aligned}
 \eta_N|_{B_{R_N}(z_0)}\equiv 1, \quad \eta_N|_{B_{R_{N-1}}^c(z_0)}\equiv 0, \quad \left|\nabla \eta_N\right| \leq \frac{cN^3}{R^2},
\end{aligned}
\end{equation}
where $c$ is a positive constant independent of $R$ and $N$.
For each $N \geq 0$, we set $p_N:=2(p/2)^N$ and $\alpha_N:=(p/2)^N$. Let $A_N:=\|f\|_{L^s(\supp\eta_N, y^{\beta-1})}$ and $w_N:=u^++A_N$ or $w_N:=u^-+A_N$. Define
\[
I(N) := \left(\int_{B_{R_N}(z_0)} |w_N|^{p_{N}} y^{\beta-1}\,dx\, dy\right)^{1/p_{N}}.
\]
Applying the energy estimate \eqref{eq:SupEstEnergyEst} with $w=w_N$, $\alpha=\alpha_{N-1}$, and $\eta=\eta_N$, we obtain for all $N\geq 1$ that
\begin{equation}
\label{eq:SupEst8}
\begin{aligned}
I(N) \leq C_0(R,N) I(N-1),
\end{aligned}
\end{equation}
where we denote
\begin{equation}
\label{eq:SupEst9}
\begin{aligned}
C_0(R,N):=\left(C|\alpha_{N-1}|\right)^{2(\xi+1)/p_{N-1}}\left(\|\sqrt{y}\nabla\eta_N\|_{L^{\infty}(\HH)}^{2/p} + |\supp\eta_N|_{\beta-1} ^{1/p-1/2}\right)^{2/p_{N-1}},
\end{aligned}
\end{equation}
and $C =C(\Lambda, n, \nu_0, \bar R)$. By applying \eqref{eq:DomainCondition_interior} and \eqref{eq:MeasureBalls}, there is a constant $c>0$ such that
\begin{equation}
\label{eq:BallsEst}
\begin{aligned}
c^{-1} R^{4/(p-2)} \leq |B_{2R}(z_0)|_{\beta-1} \leq c R^{4/(p-2)},\quad\forall\, R \in (0,\bar R],
\end{aligned}
\end{equation}
where we used the fact that $2(n+\beta-1) = 4/(p-2)$ by \eqref{eq:Defnp};
the positive constant $c$ depends only on $n$ and $\beta$ in the case of Theorem \ref{thm:MainSupremumEstimatesInterior}, and on  $n$, $\beta$ and $K$, in the case of Theorem \ref{thm:MainSupremumEstimatesBoundary}. Moreover, by \eqref{eq:SimpleCycloidBallInsideEuclid} we know that $0\leq y \leq 2 R^2$ on $B_R(z_0)$, for all $R\geq 0$. Consequently,
we have
\[
\|\sqrt{y}\nabla\eta_N\|_{L^{\infty}(\HH)}^{2/p} + |\supp\eta_N|_{\beta-1} ^{1/p-1/2} \leq c N^{6/p} R^{-2/p},
\]
and so, using \eqref{eq:BallsEst}, we obtain
\begin{equation*}
\begin{aligned}
\prod_{N \geq 1} C_0(R,N)
&\leq C_1 |B_{2R}(z_0)|_{\beta-1}^{-1/2},
\end{aligned}
\end{equation*}
where $C_1=C_1(\Lambda, n, \nu_0, \bar R, s)$. In the case of Theorem \ref{thm:MainSupremumEstimatesBoundary}, the constant $C_1$ depends in addition on $K$. By iterating \eqref{eq:SupEst8}, we obtain, after using \cite[Theorem 2.8]{Adams_1975},
\begin{equation}
\label{eq:SupEst10}
\begin{aligned}
\esssup_{B_R(z_0)} w = I(+\infty)
&\leq C_1 \left(\frac{1}{|B_{2R}(z_0)|_{\beta-1}}\int_{B_{2R}(z_0)} |w|^2 y^{\beta-1}\,dx \,dy\right)^{1/2}.
\end{aligned}
\end{equation}
Applying \eqref{eq:SupEst10} to $w$ as in \eqref{eq:Defnw} yields
\begin{equation}
\label{eq:MainSupremumEstimates_old}
\esssup_{B_R(z_0)} u^+(u^-)
\leq  C\left(|B_{2R}(z_0)|_{\beta-1}^{-1/2}\|u^+(u^-)\|_{L^2(B_{2R}(z_0),y^{\beta-1})} + \|f\|_{L^s(B_{2R}(z_0),y^{\beta-1})}\right),
\end{equation}
for all $0<R<\bar R/2$, where $C=C(\Lambda, n, \nu_0, \bar R, s)$. In the case of Theorem \ref{thm:MainSupremumEstimatesBoundary}, the constant $C_1$ depends in addition on $K$. This completes Step \ref{step:MoserIterations}.
\end{step}

\begin{step}[Completion of the proof of Theorem \ref{thm:MainSupremumEstimatesInterior}]
\label{step:CompletionSupremumEstimatesInterior}
Recall that we have chosen $\bar R$ so that $R_0 > 2\bar R^2$ (we see by \eqref{eq:SimpleCycloidBallInsideEuclid} that this implies $B_{\bar R}(z_0) \subset E_{R_0}(z_0)$). For any $R>0$, we have by \eqref{eq:EuclidBallInsideCycloid} that $E_R(z_0) \subset B_{\sqrt{R}}(z_0)$. Therefore, using \eqref{eq:EuclidBallInsideCycloid}, \eqref{eq:SimpleCycloidBallInsideEuclid} and \eqref{eq:MainSupremumEstimates_old} we obtain, for all $R>0$ obeying $2\sqrt{R} < \bar R$ or, equivalently, $R < R_0/8$,
\begin{align*}
\esssup_{E_R(z_0)} u^+(u^-)
&\leq  C\left(\|u^+(u^-)\|_{L^2(E_{R_0}(z_0),y^{\beta-1})} + \|f\|_{L^s(E_{R_0}(z_0),y^{\beta-1})}\right),
\end{align*}
where $C=C(\Lambda, n, \nu_0, R_0, s)$. We obtain the desired inequality \eqref{eq:MainSupremumEstimates2} by choosing $R_1 < R_0/8$ and setting $R=R_1$ in the preceding last inequality. This completes Step \ref{step:CompletionSupremumEstimatesInterior} and the proof of Theorem \ref{thm:MainSupremumEstimatesInterior}.
\end{step}

\begin{step}[Completion of the proof of Theorem \ref{thm:MainSupremumEstimatesBoundary}]
\label{step:CompletionSupremumEstimatesBoundary}
The proof of Theorem \ref{thm:MainSupremumEstimatesBoundary} follows
%COMMENT: At the beginning of the proof we choose \bar R, which in the case of the boundary sup estimates depends also on K (from Lemma 4.1). This is where from the dependency on K comes in the case of the boundary sup estimates.
exactly in the same way as the proof of Theorem \ref{thm:MainSupremumEstimatesInterior}, with the only observation that all constants now also depend on the cone, $K$.
(The dependence on $K$ is due to the choice of $\bar R$ via Lemma \ref{lem:RegularDomain} at the start of the proof.) This completes Step \ref{step:CompletionSupremumEstimatesBoundary} and the proof of Theorem \ref{thm:MainSupremumEstimatesBoundary}.
\end{step}
This concludes the proofs of Theorems \ref{thm:MainSupremumEstimatesInterior} and \ref{thm:MainSupremumEstimatesBoundary}.
\end{proof}

We now complete the

\begin{proof}[Proof of Corollary \ref{cor:MainSupremumEstimatesBoundary}]
Theorem \ref{thm:MainSupremumEstimatesBoundary} can be extended to the case of non-zero Dirichlet boundary condition given by a function $g \in H^1(\sO, \fw) \cap L^{\infty}_{\loc}(\bar\Gamma_1)$, in the sense that
$$
u-g \in H^1_0(\sO\cup\Gamma_0,\fw),
$$
with the aide of the following modifications to the proof of Theorem \ref{thm:MainSupremumEstimatesBoundary}. Let
\begin{align*}
M := \esssup_{\Gamma_1\cap B_{2R}(z_0)} g
\quad\hbox{and}\quad
m :=\essinf_{\Gamma_1\cap B_{2R}(z_0)} g,
\end{align*}
and replace the definitions of the functions $u^+$ and $u^-$ (the positive and negative part of the variational
subsolution and supersolution, respectively) by
$$
u^M(z):=(u(z) \vee M)^+\quad\hbox{and}\quad u^m(z):=(u(z) \wedge m)^-\quad\hbox{for a.e. } z \in B_{2R}(z_0).
$$
We also need to redefine the function $H_k$ in \eqref{eq:DefinitionOfH} by
\begin{equation*}
\begin{aligned}
H_k(t) &:=
\begin{cases}
0, & t <A+|M|,\\
t^{\alpha}-(A+|M|)^{\alpha}, & A+|M|\leq t\leq k,
\\
\alpha k^{\alpha-1}(t-k)+H_k(k), & t > k,
\end{cases}
\end{aligned}
\end{equation*}
when we apply Step \ref{step:EnergyEstimates} in the proof of Theorem \ref{thm:MainSupremumEstimatesBoundary} to the function
$w=u^M+A$ (when $u$ is assumed to be a subsolution), and by
\begin{equation*}
\begin{aligned}
H_k(t) &:=
\begin{cases}
0, & t <A+|m|,\\
t^{\alpha}-(A+|m|)^{\alpha}, & A+|m|\leq t\leq k,
\\
\alpha k^{\alpha-1}(t-k)+H_k(k), & t > k,
\end{cases}
\end{aligned}
\end{equation*}
when we apply the same step to $w=u^m+A$
(when $u$ is assumed to be a supersolution). Then, the argument used in the proof of Theorem \ref{thm:MainSupremumEstimatesBoundary} to obtain \eqref{eq:SupEst10} now yields
\begin{equation*}
\begin{aligned}
\esssup_{B_R(z_0)} u^{M}
&\leq  C_1 \left[\left( \frac{1}{|B_{2R}(z_0)|_{\beta-1}}\int_{B_{2R}(z_0)} |u^{M}|^2 y^{\beta-1}\,dx\,dy\right)^{1/2} +\|f\|_{L^s(B_{2R}(z_0), y^{\beta-1})}\right],\\
\esssup_{B_R(z_0)} u^{m}
&\leq  C_1 \left[\left( \frac{1}{|B_{2R}(z_0)|_{\beta-1}}\int_{B_{2R}(z_0)} |u^{m}|^2 y^{\beta-1}\,dx\,dy\right)^{1/2} +\|f\|_{L^s(B_{2R}(z_0), y^{\beta-1})}\right],
\end{aligned}
\end{equation*}
when $u$ is assumed a subsolution and supersolution, respectively.
%CP 1.11.2016: Updated the sentence below.
%PF 2-27-2016: OK
The preceding estimates imply \eqref{eq:MainSupremumEstimates3} and the statement in Remark \ref{rmk:MainSupremumEstimatesBoundary}, just as estimate \eqref{eq:MainSupremumEstimates_old} implies \eqref{eq:MainSupremumEstimates2} in Step
\ref{step:CompletionSupremumEstimatesBoundary} of the proof of Theorem \ref{thm:MainSupremumEstimatesBoundary}.
\end{proof}

\section{H\"older continuity for solutions to the variational equation}
\label{sec:HolderContinuityVariationalEquation}
In this section, we prove Theorems \ref{thm:MainContinuityInterior} and \ref{thm:MainContinuityBoundary} and
Corollaries \ref{cor:MainHolderContinuityBoundary} and \ref{cor:MainContinuity}, that is, local H\"older continuity on a neighborhood of $\bar\Gamma_0$ for solutions $u$ to the variational equation \eqref{eq:IntroHestonWeakMixedProblemHomogeneous}. We
consider separately the case of the interior boundary points $z_0 \in \Gamma_0$ and of the `corner points' $z_0 \in \bar\Gamma_0\cap\bar\Gamma_1$. (While $\bar\Gamma_0\cap\bar\Gamma_1$ is a set of geometric corner points for the
open subset, $\sO$, the lesson of \cite{DaskalHamilton1998} is that the solution, $u$, along $\Gamma_0$ behaves, in many respects, just as it does in the interior of $\sO$.) The proof of the second case, for corner points, is easier
than the proof of the first case as it does not require an application of the John-Nirenberg inequality. The essential difference between the proofs of Theorems \ref{thm:MainContinuityInterior} and \ref{thm:MainContinuityBoundary} and the proof of its classical analogue for variational solutions to non-degenerate elliptic equations \cite[Theorems 8.27 and 8.29]{GilbargTrudinger} consists in a modification of the methods of \cite[\S 8.6, \S 8.9, and \S 8.10]{GilbargTrudinger} when deriving our energy estimates
\eqref{eq:EnergyEstimate}, where we adapt the application of the John-Nirenberg inequality and Poincar\'e inequality to our framework of weighted Sobolev spaces. Moreover, because the balls defined by the Koch metric, $d$, do not
have good scaling properties unless they are centered at a point $z_0 \in \partial\HH$ (see Remark \ref{rmk:ScalingPropertiesKochMetric}), the Moser iteration technique applies only to such balls. Therefore, the estimate
\eqref{eq:MainContinuity2} holds only for points $z_0 \in \partial\HH$, and in order to obtain the full H\"older continuity of solutions \eqref{eq:MainContinuity3}, we need to apply a rescaling argument which is outlined in the
last steps of the arguments below. Therefore, boundary H\"older continuity does not follow in the same way as in \cite{GilbargTrudinger}. We also prove Theorem \ref{thm:StrongMaximumPrinciple}.

We now proceed to the proofs of Theorems \ref{thm:MainContinuityInterior} and \ref{thm:MainContinuityBoundary}, first in \S \ref{subsec:HolderContinuityInteriorDegenBdry} for the case of points $z_0\in \Gamma_0$ and then in \S \ref{subsec:HolderContinuityCornerPointsDegenBdry} for points $z_0\in \bar\Gamma_0\cap\bar\Gamma_1$. The
proofs of Corollaries \ref{cor:MainHolderContinuityBoundary} and \ref{cor:MainContinuity} can be found in \S \ref{subsec:HolderContinuityCornerPointsDegenBdry}.

\subsection{Local H\"older continuity in the interior the degenerate boundary}
\label{subsec:HolderContinuityInteriorDegenBdry}
In this subsection, we prove Theorem \ref{thm:MainContinuityInterior}. Let $z_0 \in \Gamma_0$ and $R_0>0$ be as in the hypotheses of Theorem \ref{thm:MainContinuityInterior}, and let $\bar R$ be small enough such that
\begin{equation}
\label{eq:ChoicebarR}
B_{\bar R}(z_0) \subset E_{R_0}(z_0),
\end{equation}
and for all $z_i=(x_i,y_i)\in B_{\bar R}(z_0)$, $i=1,2$, we have
\begin{equation}
\label{eq:ContinuityAssumption2}
0< y_1 <1, \quad 0< y_2 <1, \quad 0 \leq |z_1-z_2| <1, \quad\hbox{and}\quad 0 \leq d(z_1,z_2) <1.
\end{equation}
For $z_0\in\bar\sO$ and $0<R<\bar R$, we denote
\begin{align}
\label{defn:Esssup_function}
M_R &:= \esssup_{B_R(z_0)} u,
\\
\label{defn:Essinf_function}
m_R &:= \essinf_{B_R(z_0)} u,
\end{align}
and we let
$$
\osc_{B_R(z_0)}u:=M_R-m_R
$$
denote the oscillation of $u$ over the ball $B_R(z_0)$. From Theorem \ref{thm:MainSupremumEstimatesInterior}, we know that $M_R$ and $m_R$ are finite quantities and $\osc_{B_R(z_0)} u$ is well-defined. Before proceeding to the proof of Theorem \ref{thm:MainContinuityInterior}, we first establish the

\begin{thm}[Oscillation estimate]
\label{thm:OscillationEstimate}
There is a positive constant, $C$, depending at most on $\Lambda$, $\nu_0$, $R_0$, $n$, $s$, and a constant $\alpha_0\in (0,1)$, depending at most on $s$, $n$ and $\beta$, such that the following holds. For all $R$ such that $0<4R\leq \bar{R}$, we have
\begin{equation}
\label{eq:MainContinuity2}
\osc_{B_R(z_0)} u \leq C\left(\|f\|_{L^s(E_{R_0}(z_0),y^{\beta-1})} + \|u\|_{L^2(E_{R_0}(z_0), y^{\beta-1})}\right) R^{\alpha_0}.
\end{equation}
\end{thm}

\begin{proof}
We choose
\begin{align}
\label{eq:Choiceq}
q\in (n+\beta,s),
\\
\label{eq:Choicedelta}
\omega \in (0,2(n+\beta-1)/q),
\end{align}
and define $k(R)>0$ by
\begin{equation}
\label{eq:DefinitionK}
\begin{aligned}
k\equiv k(R) &:= \|f\|_{L^q(B_{4R}(z_0),y^{\beta-1})}+\left(|m_{\bar{R}}|+|M_{\bar{R}}|\right) R^{\omega}.
\end{aligned}
\end{equation}
The remaining steps in the proof will apply to either of the following choices of functions $w$ defined on $B_{4R}(z_0)$,
\begin{equation}
\label{eq:Two_choices_for_w}
 w = u - m_{4R} + k(R) \quad\hbox{or}\quad w = M_{4R} - u + k(R),
\end{equation}
but, for concreteness, we choose
\begin{equation}
\label{eq:DefinitionW}
\begin{aligned}
 w = u - m_{4R} + k(R).
\end{aligned}
\end{equation}
If $m_{\bar R}=M_{\bar R}=0$ or $m_{4R}=M_{4R}=0$, then we automatically have $u=0$ on $B_{4R}(z_0)$ and \eqref{eq:MainContinuity2}
holds on $B_{4R}(z_0)$.
%COMMENT: 4R can be replaced by any radius larger than R. But we do energy estimates for balls with radii from R to 2R, and for the John-Nirenberg inequality we move from radii 2R to 3R. But I wanted to make sure that we want have troubles with enlarging the balls, so I defined everything on balls of radius 4R.
Therefore,  without loss of generality, we may assume
\begin{equation}
\label{eq:Assump_max}
m_{4R}\neq 0 \quad\hbox{or}\quad M_{4R} \neq 0,
\end{equation}
and $m_{\bar R} \neq 0$ or $M_{\bar R} \neq 0$. The last assumption implies that
\begin{equation}
\label{eq:Assump_k}
k(R)\neq 0,
\end{equation}
by \eqref{eq:DefinitionK}. Therefore, we notice that both choices of $w$ in \eqref{eq:DefinitionW} are bounded, positive functions.

\setcounter{step}{0}
\begin{step} [Energy estimate for $w$]
\label{step:ContinuityVarEqEnergyEstw}
Let $\eta \in C^1_0(\bar{\HH})$ be a non-negative cutoff function with  $\supp\eta \subseteqq \bar B_{4R}(z_0)$. For any $\alpha \in \RR$ with $\alpha \neq -1$, let
\begin{equation}
\label{eq:Defnv}
v := \eta^2 w^{\alpha}.
\end{equation}
Then, $v$ is a valid test function in $H^1_0(\sO \cup\Gamma_0, \fw)$
by \cite[Lemma A.2]{Feehan_Pop_regularityweaksoln_v3}. Let
\begin{equation}
\label{eq:DefnHw}
H(w) := w^{(\alpha+1)/2},
\end{equation}
and notice that Theorem \ref{thm:MainSupremumEstimatesInterior} implies that $H(w)$ is a positive, bounded function, so the following operations are justified.
The goal in this step is to prove

\begin{claim}[Energy estimate]
\label{claim:InteriorContEnergyEst}
There are positive constants, $C=C(\Lambda, \nu_0, n, \bar R)$ and $\xi=\xi(n, \beta, q)$, such that
\begin{equation}
\label{eq:EnergyEstimate}
\begin{aligned}
\| \eta H(w) \|_{L^{p}\left(\HH,y^{\beta-1}\right)}
&\leq
C_0(R,\alpha) \|H(w)\|_{L^2\left(\supp\eta, y^{\beta-1}\right)},
\end{aligned}
\end{equation}
where the constant $C_0(R,\alpha)$ is defined by
\begin{equation}
\label{eq:EnergyEstimateConstant}
\begin{aligned}
C_0(R,\alpha) := \left(C|1+\alpha|\right)^{(\xi+1)/p} \left(1+ \|\sqrt{y}\nabla \eta\|_{L^{\infty}(\HH)}^2\right)^{1/p},
\end{aligned}
\end{equation}
and the constant $\xi$ is given by
\begin{equation}
\label{eq:EnergyEstimateConstantXiP}
\begin{aligned}
\xi \equiv \xi(p,q) &:= \frac{p(q^*-1)}{p-2q^*},
\end{aligned}
\end{equation}
where $q^*$ is the conjugate exponent for $q$ in \eqref{eq:Choiceq}, that is, $1/q+1/q^*=1$.
\end{claim}

The estimate \eqref{eq:EnergyEstimate} will be used in Moser iteration.

\begin{proof}[Proof of Claim \ref{claim:InteriorContEnergyEst}]
Notice that estimate \eqref{eq:EnergyEstimate} is similar to \eqref{eq:SupEstEnergyEst}. The proofs of the two estimates are also very similar and we only outline the differences.

Substituting the choice \eqref{eq:Defnv} of $v$ in
%CP 1.23.2016: Replaced with the more general bilinear form
%\eqref{eq:HestonWithKillingBilinearForm}
%PF 2-27-2016: OK
\eqref{eq:Operator_A_bilinear_form}, using $\nabla v = \alpha \eta^2 w^{\alpha-1} \nabla w + 2\eta\nabla\eta w^{\alpha}$ together with $\nabla H(w) = \frac{\alpha+1}{2} w^{(\alpha-1)/2} \nabla w$ (see \eqref{eq:DefnHw}) and $w \geq k$ (by \eqref{eq:DefinitionW}), gives
%CP 1.23.2016: In the inequalities below I replaced r (the zeroth order term for the Heston operator) with c (the zeroth order term for the more general operator)
%PF 2-27-2016: OK
\begin{equation}
\label{eq:Eq1}
\begin{aligned}
\int_{\HH} \eta^2 |\nabla H(w)|^2 y^{\beta} \,dx\,dy
&\leq C|1+\alpha|
\left[\int_{\HH} \left(\eta^2+y|\nabla\eta|^2\right) w^{\alpha+1} y^{\beta-1} \,dx\,dy \right.\\
&\qquad \left.+ \int_{\HH} \eta^2 \frac{|f+c(k-m_{4R})|}{k} w^{\alpha+1} y^{\beta-1} \,dx\,dy\right],
\end{aligned}
\end{equation}
where $C=C(\Lambda, \nu_0,\bar R)$. By H\"older's inequality, we have
\begin{equation}
\label{eq:Eq2}
\begin{aligned}
\int_{\HH} \eta^2 \frac{|f+c(k-m_{4R})|}{k} w^{\alpha+1} y^{\beta-1} \,dx\,dy
&\leq
\left(\int_{\supp\eta}\left|\frac{f+c(k-m_{4R})}{k}\right|^{q} y^{\beta-1}\, dx\,dy\right)^{1/q}\\
&\qquad \times
\left(\int_{\HH} \left|\eta w^{(\alpha+1)/2}\right|^{2q^*} y^{\beta-1} \, dx\,dy\right)^{1/q^*}.
\end{aligned}
\end{equation}
From our definition of $k$ in \eqref{eq:DefinitionK}, the choice of $\omega$ in \eqref{eq:Choicedelta} and \eqref{eq:MeasureBalls}, we see that
\begin{equation*}
\left(\int_{\supp\eta}\left|\frac{f+c(k-m_{4R})}{k}\right|^{q} y^{\beta-1} \, dx\,dy\right)^{1/q}
\leq
1+c + \bar R^{2(n+\beta-1)/q-\omega}
\end{equation*}
and so, because $\omega$ was chosen such that $\omega<2(n+\beta-1)/q$ in \eqref{eq:Choicedelta}, there is a positive constant, $C=C(\Lambda, \bar R)$, such that
\begin{equation}
\label{eq:Eq3}
\begin{aligned}
\left(\int_{\supp\eta}\left|\frac{f+c(k-m_{4R})}{k}\right|^{q} y^{\beta-1} \, dx\,dy\right)^{1/q}
&\leq
C.
\end{aligned}
\end{equation}
From inequalities \eqref{eq:Eq1}, \eqref{eq:Eq2} and \eqref{eq:Eq3}, we obtain
\begin{equation}
\label{eq:Eq4}
\begin{aligned}
\int_{\HH} \eta^2 |\nabla H(w)|^2 y^{\beta} \,dx\,dy
&\leq C|1+\alpha|
\left[\int_{\HH} \left(\eta^2+y|\nabla\eta|^2\right) w^{\alpha+1} y^{\beta-1} \,dx\,dy \right.\\
&\qquad \left. + \left(\int_{\HH} \left|\eta w^{(\alpha+1)/2}\right|^{2q^*}y^{\beta-1}\, dx\,dy\right)^{1/q^*} \right],
\end{aligned}
\end{equation}
where $C=C(\Lambda, \nu_0, \bar R)$.

Now, we can follow the argument used in the proof of estimate \eqref{eq:SupEstEnergyEst}. We first apply Lemma \ref{lem:AnalogSobolev} to $\eta H(w)$ which we combine with \eqref{eq:Eq4} to obtain
\begin{equation}
\label{eq:Eq5}
\begin{aligned}
{}&\int_{\HH} |\eta H(w)|^p y^{\beta-1}\,dx\,dy\\
&\leq
C|1+\alpha| \left(1+\|\sqrt{y}\nabla\eta\|^2_{L^{\infty}(\HH)}\right)\left(\int_{\supp\eta}
|H(w)|^2 y^{\beta-1}\,dx\,dy\right)^{p/2}\\
&\quad +
C|1+\alpha| \left( \int_{\HH} \eta^2|H(w)|^2 y^{\beta-1}\,dx\,dy\right)^{(p-2)/2}
\left(\int_{\HH} |\eta H(w)|^{2q^*}y^{\beta-1}\,dx\,dy\right)^{1/q^*}.
\end{aligned}
\end{equation}
Next, using the fact that $2<2q^*<p$ (by \eqref{eq:Choiceq}), we apply the interpolation inequality \cite[Inequality (7.10)]{GilbargTrudinger}, for any $\eps>0$, to give
\[
\|\eta H(w)\|_{L^{2q^*}(\HH,y^{\beta-1})} \leq \eps \|\eta H(w)\|_{L^{p}(\HH,y^{\beta-1})} + \eps^{-\xi}\|\eta H(w)\|_{L^{2}(\HH,y^{\beta-1})},
\]
where $\xi$ is given by \eqref{eq:EnergyEstimateConstantXiP}. Applying the preceding inequality in \eqref{eq:Eq5}, we obtain
\begin{equation*}
\begin{aligned}
&\|\eta H(w)\|^p_{L^p(\HH, y^{\beta-1})}\\
&\qquad\leq
C|1+\alpha|\left(1+\eps^{-2\xi}\right) \left(1+\|\sqrt{y}\nabla\eta\|^2_{L^{\infty}(\HH)}\right) \| H(w)\|^p_{L^2(\supp\eta, y^{\beta-1})}\\
&\qquad \quad+
C|1+\alpha| \eps^2 \| \eta H(w)\|^2_{L^p(\HH, y^{\beta-1})} \| \eta H(w)\|^{p-2}_{L^2(\HH, y^{\beta-1})}.
\end{aligned}
\end{equation*}
To bound the last term in the preceding inequality, we apply Young's inequality with the conjugate exponents $\left(p/2, p/(p-2)\right)$. By choosing $\eps = 1/\left( 2C|1+\alpha|\right)^{1/2}$ and taking roots of order $p$, we obtain \eqref{eq:EnergyEstimate} and \eqref{eq:EnergyEstimateConstant}. This concludes the proof of Claim \ref{claim:InteriorContEnergyEst}.
\end{proof}
This concludes Step \ref{step:ContinuityVarEqEnergyEstw}.
\end{step}

\begin{step}[Moser iteration with negative power]
\label{step:ContinuityVarEqEnergyMoserNegPower}
In this step we apply the Moser iteration technique starting with a suitable $\alpha=\alpha_0<-1$ in \eqref{eq:EnergyEstimate} to functions $w$ as in \eqref{eq:Defnw}. Let $\{\eta_N\}_{N\in\NN}$ be the sequence of cutoff functions
considered in Step \ref{step:MoserIterations} in the proof of Theorem
\ref{thm:MainSupremumEstimatesInterior}. Let $\alpha_0<-1$, $p_0:=\alpha_0+1$, $p_N:=p_0(p/2)^N$, where $p$ is as in \eqref{eq:Defnp}, and $\alpha_N+1:=p_N$. We notice
that $p_N \rightarrow -\infty$ as $N$ increases. Set
\[
I(N) := \left(\int_{B_{R_N}(z_0)} |w|^{p_{N}} y^{\beta-1}\,dx\, dy\right)^{1/p_{N}}.
\]
Applying an argument very similar to that in Step \ref{step:MoserIterations} of the proof of Theorem
\ref{thm:MainSupremumEstimatesInterior}, with the aid of \eqref{eq:EnergyEstimate} instead of \eqref{eq:SupEstEnergyEst}, we find that
\begin{equation}
\label{eq:Eq6}
\begin{aligned}
I(N) \geq C_1(R,N) I(N-1),
\end{aligned}
\end{equation}
where $C_1(R,N)$ is given by
\begin{equation}
\label{eq:Eq7}
\begin{aligned}
C_1(R,N)=\left(C|p_{N-1}|N^6\right)^{(\xi+1)/p_{N}} R^{-2/p_{N}},
\end{aligned}
\end{equation}
and $C=C(\Lambda, \nu_0, \bar R)$ is a positive constant, independent of $R$ and $N$. Using \eqref{eq:BallsEst}, we obtain
\[
\prod_{N\geq 1} C_1(R,N) \geq C_2 |B_{2R}(z_0)|_{\beta-1}^{1/|p_0|},
\]
where $C_2=C_2(\Lambda, \nu_0, \bar R, q)$. By iterating \eqref{eq:Eq6}, we obtain $I(-\infty)\geq I(0) \prod_{N\geq 1} C_0(R,N)$, which gives us
\begin{equation}
\label{eq:Eq8}
\begin{aligned}
\essinf_{B_R(z_0)} w = I(-\infty)
&\geq C_2 \left(\frac{1}{|B_{2R}(z_0)|_{\beta-1}}\int_{B_{2R}(z_0)} |w|^{p_0} y^{\beta-1}\,dx\, dy\right)^{1/p_0}.
\end{aligned}
\end{equation}
This concludes Step \ref{step:ContinuityVarEqEnergyMoserNegPower}.
\end{step}

\begin{step}[Application of Theorem \ref{thm:AbstractJohnNirenberg}]
\label{step:LocalHolderInequalityApplicationAbstractJohnNirenberg}
The purpose of this step is to show that we may apply Theorem \ref{thm:AbstractJohnNirenberg} to $w$ with $S_r=B_{(2+r)R}(z_0)$, $0\leq r \leq 1$, and $\theta_0=\theta_1=1$. By Proposition
\ref{prop:ApplicationAbstractJohnNirenberg}, we find that $w$ satisfies the inequalities \eqref{eq:AbstractJohnNirenberg1}, so it remains to show that \eqref{eq:AbstractJohnNirenberg2} holds for $\log w$. For $A$ as defined in
\eqref{eq:AbstractJohnNirenberg2} and $S_r=\BB_{(2+r)R}(z_0) = B_{(2+r)R}(z_0)$,
%PF 2-23-2016 I don't understand the qualifier below -- I assume a typo
%writing $B_{(2+r)R}(z_0)$ in place of $B_{(2+r)R}(z_0)$ for brevity,
%CP 2.24.2016: Yes, a typo. We should remove what was commented.
%PF 2-26-2016: OK
we have by H\"older's inequality that
\begin{equation*}
A
\leq \sup_{0 \leq r \leq 1} \inf_{c \in \RR}\left(\frac{1}{|B_{(2+r)R}(z_0)|_{\beta-1}} \int_{B_{(2+r)R}(z_0)} |\log w-c|^2 y^{\beta-1}\,dx\,dy\right)^{1/2},
\end{equation*}
and so, Corollary \ref{cor:PoincareInequality} gives us
\begin{equation}
\label{eq:BoundForA}
\begin{aligned}
A&\leq \sup_{0 \leq r \leq 1} \left((2+r)R\right)^2 \left(\frac{1}{|B_{(2+r)R}(z_0)|_{\beta}} \int_{B_{(2+r)R}(z_0)} |\nabla\log w|^2 y^{\beta}\,dx\,dy\right)^{1/2}.
\end{aligned}
\end{equation}
Let $\eta \in C^1_0(\bar{\HH})$ be a non-negative cutoff function such that $\eta=1$ on $B_{(2+r)R}(z_0)$, and $\eta=0$ outside $B_{4R}(z_0)$, and $|\nabla \eta| \leq C/R^2$. We choose $v =\eta^2/w$, where $w$ is given by \eqref{eq:Two_choices_for_w}, or \eqref{eq:DefinitionW} for concreteness, and notice that $v \in H^1_0(\sO \cup\Gamma_0, \fw)$, which can be shown by modifying the corresponding argument in the proof of \cite[Theorem 8.18]{GilbargTrudinger}. With this choice of $v$ as a test function in the variational equation
%CP 1.23.2016: Replaced with reference to the more general bilinear form
%\eqref{eq:HestonWithKillingBilinearForm}
%PF 2-27-2016: OK
\eqref{eq:Operator_A_bilinear_form}
satisfied by $u$, using the strict ellipticity of
%PF 2-23.2016 added clarifications because of notational conflict
%CP 2.24.2016: I think that it is better to add reference to
%\eqref{eq:Operator} than to \eqref{eq:A_nondivergence_form}.
%PF 2-25-2016: OK
the operator $y^{-1}A$ defined by
%\eqref{eq:A_nondivergence_form}
\eqref{eq:Operator}
and H\"older's inequality, we see that there is a positive constant $C=C(\Lambda, \nu_0, \bar R)$, such that
\begin{equation}
\label{eq:Eq15}
\int_{\sO} \eta^2 |\nabla \log w|^2 y^{\beta}\,dx\,dy
\leq
C \int_{\sO} (|\nabla \eta|^2 + \eta^2) y^{\beta}\,dx\,dy  + C\int_{\sO} \eta^2 \frac{|f|+|u|}{w} y^{\beta-1}\,dx\,dy.
\end{equation}
From Lemma \ref{lem:MeasureBalls}
and the fact that $|\nabla \eta| \leq C/R^2$, we have
\begin{equation}
\label{eq:Eq16}
\begin{aligned}
\int_{\sO} (|\nabla \eta|^2 + \eta^2) y^{\beta}\,dx\,dy
& \leq C \left((2+r)R\right)^{-4}|B_{(2+r)R}(z_0)|_{\beta}.
\end{aligned}
\end{equation}
Using the definition \eqref{eq:DefinitionK} of $k(R)$ and H\"older's inequality, we obtain
\begin{equation}
\label{eq:L1EstforAbstractJohnNirenberg}
\int_{\sO} \eta^2 \frac{|f|+|u|}{w} y^{\beta-1}\,dx\,dy
\leq C \left(R^{2(n+\beta-1)/q^*}  + R^{2(n+\beta-1)-\omega}\right).
\end{equation}
The condition $q>n+\beta$ implies
\begin{equation}
\label{eq:AuxJNExponent1}
2(n+\beta-1)/q^*-2(n+\beta) > -4,
\end{equation}
since $1/q+1/q^*=1$. Also, because $\omega$ is chosen in $(0,2(n+\beta-1)/q)$ in \eqref{eq:Choicedelta} and $q>n+\beta$ in \eqref{eq:Choiceq}, we see that $\omega \in (0,2)$,
and we obviously have
\begin{equation}
\label{eq:AuxJNExponent2}
-2-\omega>-4.
\end{equation}
Using \eqref{eq:AuxJNExponent1} and \eqref{eq:AuxJNExponent2}, and $0<R\leq \bar R$, we obtain in inequality \eqref{eq:L1EstforAbstractJohnNirenberg} that there is a positive constant
$C=C(\Lambda, \nu_0, \bar R)$, such that
\begin{equation}
\label{eq:Eq17}
\begin{aligned}
\int_{\sO} \eta^2 \frac{|f|+|u|}{w} y^{\beta-1}\,dx\,dy
&\leq C \left((2+r)R\right)^{-4}|B_{(2+r)R}(z_0)|_{\beta}.
\end{aligned}
\end{equation}
In the last inequality, we used Lemma \ref{lem:MeasureBalls}.
By combining equations \eqref{eq:Eq15}, \eqref{eq:Eq16} and \eqref{eq:Eq17}, we obtain
\begin{equation*}
\begin{aligned}
\int_{B_{(2+r)R}(z_0)} |\nabla \log w|^2 y^{\beta}\,dx\,dy
& \leq
C \left((2+r)R\right)^{-4}|B_{(2+r)R}(z_0)|_{\beta}.
\end{aligned}
\end{equation*}
Then, it immediately follows that the right hand side of \eqref{eq:BoundForA} is finite, and so, \eqref{eq:AbstractJohnNirenberg2} holds for $\log w$. This concludes Step \ref{step:LocalHolderInequalityApplicationAbstractJohnNirenberg}.
\end{step}

\begin{step}[Proof of inequality \eqref{eq:MainContinuity2}]
\label{step:Oscillation}
In the previous step we showed that Theorem \ref{thm:AbstractJohnNirenberg} applies to $w$ with $\theta_0=\theta_1=1$. Hence, there is a positive constant $C=C(\Lambda, \nu_0, \bar R)$, independent of $R$ and $w$, such that
\begin{multline}
\label{eq:JohnNirenbergIneqForW}
\left(\frac{1}{|B_{2R}(z_0)|_{\beta-1}}\int_{B_{2R}(z_0)} |w| y^{\beta-1}\,dx \,dy\right)
\\
\leq C
\left(\frac{1}{|B_{2R}(z_0)|_{\beta-1}}\int_{B_{2R}(z_0)} |w|^{-1} y^{\beta-1}\,dx \,dy\right)^{-1}.
\end{multline}
From \eqref{eq:Eq8} and \cite[Theorem 2.8]{Adams_1975}, we obtain
\begin{equation}
\label{eq:Eq9}
\begin{aligned}
\essinf_{B_R(z_0)} w = I(-\infty)
&\geq C \left(\frac{1}{|B_{2R}(z_0)|_{\beta-1}}\int_{B_{2R}(z_0)} |w| y^{\beta-1}\,dx \,dy\right).
\end{aligned}
\end{equation}
We now choose $w=u-m_{4R}+k$ and $w=M_{4R}-u+k$ in \eqref{eq:Eq9}. By adding the following two inequalities
\begin{equation*}
\begin{aligned}
m_R-m_{4R}+k(R)
&\geq \frac{C}{|B_{2R}(z_0)|_{\beta-1}} \int_{B_{2R}(z_0)} (u-m_{4R}) y^{\beta-1}\,dx\,dy,\\
M_{4R}-M_R+k(R)
&\geq \frac{C}{|B_{2R}(z_0)|_{\beta-1}} \int_{B_{2R}(z_0)} (M_{4R}-u) y^{\beta-1}\,dx\,dy,
\end{aligned}
\end{equation*}
we obtain
\[
(M_{4R}-m_{4R}) - (M_R-m_R) + 2k(R) \geq C \left(M_{4R}-m_{4R}\right) .
\]
Without loss of generality, we may assume $C<1$ (if not, we can make $C$ smaller on the right-hand side of the preceding inequality). Therefore, the preceding inequality can be rewritten in the form
\begin{equation}
\label{eq:AlmostHolderCont}
\begin{aligned}
\osc_{B_R(z_0)} u \leq C\osc_{B_{4R}(z_0)} u + 2k(R).
\end{aligned}
\end{equation}
Because $q \in (n+\beta, s)$ by \eqref{eq:Choiceq} and $f \in L^s(B_{\bar{R}}(z_0), \fw)$ for some $s > n+\beta$, by hypothesis in
Theorem \ref{thm:MainContinuityInterior} and the assumption $B_{\bar R}(z_0)\subset E_{R_0}(z_0)$, H\"older's inequality yields
\[
\|f\|_{L^q(B_{4R}(z_0),y^{\beta-1})} \leq C R^{2(n+\beta-1)\frac{s-q}{sq}} \|f\|_{L^s(B_{\bar{R}}(z_0),y^{\beta-1})}.
\]
Let
\[
\nu := \min \left\{\omega, 2(n+\beta-1)\frac{s-q}{sq}\right\}.
\]
Consequently, from \eqref{eq:DefinitionK}, we see that there is a positive constant $C=C(n, \beta)$, such that
\begin{equation}
\label{eq:EstimateK}
\begin{aligned}
k (R) \leq C\left( \|f\|_{L^s(B_{\bar{R}}(z_0),y^{\beta-1})} + |m_{\bar{R}}| + |M_{\bar{R}}|\right) R^{\nu}.
\end{aligned}
\end{equation}
Therefore, by applying \cite[Lemma 8.23]{GilbargTrudinger} to \eqref{eq:AlmostHolderCont} and using the inequality \eqref{eq:EstimateK}, we find that there are positive constants,
$C=C(\Lambda, \nu_0, \bar R, n, s)$ and $\alpha_0=\alpha_0(s,n,\beta) \in (0,1)$, such that
\[
\osc_{B_R(z_0)} u \leq C\left(\|f\|_{L^s(B_{\bar{R}}(z_0),y^{\beta-1})} + \|u\|_{L^{\infty}(B_{\bar{R}}(z_0))}\right) R^{\alpha_0},\quad\forall\, R \in (0,\bar R/4),
\]
Without loss of generality, we may assume that $\bar R \leq R_1$, where $R_1$ is the constant appearing in the conclusion of Theorem \ref{thm:MainSupremumEstimatesInterior}. Then the preceding estimate together with \eqref{eq:MainSupremumEstimates2} gives us \eqref{eq:MainContinuity2}. This concludes Step \ref{step:Oscillation}.
\end{step}
This concludes the proof of Theorem \ref{thm:OscillationEstimate}.
\end{proof}

We can now conclude the

\begin{proof}[Proof of Theorem \ref{thm:MainContinuityInterior}]
Notice that if $z \in B_{\bar R/16}(z_0)$, then $B_{\bar R/16}(z) \subset B_{\bar R/4}(z_0) \subset E_{R_0}(z_0)$ (by \eqref{eq:ChoicebarR}),
and so inequality \eqref{eq:MainContinuity2} applies in the form
\begin{equation}
\label{eq:MainContinuity2_new}
\osc_{B_R(z)} u \leq C\left(\|f\|_{L^s(E_{R_0}(z_0),y^{\beta-1})} + \|u\|_{L^2(E_{R_0}(z_0), y^{\beta-1})}\right) R^{\alpha_0},
\end{equation}
for all $z \in B_{\bar R/16}(z_0)$ and $0<R \leq \bar R/64$.
In the remainder of the proof of Theorem \ref{thm:MainContinuityInterior}, we assume that $R$ obeys
\begin{equation}
\label{eq:ContinuityBoundR}
0<R \leq \bar R/64.
\end{equation}
In particular,
%PF 3-7-2016: Added footnote to clear up possible coordinate/marked point confusion
for any points\footnote{Here, we are using $x_1, x_2 \in \RR^{n-1}$ to denote marked points rather than coordinates.}
%PF 2-23-2016 We should use x^\mu, \mu = 1,...,n-1, y = x^n for coordinates to avoid conflict below. Let's check usage throughout
%CP 2.24.2016: I did not find any conflict. When we denoted the center
%of a ball, we used z_0 with a lower index. This is the reason why I
%continued to use z_i with a lower index. I did the same with x_i and
%y_i. If you check in (5.43), we square y_i, and I think it is better
%in that relation to have lower indices than upper indices. But, if
%you prefer upper indices, please feel free to make this change.
%PF 2-25-2016: I will check again after redoing the Introduction
$(x_1,y_1), (x_1,0), (x_2,0) \in \bar B_R(z_0)$, the estimate \eqref{eq:MainContinuity2_new} gives
\begin{equation}
\label{eq:Eq19}
\begin{aligned}
{}&|u(x_1,y_1)-u(x_1,0)|
\\
&\qquad \leq C\left(\|f\|_{L^s(E_{R_0}(z_0),y^{\beta-1})} + \|u\|_{L^2(E_{R_0}(z_0), y^{\beta-1})}\right) d^{\alpha_0}\left((x_1,y_1),(x_1,0)\right),
\\
{}&|u(x_1,0)-u(x_2,0)|
\\
&\qquad \leq C\left(\|f\|_{L^s(E_{R_0}(z_0),y^{\beta-1})} + \|u\|_{L^2(E_{R_0}(z_0), y^{\beta-1})}\right)
d^{\alpha_0}\left((x_1,0), (x_2,0)\right).
\end{aligned}
\end{equation}
Notice that
%PF 2-23-2016 added
%CP 2.24.2016: Ok
from \eqref{eq:IntroKochDistance}
we have the simple identities,
\begin{equation}
\label{eq:SimpleDistances}
\begin{aligned}
d\left((x_1,y_1),(x_1,0)\right) &= \sqrt{y_1/2},\\
d\left((x_1,0), (x_2,0)\right)  &= \sqrt{|x_1-x_2|},
\end{aligned}
\end{equation}
and so, we can rewrite \eqref{eq:Eq19} in the form
\begin{equation}
\label{eq:Eq19Add}
\begin{aligned}
|u(x_1,y_1)-u(x_1,0)| & \leq C\left(\|f\|_{L^s(E_{R_0}(z_0),y^{\beta-1})} + \|u\|_{L^2(E_{R_0}(z_0), y^{\beta-1})}\right) |y_1|^{\alpha_0/2},\\
|u(x_1,0)-u(x_2,0)| & \leq C\left(\|f\|_{L^s(E_{R_0}(z_0),y^{\beta-1})} + \|u\|_{L^2(E_{R_0}(z_0), y^{\beta-1})}\right) |x_1-x_2|^{\alpha_0/2}.
\end{aligned}
\end{equation}
The proof of inequality \eqref{eq:MainContinuity3} now follows the proofs of
\cite[Corollary I.9.7 and Theorem I.9.8]{DaskalHamilton1998}, but with certain differences which we outline for clarity.

\begin{claim}
\label{claim:Eq22}
There are constants $C=C(\Lambda,n,\nu_0,R_0,s)>0$, and $\alpha=\alpha(\Lambda,n,\nu_0,R_0,s)\in (0,1)$ such that
\begin{equation}
\label{eq:Eq22}
|u(z_1)-u(z_2)| \leq C\left(\|f\|_{L^s(E_{R_0}(z_0),y^{\beta-1})} + \|u\|_{L^2(E_{R_0}(z_0), y^{\beta-1})}\right)d^\alpha(z_1,z_2),
\end{equation}
for all points $z_1, z_2 \in B_{\bar R/16}(z_0)$.
\end{claim}

\begin{proof}
Let $\eps \in (0, 1/8)$ be fixed and consider the following two cases.

\begin{case}[Pairs of points in $B_R(z_0)$ obeying \eqref{eq:ContinuityCase1}]
\label{case:ContinuityCase1}
Let $z_i=(x_i,y_i)\in B_R(z_0)$, for $i=1,2$, be such that
\begin{equation}
\label{eq:ContinuityCase1}
\begin{aligned}
|z_1-z_2| \geq \eps (y_1^2 +y_2^2).
\end{aligned}
\end{equation}
We want to show that \eqref{eq:Eq22} holds, for all points $z_1, z_2 \in B_R(z_0)$ satisfying \eqref{eq:ContinuityCase1}.

From \eqref{eq:ContinuityAssumption2}, we can find a positive constant $C$ such that
\begin{equation}
\label{eq:Eq20}
\begin{aligned}
|x_1-x_2| \leq C d(z_1,z_2).
\end{aligned}
\end{equation}
Using our current assumption \eqref{eq:ContinuityCase1}, in addition to \eqref{eq:ContinuityAssumption2}, we also have
\[
d(z_1,z_2) \geq \eps C y_i^2,\quad i=1,2,
\]
and so, there exists a positive constant $C$, depending on $\eps$, such that
\begin{equation}
\label{eq:Eq20Add}
\begin{aligned}
y_i \leq C d^{1/2}(z_1,z_2), \quad i=1,2.
\end{aligned}
\end{equation}
Denote $z'_i=(x_i,0)$, for $i=1,2$. Applying \eqref{eq:Eq20} and \eqref{eq:Eq20Add} in \eqref{eq:Eq19Add}, we obtain
\begin{equation*}
\begin{aligned}
|u(z_i)-u(z'_i)|  &\leq C\left(\|f\|_{L^s(E_{R_0}(z_0),y^{\beta-1})} + \|u\|_{L^2(E_{R_0}(z_0), y^{\beta-1})}\right) d^{\alpha_0/4}(z_1,z_2),    &i=1,2,\\
|u(z'_1)-u(z'_2)| &\leq C\left(\|f\|_{L^s(E_{R_0}(z_0),y^{\beta-1})} + \|u\|_{L^2(E_{R_0}(z_0), y^{\beta-1})}\right) d^{\alpha_0/2}(z_1,z_2),
\end{aligned}
\end{equation*}
and hence, using \eqref{eq:ContinuityAssumption2},
\begin{align*}
|u(z_1)-u(z_2)| &\leq |u(z_1)-u(z'_1)| + |u(z'_1)-u(z'_2)| + |u(z_2)-u(z'_2)|\\
   &\leq C\left(\|f\|_{L^s(E_{R_0}(z_0),y^{\beta-1})} + \|u\|_{L^2(E_{R_0}(z_0), y^{\beta-1})}\right) d^{\alpha_0/4}(z_1,z_2).
\end{align*}
Therefore, the estimate \eqref{eq:Eq22} holds in the special case $|z_1-z_2| \geq \eps (y_1^2+y_2^2)$.
\end{case}

Now we prove \eqref{eq:Eq22} for pairs of points obeying $|z_1-z_2| < \eps (y_1^2+y_2^2)$.

\begin{case}[Pairs of points in $B_R(z_0)$ obeying \eqref{eq:ContinuityCase2}]
\label{case:ContinuityCase2}
Now we consider points $z_i=(x_i,y_i)\in B_R(z_0)$, for $i=1,2$, such that
\begin{equation}
\label{eq:ContinuityCase2}
\begin{aligned}
|z_1-z_2| < \eps (y_1^2 +y_2^2).
\end{aligned}
\end{equation}
By scaling and using interior H\"older estimates \cite[Theorem 8.22]{GilbargTrudinger}, we show that the estimate \eqref{eq:MainContinuity3} also holds in this case. We proceed by analogy with the proofs of \cite[Theorems I.9.1--4
and Corollary I.9.7]{DaskalHamilton1998}. We may assume without loss of generality that
\begin{equation}
\label{eq:AssumptionCase2}
\begin{aligned}
1>y_2\geq y_1 \quad\hbox{and}\quad x_2=0.
\end{aligned}
\end{equation}
Let $a=y_2$.  We consider the function $v$ defined by rescaling,
\[
u(x,y) =: v(x/a, y/a).
\]
The rescaling $z \mapsto z'=z/a$ maps $\EE_{y_2/2}(z_2)$ into $\EE_{1/2}(z'_2)$. Recall that $\EE_{\rho}(z)$ denotes the Euclidean ball centered at $z$ of radius $\rho$ relative to $\HH$ (see \eqref{eq:Euclidean_balls_relative_to_the_half_space}). From our assumptions \eqref{eq:ContinuityAssumption2}, \eqref{eq:ContinuityCase2} and the choice of $\eps\in (0, 1/8)$, we see that
\begin{equation}
\label{eq:EqAux1}
\begin{aligned}
|z'_1-z'_2| \leq 2\eps y_2 < 1/4,
\end{aligned}
\end{equation}
and so $z'_1\in \EE_{1/4}(z'_2)$. From \cite[Theorem 5.10]{Daskalopoulos_Feehan_statvarineqheston}, we know that $u \in H^2_{\loc}(B_{\bar R}(z_0))$, and so by direct calculation, we conclude that $v(z')$  solves
\[
\tilde A v (z') = a f(az') \quad\hbox{on } \EE_{1/2}(z'_2),
\]
where we define
\[
\begin{aligned}
%CP 1.23.2016: Changed the definition of the operator \tilde A to variable coefficients
%PF 2-27-2016: OK
\tilde A v(z') &:= y' \left(a^{ij}(az')v_{x_ix_j}(z')+2a^{in}(az') v_{x_iy}(z') + a^{nn}(az') v_{yy}(z')\right)
\\
&\quad  +b^i(az') v_{x_i}(z')+ b^n(az') v_y(z')-c(az') v(z').
\end{aligned}
\]
On the ball $\EE_{1/2}(z'_2)$, the operator $\tilde A$ is strictly elliptic with bounded coefficients. For brevity, we denote $f_a(z') := af(a z')$. By \cite[Theorem 8.22]{GilbargTrudinger}, there are positive constants $C$ and $\alpha_1 \in (0,1)$, depending only on $\Lambda$, $n$, $\nu_0$ and $s$, such that
\begin{equation}
\label{eq:Rescaling1}
\begin{aligned}
\osc_{\EE_{R}(z'_2)} v \leq C R^{\alpha_1} \left(\|v\|_{L^{\infty}(\EE_{1/2}(z'_2))} + \|f_a\|_{L^s(\EE_{1/2}(z'_2))}\right), \quad \forall\, R \in (0,1/2],
\end{aligned}
\end{equation}
because $s$ was assumed to satisfy $s>2n$.
%PF 2-23-2016 Presumably not now
%(recall that $n=2$).
%CP 2.24.2016: Ok
We see that
\begin{equation}
\label{eq:Rescaling2}
\begin{aligned}
\|v\|_{L^{\infty}(\EE_{1/2}(z'_2))} = \|u\|_{L^{\infty}(\EE_{y_2/2}(z_2))} \leq \|u\|_{L^{\infty}(B_{\bar{R}}(z_0))},
\end{aligned}
\end{equation}
where we used the fact that $\EE_{y_2/2}(z_2) \subseteqq B_{\bar R}(z_0)$, which in turn follows
from our assumption \eqref{eq:ContinuityBoundR}.
We also have
\begin{align*}
\|f_a\|^s_{L^s(\EE_{1/2}(z'_2))}
= \int_{\EE_{1/2}(z'_2)} |a f(az')|^s \,dx'\,dy'
= \int_{\EE_{y_2/2}(z_2)} |f(z)|^s a^{s-n}\,dx\,dy,
\end{align*}
that is,
\begin{equation}
\label{eq:Rescaling3}
\begin{aligned}
\|f_a\|^s_{L^s(\EE_{1/2}(z'_2))}
= \int_{\EE_{y_2/2}(z_2)} |f(z)|^s a^{s-n}\,dx\,dy.
\end{aligned}
\end{equation}
Using the fact that $y_2/2 \leq y \leq 3 y_2/2$ for all $z=(x,y) \in \EE_{y_2/2}(z_2)$, assumption \eqref{eq:ContinuityAssumption2}, and the fact that $s>n+\beta$ by hypothesis of Theorem \ref{thm:MainContinuityInterior}, the estimate
\eqref{eq:Rescaling3} yields
\begin{equation}
\label{eq:Rescaling4}
\begin{aligned}
\|f_a\|^s_{L^s(\EE_{1/2}(z'_2))} &\leq C\int_{B_{\bar{R}}(z_0)} |f(z)|^s y^{\beta-1}\,dx\,dy,
\end{aligned}
\end{equation}
where $C$ is a positive constant depending only on $\beta$. Applying \eqref{eq:Rescaling2} and \eqref{eq:Rescaling4} in \eqref{eq:Rescaling1} yields
\begin{equation*}
\begin{aligned}
\osc_{\EE_{R}(z'_2)} v \leq C \left(\|u\|_{L^{\infty}(B_{\bar{R}}(z_0))} + \|f\|_{L^s(B_{\bar{R}}(z_0))}\right) R^{\alpha_1}, \quad \forall\, R \in (0,1/2].
\end{aligned}
\end{equation*}
In particular, because $z'_1\in \EE_{1/2}(z'_2)$, we see that
\begin{equation*}
\begin{aligned}
|v(z'_1)-v(z'_2)| \leq C\left(\|u\|_{L^{\infty}(B_{\bar{R}}(z_0))} + \|f\|_{L^s(B_{\bar{R}}(z_0))}\right) |z'_1-z'_2|^{\alpha_1},
\end{aligned}
\end{equation*}
where the positive constant $C$ depends on $\Lambda$, $n$, $\nu_0$ and $s$. By rescaling back, we obtain
\begin{equation}
\label{eq:Eq24}
\begin{aligned}
|u(z_1)-u(z_2)|
&\leq C \left(\|u\|_{L^{\infty}(B_{\bar{R}}(z_0))} + \|f\|_{L^s(B_{\bar{R}}(z_0))}\right) \left(\frac{|z_1-z_2|}{y_2} \right)^{\alpha_1}.
\end{aligned}
\end{equation}
Using \eqref{eq:ContinuityAssumption2} and the fact that $\eps\in(0,1/8)$, we see that
\begin{equation}
\label{eq:RelationshipDistances}
\frac{|z_1-z_2|}{y_2} \leq d^{1/2}(z_1,z_2).
\end{equation}
Consequently, \eqref{eq:Eq24} and \eqref{eq:MainSupremumEstimates2}  give us
\[
|u(z_1)-u(z_2)| \leq C\left(\|f\|_{L^s(E_{R_0}(z_0),y^{\beta-1})} + \|u\|_{L^2(E_{R_0}(z_0), y^{\beta-1})}\right) d^{\alpha_1/2}(z_1,z_2).
\]
This implies estimate \eqref{eq:Eq22} in the special case $|z_1-z_2| < \eps (y_1^2+y_2^2)$.
\end{case}
This completes the proof of Claim \ref{claim:Eq22} by choosing $\alpha:=\min\{\alpha_0/4,\alpha_1/2\}$.
\end{proof}

By choosing $R_1$ smaller than $(\bar R/16)^2$ and than the constant $R_1$ in the conclusion of Theorem \ref{thm:MainSupremumEstimatesInterior}, we see by \eqref{eq:EuclidBallInsideCycloid} that $E_{R_1}(z_0)\subset B_{\bar R/16}(z_0)$, and so estimates \eqref{eq:Eq22} and \eqref{eq:MainSupremumEstimates2} now give us \eqref{eq:MainContinuity3}. This completes the proof of Theorem \ref{thm:MainContinuityInterior}.
\end{proof}

\subsection{H\"older continuity on neighborhoods of the corner points of the degenerate boundary}
\label{subsec:HolderContinuityCornerPointsDegenBdry}
We now have
%CP 1.11.2016: Added
the

\begin{proof} [Proof of Theorem \ref{thm:MainContinuityBoundary}]
Suppose $z_0\in \bar\Gamma_0\cap\bar\Gamma_1$. We let $\bar R$ be as in the proof of Theorem \ref{thm:MainContinuityInterior}, but in addition we require that $\bar R$
be small enough so that the conclusion of Lemma \ref{lem:RegularDomain} holds with the cone, $K$, given
in the hypotheses of Theorem \ref{thm:MainContinuityBoundary}.
From the standard theory of non-degenerate elliptic partial differential equations (for example, \cite[Theorem
8.30]{GilbargTrudinger}), we know that
\begin{equation}
\label{eq:Continuity_u_on_Gamma_1}
u\in C(\bar B_{\bar R}(z_0)\cap \HH)\quad\hbox{and}\quad u=0 \quad\hbox{on } \partial B_{\bar R}(z_0) \cap \Gamma_1.
\end{equation}
Recalling that $u^+ = \max\{u,0\}$ and $u^- = \max\{-u,0\}$ denote the positive and negative parts of $u$, respectively, we have that $u^{\pm} \in C(\bar B_{\bar R}(z_0)\cap \HH)$ and $u^{\pm}=0$ along the portion of the boundary $\partial B_{\bar R}(z_0) \cap \Gamma_1$.

Our goal is first to prove that there are constants $C$, depending only on
$\Lambda$, $\nu_0$, $K$, $n$, $s$, $\bar R$, and $\alpha_0$, depending only on $n$, $s$ and $\beta$, such that
\begin{equation}
\label{eq:RestatementCorners_pos_neg}
\osc_{B_R(z_0)}u^{\pm} \leq C\left(\|f\|_{L^s(E_{R_0}(z_0), y^{\beta-1})} + \|u\|_{L^2(E_{R_0}(z_0), y^{\beta-1})}\right) R^{\alpha_0}, \quad\forall\, R \in (0,\bar R/4],
\end{equation}
which obviously implies that \eqref{eq:MainContinuity2} holds for $u$, for possibly a different constant $C$ with the same dependency as above.

Our proof uses the same method as in the case of points in $\Gamma_0$ but a choice of $w$ which is different from that of \eqref{eq:Defnw}, and a choice of test function $v$ which is different from that of \eqref{eq:Defnv}.
Moreover, we do not need to appeal to the John-Nirenberg inequality. Since $z_0\in \bar\Gamma_0\cap\bar\Gamma_1$, however, it is important to make a distinction between $B_R(z_0)$ and $\BB_R(z_0)$.

We denote
\begin{equation}
\label{eq:Maximum_plus_minus}
M^{\pm}_R : = \esssup_{B_{R}(z_0)} u^{\pm}.
\end{equation}
Let $k\equiv k(R)$ be defined as in \eqref{eq:DefinitionK}. Therefore, we now define $w^{\pm}$ on $\BB_{4R}(z_0)$ by
\begin{equation}
\label{eq:ContCornersW}
\begin{aligned}
w^{\pm}(z) &:= k +
\begin{cases}
-u^{\pm}(z) + M^{\pm}_{4R}, & z \in \BB_{4R}(z_0) \cap B_{4R}(z_0),\\
     +M^{\pm}_{4R}, & z \in \BB_{4R}(z_0)\backslash B_{4R}(z_0).
\end{cases}
\end{aligned}
\end{equation}
As in the case of points in $\Gamma_0$, we may assume without loss of generality that \eqref{eq:Assump_max} and  \eqref{eq:Assump_k} hold. From \eqref{eq:Continuity_u_on_Gamma_1}, we notice that $M_{4R} \geq 0$ and $m_{4R} \leq 0$, and so it follows that $M_{4R}=M^+_{4R}$ and $m_{4R}=-M^-_{4R}$. Therefore, assumption \eqref{eq:Assump_max} becomes
\[
M^+_{4R} \neq 0 \quad\hbox{or}\quad M^-_{4R} \neq 0.
\]
If $M^-_{4R} = 0$, then $u=u^+$ on $B_{4R}(z_0)$, and it suffices to continue the following argument only for $u^+$. The same remark applies to $M^+_{4R}=0$. Thus, we may assume without loss of generality that
\begin{equation}
\label{assump:Non_zero_max_and_min}
M^+_{4R} \neq 0 \quad\hbox{and}\quad M^-_{4R} \neq 0.
\end{equation}
Let $\alpha <-1$, and let $\eta$ be a smooth cutoff function such that $\supp \eta \subseteqq \BB_{4R}(z_0)$. We now define
\begin{equation}
\label{eq:CornersContV}
v^{\pm} := \eta^2 \left(\left(w^{\pm}\right)^{\alpha}-(k+M^{\pm}_{4R})^{\alpha}\right).
\end{equation}
We notice that $v^{\pm}$ is a well-defined function, for any choice of $\alpha\in\RR$, by \eqref{assump:Non_zero_max_and_min} and \eqref{eq:Assump_k},
and $v^{\pm} \in H^1_0(\sO\cup\Gamma_0, \fw)$ is a valid test function in
%CP 1.23.2016: Replaced with the more general bilinear form
%\eqref{eq:HestonWithKillingBilinearForm}
%PF 2-27-2016: OK
\eqref{eq:Operator_A_bilinear_form}
by \cite[Lemma A.3]{Feehan_Pop_regularityweaksoln_v3}.
We observe that the function $w^{\pm}$ obeys
\[
k \leq w^{\pm} \leq k+M^{\pm}_{4R} \quad\hbox{on }\BB_{4R}(z_0),
\]
and, because $\alpha$ is non-positive, we also have
\[
k^{\alpha} \geq \left(w^{\pm}\right)^{\alpha} \geq \left(k+M^{\pm}_{4R}\right)^{\alpha} \quad\hbox{on }\BB_{4R}(z_0).
\]
These inequalities are important in deriving the analogues of the energy estimates in the proof of Theorem \ref{thm:MainContinuityInterior} for points in $\Gamma_0$. Steps \ref{step:ContinuityVarEqEnergyEstw} and \ref{step:ContinuityVarEqEnergyMoserNegPower} in the proof of Theorem \ref{thm:MainContinuityInterior} for points in $\Gamma_0$ apply to our current choice of $w^{\pm}$ for points in
$\bar\Gamma_0\cap\bar\Gamma_1$, with the only exception that we now define $I(N)$ by
\[
I(N) := \left(\int_{\BB_{R_N}(z_0)} |w^{\pm}|^{p_N} y^{\beta-1}\,dx\,dy\right)^{1/P_N}.
\]
Therefore, using the fact that
\begin{equation}
\label{eq:ContinuityDomainCondition}
\begin{aligned}
 |\BB_R(z_0)\backslash B_R(z_0)|_{\beta-1} \neq 0,
\end{aligned}
\end{equation}
we obtain the analogue of \eqref{eq:Eq8},
\begin{equation}
\label{eq:Inequality1_w_plus_minus}
\essinf_{\BB_R} w^{\pm}
\geq C \left(\frac{1}{|\BB_{2R}(z_0)|_{\beta-1}}\int_{\BB_{2R}(z_0)\backslash B_{2R}(z_0)} |w^{\pm}|^{p_0} y^{\beta-1}\,dx \,dy\right)^{1/p_0},
\end{equation}
where $p_0$ is a negative power and $C=C(K, \Lambda, \nu_0, n, s)$. Condition \eqref{eq:ContinuityDomainCondition} is implied by \eqref{eq:DomainCondition_exterior}, which follows from the
exterior cone condition on $\bar \Gamma_0 \cap \bar \Gamma_1$, by \eqref{eq:DomainCondition_exterior}. Notice that \eqref{eq:ContCornersW} implies
\begin{align}
\label{eq:Inequality2_w_plus_minus}
w^{\pm} &= k+M^{\pm}_{4R}\geq M^{\pm}_{4R}\quad\hbox{on}\quad \BB_{2R}(z_0)\backslash B_{2R}(z_0),
\\
\label{eq:Infimum_w_plus_minus}
\essinf_{\BB_R(z_0)} w^{\pm} &= k-M^{\pm}_{R}+M^{\pm}_{4R}.
\end{align}
Using \eqref{eq:Infimum_w_plus_minus} on the left-hand-side of \eqref{eq:Inequality1_w_plus_minus} and \eqref{eq:Inequality2_w_plus_minus} on the right-hand-side of \eqref{eq:Inequality1_w_plus_minus}, we obtain
\begin{equation}
\label{eq:Eq18}
k(R)-M^{\pm}_{R}+M^{\pm}_{4R}
\geq
C M^{\pm}_{4R}.
\end{equation}
Indeed, \eqref{eq:Eq18} follows because $p_0<0$ and
\[
\frac{|\BB_{2R}(z_0)|_{\beta-1}}{|\BB_{2R}(z_0)\backslash B_{2R}(z_0)|_{\beta-1}} \geq 1.
\]
We rewrite \eqref{eq:Eq18}, using $\osc_{B_R(z_0)}u^{\pm} =M^{\pm}_R$, as
\[
\osc_{B_R(z_0)} u^{\pm} \leq C \osc_{B_{4R}(z_0)} u^{\pm} + k(R),
\]
where $C \in (0,1)$ is a constant independent of $R$. Just as in the proof of Theorem \ref{thm:MainContinuityInterior} for the case of points in $\Gamma_0$, we can apply \cite[Lemma 8.23]{GilbargTrudinger} to conclude that \eqref{eq:RestatementCorners_pos_neg} holds for $u^{\pm}$ with positive constants $C=C(K, \Lambda, \nu_0, n, s, \bar R)$, and $\alpha_0=\alpha_0(s,n,\beta)\in (0,1)$, which implies that \eqref{eq:MainContinuity2} holds for $u$, for possibly a different constant $C$ with the same dependencies as before.

To establish \eqref{eq:MainContinuity3}, we proceed as in the proof of Theorem \ref{thm:MainContinuityInterior} for the case of points in $\Gamma_0$. In order to adapt the argument for the case of points in $\Gamma_0$ to points in $\bar{\Gamma}_0 \cap \bar{\Gamma}_1$, we need analogues of the inequalities \eqref{eq:Eq19} to hold in a neighborhood in $\sO$ of $z_0 \in \bar{\Gamma}_0 \cap \bar{\Gamma}_1$. Given these analogues of the inequalities \eqref{eq:Eq19}, we can apply the same argument as used in the proof of Theorem \ref{thm:MainContinuityInterior} for the case of points in $\Gamma_0$, but instead of applying \cite[Theorem 8.22]{GilbargTrudinger}, we now apply \cite[Theorem 8.27]{GilbargTrudinger}. As before, we assume
\eqref{eq:ContinuityBoundR} holds.

Without loss of generality, we may assume
%PF 2-23-2016 As before, let's make sure we use x^\mu for coordinates and z_i, x_i, y_i,  for marked points
%CP 2.24.2016: I did not find any conflict and for the previous reason I prefer to keep this notation.
$z_0=(0,0)$. Let $z_1=(x_1,0)$, $z_2=(x_2,0)$, $z_3=(x,y)$ and $z_4=(x,0)$ be points in $\bar B_{R}(z_0)$.
%CP 2.24.2016: We don't need these inequalities anymore.
%PF 2-25-2016: OK
%We may assume
%PF 2-23-2016 Below needs adjustment for n \geq 2 or we saying that the proof for n \geq 2 is identical to that of n=2?
%$x_2\geq x_1$ and $x, x_1,x_2 \geq 0$.
We claim that the following
analogues of the inequalities \eqref{eq:Eq19} (for points $z_0\in \Gamma_0$) hold for points $z_0 \in \bar{\Gamma}_0 \cap \bar{\Gamma}_1$,
\begin{equation}
\label{eq:ContCornersEq1}
\begin{aligned}
|u(z_1)-u(z_2)| &\leq C_3\left(\|f\|_{L^s(E_{R_0}(z_0), y^{\beta-1})} + \|u\|_{L^2(E_{R_0}(z_0,y^{\beta-1}))}\right) d^{\alpha_3}(z_1,z_2),
\\
|u(z_3)-u(z_4)| &\leq C_3\left(\|f\|_{L^s(E_{R_0}(z_0), y^{\beta-1})} + \|u\|_{L^2(E_{R_0}(z_0,y^{\beta-1}))}\right) d^{\alpha_3}(z_3,z_4),
\end{aligned}
\end{equation}
for some positive constant $C_3$ and $\alpha_3 \in (0,1)$ satisfying the same dependency conditions as in the statement of Theorem \ref{thm:MainContinuityBoundary}. For the \emph{first} inequality in \eqref{eq:ContCornersEq1}, we consider two cases.

\setcounter{case}{0}
\begin{case}[Points $z_1,z_2\in \bar B_{R}(z_0)$ obeying \eqref{eq:FirstEqCase1}]
If
\begin{equation}
\label{eq:FirstEqCase1}
\begin{aligned}
d(z_1,z_2) \geq \frac{1}{8} \max \left\{d(z_1,z_0), d(z_2,z_0)\right\},
\end{aligned}
\end{equation}
then we have
\begin{equation*}
\begin{aligned}
|u(z_1)-u(z_2)| &\leq |u(z_1)-u(z_0)| + |u(z_2)-u(z_0)|
\\
&\leq C \left(\|f\|_{L^s(E_{R_0}(z_0), y^{\beta-1})} + \|u\|_{L^2(E_{R_0}(z_0,y^{\beta-1}))}\right) d^{\alpha_0}(z_1,z_2) \quad\hbox{(by \eqref{eq:MainContinuity2} and \eqref{eq:FirstEqCase1}),}
\end{aligned}
\end{equation*}
and so the first inequality in \eqref{eq:ContCornersEq1} holds in this case.
\end{case}

\begin{case}[Points $z_1,z_2\in \bar B_{R}(z_0)$ obeying \eqref{eq:FirstEqCase2}]
If
\begin{equation}
\label{eq:FirstEqCase2}
\begin{aligned}
d(z_1,z_2) \leq \frac{1}{8} \max \left\{d(z_1,z_0), d(z_2,z_0)\right\},
\end{aligned}
\end{equation}
then, we apply \eqref{eq:Eq22} on the ball $B_{\tilde R}(z_2)$ with $\tilde R=d(z_1,z_2)$.
\end{case}
Combining the preceding two cases, we obtain the first inequality in \eqref{eq:ContCornersEq1}.

Next, we consider the \emph{second} inequality in \eqref{eq:ContCornersEq1}. By \eqref{eq:SimpleDistances}, we have
\begin{equation}
\label{eq:AdaptedSimpleDistances}
\begin{aligned}
d(z_3,z_4) = \sqrt{y/2}
\quad\hbox{and}\quad
%PF 2-23-2016 where \sqrt{x} means ...
%C 2.24.2016: x should have absolute value, which I added now
%PF 2-25-2016: OK
d(z_4,z_0) = \sqrt{|x|}.
\end{aligned}
\end{equation}
As in the proof of the first inequality in \eqref{eq:ContCornersEq1}, we consider two possible cases.

\setcounter{case}{0}
\begin{case}[Points $z_3,z_4\in \bar B_{R}(z_0)$ obeying \eqref{eq:SecondEqCase1}]
If
\begin{equation}
\label{eq:SecondEqCase1}
%PF 2-23-2016 where \geq means ...
%CP 2.24.2016: x should have absolute value
%PF 2-25-2016: OK
|x| \geq 32y,
\end{equation}
then, by \eqref{eq:AdaptedSimpleDistances}, we have $d(z_3,z_4) \leq (1/8) d(z_4,z_0)$. We may apply \eqref{eq:Eq22} on the ball $B_{\tilde R}(z_4)$ with $\tilde R=d(z_3,z_4)$, and we obtain the second inequality
in \eqref{eq:ContCornersEq1}.
\end{case}

\begin{case}[Points $z_3,z_4\in \bar B_{R}(z_0)$ obeying \eqref{eq:SecondEqCase2}]
If
\begin{equation}
\label{eq:SecondEqCase2}
%PF 2-23-2016 where < means ...
%PF 2-25-2016: OK
%CP 2.24.2016: x should have absolute value
|x|<32y,
\end{equation}
then we have $d(z_4,z_0) \leq 8 d(z_3,z_4)$. Also, a direct calculation gives us $d(z_3,z_0) \leq C d(z_3,z_4)$, for some positive constant $C$. By \eqref{eq:MainContinuity2}, we obtain
\begin{equation*}
\begin{aligned}
|u(z_3)-u(z_4)| &\leq |u(z_3)-u(z_0)| + |u(z_4)-u(z_0)|\\
& \leq 2C \left(\|f\|_{L^s(E_{R_0}(z_0), y^{\beta-1})} + \|u\|_{L^2(E_{R_0}(z_0,y^{\beta-1}))}\right) d^{\alpha_0}(z_3,z_4),
\end{aligned}
\end{equation*}
and we obtain the second inequality in \eqref{eq:ContCornersEq1}.
\end{case}

The proof of \eqref{eq:ContCornersEq1} is complete. We may now conclude that \eqref{eq:MainContinuity3} holds at points $z_0\in\bar\Gamma_0\cap\bar\Gamma_1$, by applying the same argument as in the proof of Theorem \ref{thm:MainContinuityInterior}.
\end{proof}

\begin{proof}[Proof of Corollary \ref{cor:MainHolderContinuityBoundary}]
Theorem \ref{thm:MainContinuityBoundary} can now be extended to the case when we assume that the Dirichlet boundary condition along $\Gamma_1$ is defined by a function
$g \in H^1(\sO, \fw) \cap C^{\gamma}_{s,\loc}(\bar\Gamma_1)$ with $\gamma \in (0,1]$ or a function $g \in H^1(\sO, \fw) \cap C_{\loc}(\bar\Gamma_1)$, so that $u-g \in H^1_0(\sO\cup\Gamma_0, \fw)$.
Corollary \ref{cor:MainSupremumEstimatesBoundary} and Remark \ref{rmk:MainSupremumEstimatesBoundary} shows that the solutions are essentially bounded in neighborhoods of points $z_0\in\bar\Gamma_0\cap\bar\Gamma_1$. In the proof of Theorem \ref{thm:MainContinuityBoundary} for points $z_0 \in \bar\Gamma_0\cap\bar\Gamma_1$, we need to make the following modifications. Let
\begin{align*}
M := \esssup_{\Gamma_1\cap B_{4R}(z_0)} g
\quad\hbox{and}\quad
m :=\essinf_{\Gamma_1\cap B_{4R}(z_0)} g.
\end{align*}
As in the proof of \cite[Theorem 8.27]{GilbargTrudinger}, we replace our definitions of the functions
$u^{\pm}$ (the positive and negative part of the
variational solution, respectively), $w^{\pm}$ in \eqref{eq:ContCornersW} and $v^{\pm}$ in \eqref{eq:CornersContV} by
\begin{align*}
u_M(z):= u(z)\vee M\quad\hbox{and}\quad u_m(z):=u(z)\wedge m\quad\hbox{for a.e. } z \in B_{4R},
\end{align*}
and
\begin{align*}
w_M(z) &:= k +
\begin{cases}
-u_M(z) + M_{4R}, &\hbox{for a.e. } z \in \BB_{4R}(z_0) \cap B_{4R}(z_0),
\\
-M     +M_{4R}, &\hbox{for a.e. } z \in \BB_{4R}(z_0)\backslash B_{4R}(z_0),
\end{cases}
\\
w_m(z) &:= k +
\begin{cases}
u_m(z) - m_{4R}, &\hbox{for a.e. } z \in \BB_{4R}(z_0) \cap B_{4R}(z_0),
\\
m -m_{4R}, &\hbox{for a.e. } z \in \BB_{4R}(z_0)\backslash B_{4R}(z_0).
\end{cases}
\end{align*}
and
\begin{align*}
v_M &:= \eta^2 \left((w_M)^{\alpha}-(k+M_{4R}-M)^{\alpha}\right),
\\
v_m &:= \eta^2 \left((w_m)^{\alpha}-(k+m-m_{4R})^{\alpha}\right).
\end{align*}
Inequality \eqref{eq:Inequality1_w_plus_minus} applied to $w_M$ and $w_m$ now becomes
\begin{align*}
k +M_{4R}-M_R &\geq C \left(k+M_{4R} - M\right),\\
k +m_{R}-m_{4R} &\geq C \left(k+m - m_{4R}\right),
\end{align*}
and by adding, we obtain
$$
(1-C)(M_{4R}-m_{4R}) \geq 2(C-1) k +(M_R-m_R)-C(M-m),
$$
for a constant $C\in (0,1)$. Therefore, instead of
$$
\osc_{B_{R}(z_0)} u \leq C \osc_{B_{4R}(z_0)} u + k(R),
$$
we now obtain
\begin{equation}
\label{eq:Oscillation_boundary}
\osc_{B_{R}(z_0)} u \leq C \osc_{B_{4R}(z_0)} u +C \osc_{\bar\Gamma_1\cap\bar B_{4R}(z_0)} g + k(R).
\end{equation}
Assuming that $g \in C^{\gamma}_{s,\loc}(\bar\Gamma_1)$ with H\"older exponent $\gamma\in (0,1]$,
we see that
$$
\osc_{\bar\Gamma_1\cap\bar B_{4R}(z_0)} g \leq C [g]_{C^{\gamma}_s(\bar\Gamma_1\cap\bar B_{4R}(z_0))} R^{\gamma},
$$
for a positive constant $C=C(\Lambda, n, \nu_0)$. Applying \cite[Lemma 8.23]{GilbargTrudinger} and proceeding as in Step \ref{step:Oscillation} in the proof of Theorem \ref{thm:MainContinuityInterior}, we again obtain the following analogue of estimate \eqref{eq:MainContinuity2}
\begin{equation*}
\osc_{B_R(z_0)} u \leq C\left(\|f\|_{L^s(E_{R_0}(z_0),y^{\beta-1})} + \|u\|_{L^2(E_{R_0}(z_0), y^{\beta-1})}+[g]_{C^{\gamma}_s(\bar\Gamma_1\cap\bar E_{R_0}(z_0))}\right) R^{\alpha_0},
\end{equation*}
where the constants $C$ and $\alpha_0$ satisfy the same dependencies, with the exception that $\alpha_0$ depends now in addition on $\gamma$. Then the
argument in the proof of Theorem \ref{thm:MainContinuityBoundary}, following the oscillation estimate \eqref{eq:MainContinuity2} at points $z_0 \in \bar\Gamma_0\cap\bar\Gamma_1$, can be applied to show that $u$ satisfies \eqref{eq:MainContinuity4} with $\alpha$ depending now in addition on $\gamma$.

Set $\varphi(R):=(R\bar R)^{1/2}$. When $\gamma=0$, that is, when we assume $g \in C_{\loc}(\bar\Gamma_1)$, \cite[Lemma 8.23]{GilbargTrudinger} applied to \eqref{eq:Oscillation_boundary} with $\mu=1/2$ gives
\begin{equation*}
\osc_{B_{R}(z_0)} u \leq C\left(R^{\alpha} \|u\|_{L^{\infty}(B_{4R}(z_0))} + \osc_{\bar\Gamma_1\cap\bar B_{4\varphi(R)}(z_0)} g + k(\varphi(R))\right).
\end{equation*}
for some positive constants, $C$ and $\alpha\in (0,1)$, depending only on $K$, $\Lambda$, $n$, $\nu_0$, $R_0$ and $s$. Because the right-hand-side in the preceding inequality converges to $0$ as $R$ tends to $0$, we see that $u$ is continuous at $z_0$. Therefore, using also Theorem \ref{thm:MainContinuityInterior} and \cite[Theorem 8.27]{GilbargTrudinger}, we obtain that $u \in C(\bar B_{\bar R/4}(z_0))$. Letting
$4R_1^2 =\bar R$, we see by \eqref{eq:EuclidBallInsideCycloid} that $B_{\bar R/4}(z_0) \subset E_{R_1}(z_0)$, and so $u \in C(\bar E_{R_1}(z_0))$.
\end{proof}

We now have the

\begin{proof}[Proof of Corollary \ref{cor:MainContinuity}]
The proof of the corollary follows by a standard covering argument as in \cite[Lemma 3.17]{Feehan_Pop_mimickingdegen_pde} with the aid of Theorem \ref{thm:MainContinuityInterior} and Corollary \ref{cor:MainContinuity} in place of \cite[Theorem 3.8 and Proposition 3.13]{Feehan_Pop_mimickingdegen_pde}. More details can be found in the proof of \cite[Corollary 1.17]{Feehan_Pop_regularityweaksoln_v3}.
\end{proof}

We conclude this section with the

\begin{proof}[Proof of Theorem \ref{thm:StrongMaximumPrinciple}]
Suppose first that $E_{R_0}(z_0)\Subset \sO$. Then the classical strong maximum principle \cite[Theorem 8.19]{GilbargTrudinger} implies that $u$ is constant on $\sO$, since the hyotheses \cite[Equations (8.5), (8.6), and (8.8)]{GilbargTrudinger} are obeyed on precompact open subdomains of $\HH$, as one can easily see by examining the coefficients of our bilinear form
%PF 2-25-2016 - Wrong reference
%\eqref{eq:BilinearForm},
%CP 2.25.2016: Ok
\eqref{eq:Operator_A_bilinear_form}
and this is sufficient for the proof of \cite[Theorem 8.19]{GilbargTrudinger}.

%PF 5.28.2013: Made next few lines more precise
Otherwise, by \eqref{eq:SimpleCycloidBallInsideEuclid}, we may assume that there is a constant $R>0$ and a point $z'_0\in\sO\cup\Gamma_0$ such that $B_{4R}(z'_0)\Subset \sO\cup\Gamma_0$ and
$$
\esssup_{B_{R}(z'_0)} u = \esssup_{\sO} u.
$$
If $B_{4R}(z'_0)\Subset \sO$, we can apply \eqref{eq:EuclidBallInsideCycloid} to find a ball $E_{R_1}(z_0'')\Subset \sO$ obeying the hypothesis of \cite[Theorem 8.19]{GilbargTrudinger} and the previous case applies. If $z_0'\in \Gamma_0$, the argument in the proof of \cite[Theorem 8.19]{GilbargTrudinger} applies to show that $u$ is constant on a ball centered at $z_0'$, except that instead of using the classical weak Harnack inequality \cite[Inequality (8.47)]{GilbargTrudinger}, we use estimate \eqref{eq:Eq9} applied to $w=M_{4R}-u$, where we recall that $M_{4R}=\esssup_{B_{4R}(z'_0)} u$. Notice that in the definition of $w=M_{4R}-u+k(R)$ in \eqref{eq:Two_choices_for_w}, we can take $k(R)=0$ because $u$ is a subsolution to equation \eqref{eq:IntroHestonWeakMixedProblemHomogeneous} with $f=0$. To complete the proof, we can use the argument employed in the proof of \cite[Theorems 2.2 and 8.19]{GilbargTrudinger}, except that when $z_0'\in \Gamma_0$, the role of the Euclidean ball is replaced by that of the ball defined by the cycloidal distance function.
\end{proof}

\section{H\"older continuity for solutions to the variational inequality}
\label{sec:HolderContinuityVariationalInequality}
%CP 1.23.2016: We should say that the results in \cite{Daskalopoulos_Feehan_statvarineqheston} are proved for the Heston operator
%PF 2-26-2016: Done where appropriate
In this section, we use the penalization method and \apriori estimates for solutions to the penalized equation derived in \cite{Daskalopoulos_Feehan_statvarineqheston} together with Theorems \ref{thm:MainContinuityInterior} and \ref{thm:MainContinuityBoundary} to prove local
H\"older continuity on a neighborhood of $\bar\Gamma_0$ in $\bar\sO$ for solutions $u$ to the variational inequality
%CP 1.23.2016: Changed reference to the variational inequality
%\eqref{eq:IntroObstacleProblem}
%PF 2-27-2016: OK
\eqref{eq:VIProblemHeston} (Theorem \ref{thm:MainHolderContinuityVariationalInequality}).

\subsection{Reduction to an open subset with finite-height}
\label{subsec:HolderContinuityVariationalInequality_Reduction}
If $\height(\sO)=\infty$, we shall need to avail of the second condition in \eqref{eq:VISolutionW1infinity} to enable cutting off the solution and use localization to reduce to the case of an open subset with finite-height.\footnote{It is important to remember that we cannot use a cutoff function to localize solutions to the variational equation or inequality without assuming information about regularity of the solution $u$ up to $\Gamma_0$ that is stronger than what we are trying to prove.}
Let $\sU \subseteqq \sO$ be an open subset. Suppose we are given an open subset $\sV \subset \sU$ with $\bar\sV\less\partial\sO \subset \sU$ and
\begin{equation}
\label{eq:DistanceUPrimeUPositive}
\dist(\sO\cap\partial\sV,\sO\cap\partial\sU) > 0.
\end{equation}
Let $\zeta \in C^\infty(\bar\HH)$ be a cutoff function such that $0\leq \zeta\leq 1$ on $\HH$, $\zeta = 1$ on $\sV$, $\zeta>0$ on $\sU$, and $\zeta  = 0$ on $\sO\less \sU$. By \eqref{eq:DistanceUPrimeUPositive} and construction of
$\zeta$, there is a positive constant, $C_0$, depending only on $\dist(\sO\cap\partial\sV,\sO\cap\partial\sU)$ such that
\begin{equation}
\label{eq:CutoffFunctionSecondDerivativeBounds}
\|\zeta\|_{C^2(\HH)} \leq C_0.
\end{equation}
We obtain $\zeta\psi \in H^1(\sU,\fw)$ by \eqref{eq:CutoffFunctionSecondDerivativeBounds} and the fact that $\psi \in H^1(\sO,\fw)$. Because $\zeta = 0$ on $\partial\sU\less\partial\sO$ and $\psi\leq 0$ on $\Gamma_1 =
\partial\sO\less\bar\Gamma_0$ (trace sense), then $\zeta\psi\leq 0$ on $\partial\sU\less\bar\Gamma_0$ (trace sense). Similarly, as $\zeta=0$ on $\partial\sU\less\partial\sO$ and $u=0$ on $\partial\sO\less\bar\Gamma_0$ (trace
sense), then $\zeta u = 0$ on $\partial\sU\less\bar\Gamma_0$ (trace sense) and therefore
\begin{equation}
\label{eq:CutoffuH1Zero}
\zeta u \in H^1_0(\sU\cup\Gamma_0,\fw)
\end{equation}
by
%CP 1.11.2016: Changed A.32 to A.31
%PF 2-27-2016: OK
\cite[Lemma A.31]{Daskalopoulos_Feehan_statvarineqheston}.

\begin{lem}[Localization of solutions to variational inequalities]
\cite[Claim 6.16]{Daskalopoulos_Feehan_statvarineqheston}
\label{lem:ObstacleFunctionLocalization}
If $u\in H^1_0(\sO\cup\Gamma_0,\fw)$ is a solution to
%CP 1.23.2016: Changed reference to the variational inequality
%\eqref{eq:IntroObstacleProblem}
%PF 2-27-2016: OK
\eqref{eq:VIProblemHeston} with obstacle function, $\psi \in H^1(\sO,\fw)$ with $\psi^+ \in H^1_0(\sO\cup\Gamma_0,\fw)$, and source function, $f\in L^2(\sO,\fw)$, then
$\zeta u \in H^1_0(\sU\cup\Gamma_0,\fw)$ is a solution to the variational inequality
%CP 1.23.2016: Changed reference to the variational inequality
%\eqref{eq:IntroObstacleProblem}
%PF 2-27-2016: OK
\eqref{eq:VIProblemHeston}
on $\sU$ with obstacle function, $\zeta\psi\in H^1(\sU,\fw)$ with $\zeta\psi^+ \in H^1_0(\sU\cup\Gamma_0,\fw)$,
and source function,
\begin{equation}
\label{eq:fCutoff}
f_\zeta := \zeta f+[A,\zeta]u \in L^2(\sU,\fw),
\end{equation}
%CP 1.11.2016: Added
where
%PF 2-23-2016 added
%CP 2.25.2016: Ok
$A$ is as in \eqref{eq:Operator}
and the commutator $[A,\zeta]u$ is given by
%CP 1.23.2016: Changed formula of the commutator with the more general expression
\begin{comment}
\begin{align*}
[A,\zeta]u &:=-y\left[(\zeta_x+\rho\sigma\zeta_y)u_x+(\rho\sigma\zeta_x+\sigma^2\zeta_y)u_y\right]
-\frac{y}{2}(\zeta_{xx}+2\rho\sigma\zeta_{xy}+\sigma^2\zeta_{yy})u\\
&\quad -\left(r-q-\frac{y}{2}\right)\zeta_xu-\kappa(\theta-y)\zeta_yu.
\end{align*}
%
\begin{align*}
[A,\zeta]u &:=-y\left[(a^{ij}\zeta_{x_j}+a^{ji}\zeta_{x_i})u_{x_i}
+2(a^{in}\zeta_{x_i}+a^{nn}\zeta_y)u_y\right]
-y(a^{ij}\zeta_{x_ix_j}+2a^{in}\zeta_{x_iy}+a^{nn}\zeta_{yy})u
\\
&\qquad -b^i\zeta_{x_i}u-b^n\zeta_yu.
\end{align*}
\end{comment}
%PF 2-23-2016 changed; from summation convention, what was there
%before did not quite make sense and there were missing terms
%CP 2.25.2016: Ok
\begin{align*}
[A,\zeta]u &=-y\left(2a^{ij}\zeta_{x_i}u_{x_j}
+ 2a^{in}\zeta_yu_{x_i} + 2(a^{in}\zeta_{x_i} + a^{nn}\zeta_y)u_y\right)
\\
&\quad - y(a^{ij}\zeta_{x_ix_j} + 2a^{in}\zeta_{x_iy} + a^{nn}\zeta_{yy})u
- b^i\zeta_{x_i}u - b^n\zeta_yu.
\end{align*}
\end{lem}

\begin{rmk}[Reduction to the case of an open subset with finite-height]
\label{rmk:FiniteHeightStrip}
In order to reduce the case of an open subset $\sO\subseteqq\HH$ with $\height(\sO)=\infty$ to the case of an open subset $\sO\subseteqq\RR\times(0,\delta)$ with finite height $\delta>0$, we can apply Lemma \ref{lem:ObstacleFunctionLocalization}
to the choice
\begin{equation}
\label{eq:ZetaStrip}
\zeta = \begin{cases} 1 &\hbox{on }\RR\times(-\infty,\delta/2], \\ 0 &\hbox{on }\RR\times[3\delta/4,\infty), \end{cases}
\end{equation}
given by $\zeta(x,y) = \chi(y/\delta)$, for $(x,y)\in\RR^2$, where $\chi\in C^\infty(\bar\RR)$ is a cutoff function with $0\leq\chi\leq 1$ on $\RR$, $\chi(t)=1$ for $t\leq 1/2$, and $\chi(t)=0$ for $t\geq 3/4$. Observe that
$\supp[A,\zeta]u\subset \RR\times[\delta/2,3\delta/4]$ in \eqref{eq:fCutoff} and that, because $u$ obeys \eqref{eq:VISolutionW1infinity}, we obtain
$$
f_\zeta \in L^2(\sO_\delta,\fw) \cap L^\infty(\sO_\delta),
$$
and thus $f_\zeta$ obeys \eqref{eq:HolderSourceFunctionUpToBdry}, while
\begin{equation}
\label{eq:WhereSolnandCutoffSolnAgree}
\zeta u = u \quad\hbox{on } \sO_{\delta/2},
\end{equation}
with $\sO_\delta$ as in Hypothesis \ref{hyp:SourceObstacleFunctionBoundaryRegularity}.
\end{rmk}

\subsection{Proof of H\"older continuity up to $\bar\Gamma_0$ for solutions to the variational inequality}
\label{subsec:HolderContinuityVariationalInequality_ProofMainTheorem}
By Remark \ref{rmk:FiniteHeightStrip}, we may assume without loss of generality for the remainder of this section that $\sO$ has \emph{finite height},
%PF 2-26-2016: Changed from 2-dimensional to n-dimensional
\begin{equation}
\label{eq:FiniteHeightStrip}
\sO\subseteqq\RR^{n-1}\times(0,\delta),
\end{equation}
where $\delta>0$ is as in Hypothesis \ref{hyp:SourceObstacleFunctionBoundaryRegularity}, with source function (relabeled if necessary), $f$, obeying \eqref{eq:HolderSourceFunctionUpToBdry} and obtain the desired H\"older
continuity for $u$ along the open subset $\sO_{\delta/2}$ via \eqref{eq:WhereSolnandCutoffSolnAgree}.

We shall prove Theorem \ref{thm:MainHolderContinuityVariationalInequality} using the \emph{method of penalization}, following the pattern in \cite{Daskalopoulos_Feehan_statvarineqheston}, by first deriving an $L^\infty$ bound on a
\emph{penalization term}, $\beta_{\eps}(u_{\eps}-\psi)$ in the semilinear \emph{penalized equation} \eqref{eq:PenalizedEq1} corresponding to the variational inequality \eqref{eq:VIProblemHeston}, which is uniform with
respect to $\eps\in(0,\eps_0]$, for some sufficiently small positive constant $\eps_0$. We then appeal to Theorems
\ref{thm:MainContinuityInterior} and \ref{thm:MainContinuityBoundary} to conclude that the family of functions $\{u_\eps\}_{\eps\in(0,\eps_0]}$ solving the
penalized equation is $C^{\alpha_0}$-continuous up to $\bar\Gamma_0$ and hence, by passing to a subsequence and taking limits, via the convergence results in \cite{Daskalopoulos_Feehan_statvarineqheston}, that the same is true for
a solution, $u\in H^1_0(\sO\cup\Gamma_0,\fw)$, to \eqref{eq:VIProblemHeston}. Following \cite[Equations (3.1) and (3.2)]{Daskalopoulos_Feehan_statvarineqheston}, we denote\footnote{We add a term $\lambda(1+y)u$, rather than just $\lambda u$, due to the presence of the factor $1+y$ in our definition \eqref{eq:H1NormHeston} of the norm $H^1(\sO,\fw)$.}
%PF 2-26-2016 added footnote
\begin{align}
\label{eq:BilinearFormCoerciveHeston}
\fa_\lambda(u,v) &:= \fa(u,v) + \lambda((1+y)u,v)_{L^2(\sO,\fw)}, \quad\forall\, u,v \in H^1(\sO,\fw),
\\
\label{eq:CoerciveHestonOperator}
A_\lambda &:= A + \lambda(1+y),
\end{align}
where $\lambda\geq 0$ and, as usual, $\fa(u,v)$ is given by
%CP 1.23.2016: This should be replaced by \eqref{eq:Operator_A_bilinear_form}
%\eqref{eq:HestonWithKillingBilinearForm}
%PF 2-23-2016 replaced ...
%CP 2.25.2016: Ok
\eqref{eq:Operator_A_bilinear_form}
and $A$ by
%CP 1.23.2016: This should be replaced by \eqref{eq:Operator}
%\eqref{eq:OperatorHestonIntro}.
%PF 2-23-2016 replaced ...
%CP 2.25.2016: Ok
\eqref{eq:Operator}.

\begin{lem}[Uniform bound on the penalization term]
\label{lem:TroianielloTheorem4.38}
Let $f \in L^2(\sO,\fw)\cap L^\infty(\sO)$ and $\psi \in H^2(\sO, \fw)\cap L^{\infty}(\sO)$ obey \eqref{eq:ObstacleLessThanZeroBoundary}. For $u \in H^1_0(\sO\cup\Gamma_0, \fw)$ obeying $u \geq \psi$ a.e on $\sO$ and $\lambda\geq
0$, and $\eps>0$, let $u_{\eps}\in H^1_0(\sO\cup\Gamma_0,\fw)\cap L^{\infty}(\sO)$ be a solution to the \emph{penalized equation},
\begin{equation}
\label{eq:PenalizedEq1}
\fa_{\lambda}(u_{\eps},v) + (\beta_{\eps}(u_{\eps}-\psi),v)_{L^2(\sO, \fw)} = (f_{\lambda},v)_{L^2(\sO,\fw)}, \quad \forall\, v \in H^1_0(\sO\cup\Gamma_0,\fw),
\end{equation}
defined by the \emph{penalization function},
\begin{equation}
\label{eq:DefnBetaEps}
\beta_\eps(t) := -\frac{1}{\eps}t^-, \quad t\in\RR,
\end{equation}
where $t^- := -\min\{t,0\}$, and\footnote{Not to be confused with $f_\zeta$ as defined in equation \eqref{eq:fCutoff}.}
\begin{equation}
\label{eq:flambda}
f_\lambda := f + \lambda(1+y)u  \in L^2(\sO,\fw).
\end{equation}
%PF 2-23-2016 replaced r by \essinf_\sO c
%CP 2.25.2016: Ok
If
%$\lambda+r>0$,
%CP 2.25.2016: Ok
$\underline{c} := \essinf_\sO c$ and \footnote{Recall that $\|c\|_{L^\infty(\sO)} \leq \Lambda$ by \eqref{eq:Operator_A_boundedness}.}
$\lambda + \underline{c} > 0$,
there is a positive constant $\eps_0$, depending only on $n$, $\lambda$,
$\Lambda$ and $\nu_0$, such that
\begin{equation}
\label{eq:PenalizedEqBoundBetaEps}
\|\beta_{\eps}(u_{\eps}-\psi)\|_{L^{\infty}(\sO)} \leq 2\esssup_\sO (A\psi - f)^+, \quad \forall\, \eps\in (0,\eps_0].
\end{equation}
\end{lem}

\begin{proof}
We adapt an argument used in the proof of \cite[Theorem 4.38]{Troianiello}. Integration by parts \cite[Lemma 2.23]{Daskalopoulos_Feehan_statvarineqheston} with $\psi \in H^2(\sO, \fw)$ and $v \in H^1_0(\sO\cup\Gamma_0,\fw)$
yields
\begin{equation}
\label{eq:PenalizedEq2}
\fa_{\lambda}(\psi,v) = (A_{\lambda} \psi,v)_{L^2(\sO,\fw)}.
\end{equation}
Since $u_{\eps}\in H^1_0(\sO\cup\Gamma_0,\fw)$ and $\psi^+\in H^1_0(\sO\cup\Gamma_0,\fw)$, it follows that $\beta_{\eps}(u_{\eps}-\psi) \in H^1_0(\sO\cup\Gamma_0,\fw)$ by the proof of
%CP 1.11.2016: Changed A.34 to A.33
%PF 2-27-2016: OK
\cite[Lemma A.33]{Daskalopoulos_Feehan_statvarineqheston}.
%COMMENT: Maybe add details
In order to use $\beta_\eps(u_\eps-\psi)$ to construct suitable test functions, we need the forthcoming Claim \ref{claim:EssBoundedPenaltyTerm} and that relies in turn on the

%PF 2-26-2016 Added next claim and its proof
%PF 3-7-2016: Updated hypotheses and removed references to \sO bounded
\begin{claim}[Boundedness of the solution $u$ to the variational inequality]
%for $\sO$ bounded and $c$ non-negative]
\label{claim:EssBoundedSolution_u_VI}
%PF 3-7-2016: Moved to theorem hypotheses
%If $g=0$, then t
%PF 3-9-2016: Added back two possible WMP hypotheses for clarity
Assume that $\sO$ is bounded and $c \geq 0$ a.e. on $\sO$ \emph{or} that $\sO$ is unbounded and $c \geq c_0 > 0$ a.e. on $\sO$, for a positive constant $c_0$, and $\tau > 0$ in \eqref{eq:HestonWeight}.
Then the solution $u$ to the variational inequality \eqref{eq:VIProblemHeston} belongs to $L^\infty(\sO)$.
\end{claim}

\begin{proof}[Proof of Claim \ref{claim:EssBoundedSolution_u_VI}]
According to \cite[Theorem 3.16]{Daskalopoulos_Feehan_statvarineqheston}, there exists a solution $w \in H_0^1(\sO\cup\Gamma_0,\fw)$ to the inhomogeneous variational equation \eqref{eq:IntroHestonWeakMixedProblemHomogeneous}, namely
\begin{equation}
\label{eq:VE_w}
\fa(w, v) = (f, v)_{L^2(\sO,\fw)}, \quad\forall\, v \in H_0^1(\sO\cup\Gamma_0,\fw).
\end{equation}
%CP 2.28.2016: To obtain uniqueness in this theorem you use the fact
%that r>0 (although uniqueness is not important for the proof of this
%claim). Thus, when we apply this theorem, we assume that c(x,y) \geq
%c_0 >0.
%PF 2.29.2016: The condition c\geq c_0 > is not required for existence, so we do not need that condition when applying the theorem. I added a note below.
(While \cite[Theorem 3.16]{Daskalopoulos_Feehan_statvarineqheston} was proved for the Heston operator \eqref{eq:OperatorHestonIntro}, the proof for the more general operator $A$ in \eqref{eq:Operator} is identical; moreover, as is clear from the proof of existence in \cite[pp. 34-35]{Daskalopoulos_Feehan_statvarineqheston}, a condition such as $r>0$ in \eqref{eq:OperatorHestonIntro} or more generally $c\geq c_0>0$ a.e. on $\sO$ in \eqref{eq:Operator} is not required for existence even when $\sO$ is unbounded.) We rewrite the preceding  variational equation as
\[
\fa(w, v-u) = (f, v-u)_{L^2(\sO,\fw)}, \quad\forall\, v \in H_0^1(\sO\cup\Gamma_0,\fw)
\]
and subtract from the variational inequality \eqref{eq:VIProblemHeston} to give the equivalent variational inequality,
\[
\fa(u-w, v-u) \geq 0, \quad\forall\, v \in H_0^1(\sO\cup\Gamma_0,\fw), \quad v \geq \psi.
\]
Set $u_0 := u-w$ and $\psi_0 := \psi-w$ and $v_0 = v-w$, and observe that $u \in H_0^1(\sO\cup\Gamma_0,\fw)$ is a solution to the preceding variational inequality if and only if $u_0 \in H_0^1(\sO\cup\Gamma_0,\fw)$ is a solution to the variational inequality,
\[
\fa(u_0, v_0-u_0) \geq 0, \quad\forall\, v_0 \in H_0^1(\sO\cup\Gamma_0,\fw), \quad v_0 \geq \psi_0.
\]
%PF 3-7-2016: Changed WMP references according to two cases
The bilinear form $\fa$ given by \eqref{eq:Operator_A_bilinear_form} has the weak maximum principle property by \cite[Theorem 8.7]{Feehan_maximumprinciple} when $\sO$ is a bounded domain and $c \geq 0$ a.e. on $\sO$ while $\fa$ has the weak maximum principle property by \cite[Theorem 8.14]{Feehan_maximumprinciple}
when $\sO$ is a (finite height) unbounded domain and
%PF 3-9-2016: Corrected
$c \geq c_0 > 0$ a.e. on $\sO$ and $\tau > 0$.
%CP 2.28.2016: I think that in the hypotheses of this WMP it is needed that c(x,y) \geq c_0 > 0. Because otherwise, let's assume that c \equiv 0 and A is the Heston operator. Then given any solution u to the v.i., then u + C is also a solution to v.i., for all positive constants C. Let's also assume that \sO is bounded, so that the solutions shifted by constants are in the weighted Sobolev spaces.
%PF 2-29-2016: Our solutions obey a Dirichlet boundary condition on Gamma_1, so the above non-uniqueness could only occur when Gamma_0 is the entire boundary of \sO. In our application of the WMP, that cannot happen since it is being applied only when \sO is bounded. In my maximum principle paper, I was careful to distinguish between such cases in the setting of PDEs and obstacle problems: see Theorem 5.1, for example. I will have to recheck my paper for the case of VEs and VIs, as I do not seem to make the same distinction there but evidently should have. In any event, the issue does not arise here since Gamma_1 will be non-empty.
%CP 2.29.2016: Ok
%PF 3-7-2016: Split reference
Consequently, the weak maximum principle estimates
\cite[Proposition 7.9 (1) and (3)]{Feehan_maximumprinciple} for $\fa$ imply that
\[
0 \leq u_0 \leq 0\vee\esssup_\sO \psi \quad\hbox{a.e. on }\sO,
\]
where $x\vee y := \max\{x,y\}, \forall\, x,y \in \RR$. In particular, $u_0 \in L^\infty(\sO)$ in either case. Moreover, the local supremum estimates provided by Theorems \ref{thm:MainSupremumEstimatesInterior} and \ref{thm:MainSupremumEstimatesBoundary} for the solution $w \in H_0^1(\sO\cup\Gamma_0,\fw)$ to the variational equation imply that
$w \in L^\infty(\sO)$ and thus $u = u_0 + w \in L^\infty(\sO)$, as desired.
%PF 3-7-2016: Added
This completes the proof of Claim \ref{claim:EssBoundedSolution_u_VI}.
\end{proof}

Next, we have the key\footnote{The hypothesis in Claim \ref{claim:EssBoundedPenaltyTerm} that $\tau > 0$ can be removed by the alternative proof for that case in Remark \ref{rmk:ClaimEssBoundedPenaltyTerm_alternative_proof}.}

\begin{claim}[Boundedness of the penalization term]
\label{claim:EssBoundedPenaltyTerm}
%CP 3.7.2016: Left the simpler statement below because everything else is not needed now.
%If $\sO$ is bounded and $c$ is non-negative \emph{or} $c \geq c_0$ for a positive constant $c_0$, then the penalization term, $\beta_\eps(u_\eps-\psi)$, belongs to $L^\infty(\sO)$.
%PF 3-9-2016: Added back two possible WMP hypotheses for clarity
Assume that $\sO$ is bounded and $c \geq 0$ a.e. on $\sO$ \emph{or} that $\sO$ is unbounded and $c \geq c_0 > 0$ a.e. on $\sO$, for a positive constant $c_0$, and $\tau > 0$ in \eqref{eq:HestonWeight}.
Then the penalization term, $\beta_\eps(u_\eps-\psi)$, belongs to $L^\infty(\sO)$.
\end{claim}

\begin{proof}[Proof of Claim \ref{claim:EssBoundedPenaltyTerm}]
%CP 3.7.2016: Sentence no longer needed
%We first consider the case where $\sO$ is bounded and $c$ is only assumed to be non-negative.
Since $\beta_\eps(u-\psi) \leq 0$ a.e. on $\sO$, we have
\begin{equation}
%CP 2.25.2016: Removed \lambda and added label
%PF 2-26-2016: Added \lambda back
\label{eq:Eq_with_penalization_lambda}
\fa_\lambda(u_\eps,v) =
%CP 2.28.2016: Added \lambda to f
%PF 2-29-2016: OK
%(f,v)_{L^2(\sO,\fw)}
(f_{\lambda},v)_{L^2(\sO,\fw)}
- (\beta_\eps
%CP 1.23.2016: Added \eps to u
%PF 2-26-2016: Corrected f to f_\lambda
(u_{\eps}-\psi),v)_{L^2(\sO,\fw)} \geq (f_\lambda,v)_{L^2(\sO,\fw)},
\end{equation}
for all $v\in H^1_0(\sO\cup\Gamma_0,\fw)$ with $v\geq 0$ a.e. on $\sO$.
%PF 2-26-2016: Corrected
%PF 3-7-2016: Split WMP into two cases, since WMP in case \sO is bounded has simpler proof, and expanded explanation
The bilinear form $\fa_\lambda$ given by \eqref{eq:BilinearFormCoerciveHeston} has the weak maximum principle property by \cite[Theorem 8.7]{Feehan_maximumprinciple} when $\sO$ is a bounded domain and $c \geq 0$ a.e. on $\sO$ by \cite[Theorem 8.14]{Feehan_maximumprinciple} when $\sO$ is a (finite height) unbounded domain and
%PF 3-9-2016: Corrected
$c \geq c_0 > 0$ a.e. on $\sO$ and $\tau > 0$. Hence, the \apriori weak maximum principle estimate \cite[Proposition 6.5 (4)]{Feehan_maximumprinciple} for $\fa_\lambda$ implies that
$$
%CP 1.23.2016: I think that we can remove r from below (the general operator A has zeroth order term c which may be a variable coefficient so we don't want it to show up here). We should require that \lambda >\|c\|_{L^{\infty}}.
%PF 2-23-2016 replaced r by \essinf_\sO c
%CP 2.25.2016: Ok
%PF 2-26-2016: Corrected f to f_\lambda
u_\eps \geq 0\wedge\frac{1}{\lambda+\underline{c}}\essinf_\sO f_\lambda \quad\hbox{a.e. on }\sO,
$$
where $x\wedge y := \min\{x,y\},$
%CP 2.28.2016: Replaced \forall by for all
%PF 2-29-2016: OK
for all $x,y \in \RR$.
%PF 2-26-2016: Added
Because $f_\lambda = f + \lambda(1+y)u$ and $u \in L^\infty(\sO)$ by Claim \ref{claim:EssBoundedSolution_u_VI} and $\sO$ has finite height and $f \in L^\infty(\sO)$ by hypothesis \eqref{eq:HolderSourceFunctionUpToBdry}, then $f_\lambda \in L^\infty(\sO)$ and the preceding lower bound for $u_\eps$ is indeed finite. In particular,
\begin{align*}
%CP 1.23.2016: I think that we can remove r from below (the general operator A has zeroth order term c which may be a variable coefficient so we don't want it to show up here). We should require that \lambda >\|c\|_{L^{\infty}}.
%PF 2-23-2016 replaced r by \essinf_\sO c
%CP 2.25.2016: Ok
%PF 2-26-2016: Corrected f to f_\lambda below
(u_\eps-\psi)^-
&\leq \left(\esssup_\sO\psi - 0\wedge\frac{1}{\lambda+\underline{c}}\essinf_\sO f_\lambda\right)^+ \quad\hbox{a.e. on }\sO.
\end{align*}
Since $(u_\eps-\psi)^- \geq 0$ and $\psi\in L^\infty(\sO)$ by hypothesis, it follows that $(u_\eps-\psi)^- \in L^\infty(\sO)$ and thus $\beta_\eps(u_\eps-\psi) \in L^\infty(\sO)$.
%PF 3-7-2016: below redundant now
%, as desired for this case.
%PF 2-26-2016: Added slightly simplified version of your proof
%CP 3.7.2016: Removed
%PF 3-7-2016: Moved to a remark after lemma proof
This completes the proof of Claim \ref{claim:EssBoundedPenaltyTerm}.
\end{proof}

%CP 2.25.2016: In the remaining proof some signs and inequalities where not written in the right way. I corrected them.
%PF 2-27-2016: I kept your sign and inequality corrections but added back \lambda
If $F(t) := t^{q-1}$, for $q>2$, and $F'(t) = (q-1)t^{q-1}$, for $t\in\RR$, then the proofs of \cite[Lemmas 7.5 and 7.6 and Theorem 7.8]{GilbargTrudinger} (see
%CP 1.11.2016: Changed A.34 to A.33
%PF 2-27-2016: OK
\cite[Lemma A.33]{Daskalopoulos_Feehan_statvarineqheston} and its
proof) and the fact that $\beta_\eps(u_\eps-\psi) \in L^\infty(\sO)$ by Claim \ref{claim:EssBoundedPenaltyTerm} show that
%COMMENT: Maybe add details
\begin{equation}
\label{eq:PenalizedEqTestFunction}
v := |\beta_{\eps}(u_{\eps}-\psi)|^{q-1} \in H^1_0(\sO\cup\Gamma_0,\fw).
\end{equation}
By subtracting \eqref{eq:PenalizedEq2} from \eqref{eq:PenalizedEq1} and choosing $v$ as in \eqref{eq:PenalizedEqTestFunction}, we obtain
\begin{equation}
\begin{aligned}
\label{eq:PenalizedEq3}
& \fa_\lambda(u_{\eps}-\psi,|\beta_{\eps}(u_{\eps}-\psi)|^{q-1}) + (\beta_{\eps}(u_{\eps}-\psi),|\beta_{\eps}(u_{\eps}-\psi)|^{q-1})_{L^2(\sO, \fw)}
\\
&\qquad = (f_{\lambda}-A_{\lambda} \psi,|\beta_{\eps}(u_{\eps}-\psi)|^{q-1})_{L^2(\sO,\fw)}.
\end{aligned}
\end{equation}
Since $u \geq \psi$ a.e. on $\sO$ by hypothesis, the term on the right-hand side of equation \eqref{eq:PenalizedEq3} obeys
\begin{equation}
\label{eq:PenalizedEq4}
(f_{\lambda}-A_{\lambda} \psi,|\beta_{\eps}(u_{\eps}-\psi)|^{q-1})_{L^2(\sO,\fw)}
\geq
(f-A \psi,|\beta_{\eps}(u_{\eps}-\psi)|^{q-1})_{L^2(\sO,\fw)},
\end{equation}
since $f_{\lambda}-A_{\lambda} \psi = f + \lambda(1+y)(u-\psi) - A\psi \geq f-A \psi$ a.e. on $\sO$ by \eqref{eq:CoerciveHestonOperator} and \eqref{eq:flambda}. Notice that
\begin{equation}
\label{eq:PenalizedEq5}
(\beta_{\eps}(u_{\eps}-\psi),|\beta_{\eps}(u_{\eps}-\psi)|^{q-1})_{L^2(\sO, \fw)}
=
-\int_{\sO} |\beta_{\eps}(u_{\eps}-\psi)|^q \fw\, dx\,dy,
\end{equation}
and so \eqref{eq:PenalizedEq3}, \eqref{eq:PenalizedEq4}, and \eqref{eq:PenalizedEq5} yield
\begin{equation}
\label{eq:PenalizedEquationBound}
\begin{aligned}
{}&\fa_\lambda(u_{\eps}-\psi,|\beta_{\eps}(u_{\eps}-\psi)|^{q-1}) -\int_{\sO} |\beta_{\eps}(u_{\eps}-\psi)|^q \fw\, dx\,dy
\\
&\quad \geq (f-A \psi,|\beta_{\eps}(u_{\eps}-\psi)|^{q-1})_{L^2(\sO,\fw)}.
\end{aligned}
\end{equation}
Observe that \eqref{eq:PenalizedEqTestFunction}
%PF 3-2-2016: Added
and the fact that $|\beta_{\eps}(u_{\eps}-\psi)| = -\beta_{\eps}(u_{\eps}-\psi)$ by \eqref{eq:DefnBetaEps}
gives
%CP 1.23.2016: Added multidimensions
%PF 2-25-2016 added label
%CP 2.25.2016: Ok
\begin{equation}
\label{eq:Derivatives_test_function_v}
v_{x_i} =
-(q-1)|\beta_{\eps}(u_{\eps}-\psi)|^{q-2}\beta_{\eps}'(u_{\eps}-\psi)(u_{\eps}-\psi)_{x_i},\quad 1\leq i\leq n-1,
\end{equation}
and similarly for $v_y$. By a straightforward calculation using the expression \eqref{eq:BilinearFormCoerciveHeston} for $\fa_\lambda(u,v)$
%PF 2-23-2016 added
%CP 2.25.2016: Ok
and \eqref{eq:Operator_A_bilinear_form} for $\fa(u,v)$ and the expressions \eqref{eq:Derivatives_test_function_v} for $v_{x_i}$ (and $v_y$),
we find that
%CP 1.23.2016: Changed this expression with the more general bilinear form
%PF 2-27-2016: OK
\begin{comment}
\begin{equation}
\label{eq:PenalizedEq6}
\begin{aligned}
&\fa_\lambda(u_{\eps}-\psi,|\beta_{\eps}(u_{\eps}-\psi)|^{q-1}) \\
& \quad =
\frac{1}{2} \int_{\sO} \left[(u_{\eps}-\psi)^2_x + 2\varrho\sigma(u_{\eps}-\psi)_x (u_{\eps}-\psi)_y + \sigma^2 (u_{\eps}-\psi)^2_y\right]\\
&\qquad\qquad\qquad
                        \times (q-1) |\beta_{\eps}(u_{\eps}-\psi)|^{q-2} \beta_{\eps}'(u_{\eps}-\psi)  y \fw\, dx\,dy\\
&\qquad
-\frac{\gamma}{2} \int_{\sO} \left[(u_{\eps}-\psi)_x+\varrho\sigma(u_{\eps}-\psi)_y\right] |\beta_{\eps}(u_{\eps}-\psi)|^{q-1}\sign(x)y \fw\, dx\,dy\\
&\qquad
-\frac{1}{2} \int_{\sO} a_1(u_{\eps}-\psi)_x |\beta_{\eps}(u_{\eps}-\psi)|^{q-1} y \fw\, dx\,dy\\
& \qquad
+ \int_{\sO} \left(r+\lambda(1+y)\right)(u_{\eps}-\psi) |\beta_{\eps}(u_{\eps}-\psi)|^{q-1}y \fw\, dx\,dy.
\end{aligned}
\end{equation}
\end{comment}
%CP 3.7.2016: Added the new terms in the bilinear form
\begin{equation}
\label{eq:PenalizedEq6}
\begin{aligned}
&\fa_\lambda(u_{\eps}-\psi,|\beta_{\eps}(u_{\eps}-\psi)|^{q-1}) \\
& \quad =
%CP 2.25.2016: Added minus in front of the first integral. This is from the expression of v_{x_i}
-\int_{\sO} \left(a^{ij}(u_{\eps}-\psi)_{x_i}(u_{\eps}-\psi)_{x_j}
+2a^{in}(u_{\eps}-\psi)_{x_i}(u_{\eps}-\psi)_y+a^{nn}((u_{\eps}-\psi)_y)^2\right)
%PF 3-2-2016: There should not be a measure here
% y\fw \,dxdy
%PF 3-7-2016: \gamma --> \tau
\\
&\qquad\qquad\qquad
                        \times (q-1) |\beta_{\eps}(u_{\eps}-\psi)|^{q-2} \beta_{\eps}'(u_{\eps}-\psi)  y \fw\, dx\,dy\\
&\qquad +\int_{\sO} \left(\partial_{x_j}a^{ij}+\partial_y a^{in}+\hat b_i-\tau a^{ij}\frac{x_j}{|x|}-\mu a^{in}\right)(u_{\eps}-\psi)_{x_i}|\beta_{\eps}(u_{\eps}-\psi)|^{q-1} y\fw \,dxdy\\
&\qquad +\int_{\sO} \left(\partial_{x_i}a^{in}+\partial_y a^{nn}+\hat b_n-\tau a^{in}\frac{x_i}{|x|}-\mu a^{nn}\right)(u_{\eps}-\psi)_y |\beta_{\eps}(u_{\eps}-\psi)|^{q-1} y\fw \,dxdy\\
&\qquad+\int_{\sO} (c+\lambda(1+y))(u_{\eps}-\psi) |\beta_{\eps}(u_{\eps}-\psi)|^{q-1} \fw \,dxdy.
\end{aligned}
\end{equation}
We write the sum of integrals on the right-hand side of \eqref{eq:PenalizedEq6} as $I_1+I_2+I_3+I_4$. By the strict ellipticity of the operator $y^{-1}A$, we find that there exists a positive constant, $C_1=C_1(\Lambda, \nu_0)$, such that
$$
%CP 2.25.2016: Added minus in front of I_1
%PF 2-27-2016: OK
-I_1 \geq (q-1) C_1 \int_{\sO} |\nabla(u_{\eps}-\psi)|^2\beta_{\eps}'(u_{\eps}-\psi)|\beta_{\eps}(u_{\eps}-\psi)|^{q-2} y\fw\, dx\,dy,
$$
noting that $\beta_\eps'(t)\geq 0$ a.e. $t\in \RR$. Indeed, by \eqref{eq:DefnBetaEps} we have\footnote{Recall that we define $t^- = 0\vee(-t)$.}
\[
\beta_{\eps}'(t) = \frac{1}{\eps} 1_{\{t \leq 0\}} \leq \frac{1}{\eps} \quad\hbox{a.e. } t \in \RR,
\]
and so the identity,
\begin{equation}
\label{eq:GradientPenaltyTerm}
\nabla\beta_{\eps}(u_{\eps}-\psi) = \beta_\eps'(u_{\eps}-\psi)\nabla(u_{\eps}-\psi) = \frac{1}{\eps} 1_{\{u_{\eps}\leq\psi\}}\nabla(u_{\eps}-\psi)  \quad\hbox{a.e. on } \sO,
\end{equation}
yields
\begin{align*}
|\nabla(u_{\eps}-\psi)|^2\beta_{\eps}'(u_{\eps}-\psi) &= \frac{1}{\eps}|\nabla(u_{\eps}-\psi)|^2 1_{\{u_{\eps}\leq\psi\}}
\\
&= \eps|\nabla\beta_{\eps}(u_{\eps}-\psi)|^2 1_{\{u_{\eps}\leq\psi\}}
\\
&= \eps|\nabla\beta_{\eps}(u_{\eps}-\psi)|^2 \quad \hbox{a.e. on } \sO.
\end{align*}
Hence, by combining the preceding inequality and identity, we see that
\begin{equation}
\label{eq:PenalizedEq7}
%CP 2.25.2016: Added minus on the right hand side and changed inequality
I_1 \leq -\eps (q-1) C_1 \int_{\sO} |\nabla\beta_{\eps}(u_{\eps}-\psi)|^2|\beta_{\eps}(u_{\eps}-\psi)|^{q-2} y\fw\, dx\,dy.
\end{equation}
Using \eqref{eq:GradientPenaltyTerm} and the fact that $\beta_{\eps}(t)1_{\{t \leq 0\}} = \beta_{\eps}(t)$, we can write $I_2$ in the form
\begin{align*}
%CP 1.23.2016: Changed I_2 to the more general operator A
%PF 2-27-2016: OK
%CP 3.7.2016: Updated I_2 with the new terms
%PF 3-7-2016: \gamma --> \tau
I_2 &= \eps\int_{\sO} \left(\partial_{x_j}a^{ij}+\partial_y a^{in}+\hat b_i-\tau a^{ij}\frac{x_j}{|x|}-\mu a^{in}\right)\left(\beta_{\eps}(u_{\eps}-\psi)\right)_{x_i}
\\
&\qquad \times|\beta_{\eps}(u_{\eps}-\psi)|^{(q-2)/2} |\beta_{\eps}(u_{\eps}-\psi)|^{q/2} y \fw\, dx\,dy.
\end{align*}
Hence, there is a positive constant $C_2$, depending only on
$\Lambda$, $\nu_0$ and $\tau$,
%PF 3-7-2016: \gamma --> \tau and omitted redundant D-F reference
%$\gamma$ (which in turn, by \cite{Daskalopoulos_Feehan_statvarineqheston}, can be assumed to depend only on those coefficients),
such that for any $\eta>0$,
\begin{equation}
\label{eq:PenalizedEq8}
\begin{aligned}
|I_2| \leq  \eps\eta \int_{\sO} |\nabla\beta_{\eps}(u_{\eps}-\psi)|^2|\beta_{\eps}(u_{\eps}-\psi)|^{q-2} y\fw\, dx\,dy
+C_2\frac{\eps}{\eta} \int_{\sO}  |\beta_{\eps}(u_{\eps}-\psi)|^{q} y \fw\, dx\,dy.
\end{aligned}
\end{equation}
Similarly, we obtain for $I_3$, for any $\eta>0$,
\begin{equation}
\label{eq:PenalizedEq9}
\begin{aligned}
|I_3| \leq  \eps\eta \int_{\sO} |\nabla\beta_{\eps}(u_{\eps}-\psi)|^2|\beta_{\eps}(u_{\eps}-\psi)|^{q-2} y\fw\, dx\,dy
+C_3\frac{\eps}{\eta} \int_{\sO}  |\beta_{\eps}(u_{\eps}-\psi)|^{q} y \fw\, dx\,dy,
\end{aligned}
\end{equation}
where $C_3$ is a positive constant depending only on
$\Lambda$ and $\nu_0$. We can also estimate $I_4$ by
\begin{equation}
\label{eq:PenalizedEq10}
\begin{aligned}
|I_4| \leq  \eps C_4 \int_{\sO}  |\beta_{\eps}(u_{\eps}-\psi)|^{q}  \fw\, dx\,dy,
\end{aligned}
\end{equation}
where $C_4$ is a positive constant depending only on $\lambda$,
$\Lambda$, $\nu_0$, and the height of the open subset $\sO$. Substituting \eqref{eq:PenalizedEq7}, \eqref{eq:PenalizedEq8},
\eqref{eq:PenalizedEq9} and \eqref{eq:PenalizedEq10} in \eqref{eq:PenalizedEq6}, we obtain
\begin{equation*}
\begin{aligned}
%CP 2.25.2016: Removed \lambda and corrected inequalities and signs
%PF 2-26-2016: Added lambda back
{}&\fa_\lambda(u_{\eps}-\psi,|\beta_{\eps}(u_{\eps}-\psi)|^{q-1}) \\
&\quad \leq
\eps\left(\frac{C_2}{\eta}+\frac{C_3}{\eta}+C_4\right) \int_{\sO}  |\beta_{\eps}(u_{\eps}-\psi)|^{q}  \fw\, dx\,dy\\
&\qquad
-\eps((q-1)C_1-2\eta) \int_{\sO} |\nabla\beta_{\eps}(u_{\eps}-\psi)|^2|\beta_{\eps}(u_{\eps}-\psi)|^{q-2} y\fw\, dx\,dy.
\end{aligned}
\end{equation*}
Choose $\eta := C_1/2$ and, noting that $q>2$, we have $(q-1)C_1-2\eta \geq 0$ and thus
\begin{equation}
\label{eq:PenalizedEq11}
\begin{aligned}
%CP 2.25.2016: Removed \lambda and corrected inequalities and signs
%PF 2-26-2016: Added lambda back
\fa_\lambda(u_{\eps}-\psi,|\beta_{\eps}(u_{\eps}-\psi)|^{q-1})
& \leq
\eps C \int_{\sO}  |\beta_{\eps}(u_{\eps}-\psi)|^{q}  \fw\, dx\,dy,
\end{aligned}
\end{equation}
where $C := 2C_2/C_1 + 2C_3/C_1 + C_4$. But \eqref{eq:PenalizedEquationBound} gives
\begin{align*}
{}&\int_{\sO}  |\beta_{\eps}(u_{\eps}-\psi)|^{q}  \fw\, dx\,dy
\\
&\quad\leq
-(f-A \psi,|\beta_{\eps}(u_{\eps}-\psi)|^{q-1})_{L^2(\sO,\fw)} + \eps C \int_{\sO}  |\beta_{\eps}(u_{\eps}-\psi)|^{q}  \fw\, dx\,dy \quad\hbox{(by \eqref{eq:PenalizedEq11})}
\end{align*}
and thus,
\begin{align*}
(1-\eps C) \int_{\sO}  |\beta_{\eps}(u_{\eps}-\psi)|^{q}  \fw\, dx\,dy
&\leq ((A\psi - f)^+,|\beta_{\eps}(u_{\eps}-\psi)|^{q-1})_{L^2(\sO,\fw)}.
\end{align*}
Now choose $\eps_0 = 2/C$ and so $(1-\eps C) \geq 1/2$, for any $0<\eps \leq \eps_0$. By applying the H\"older inequality on the right-hand side, we see that
\begin{equation*}
\begin{aligned}
\frac{1}{2} \|\beta_{\eps}(u_{\eps}-\psi) \|_{L^{q}(\sO,\fw)}
&\leq
\|(A\psi - f)^+\|_{L^{q}(\sO,\fw)}, \quad\hbox{for } q>2 \hbox{ and }  0<\eps \leq \eps_0,
\end{aligned}
\end{equation*}
which yields, by taking the limit as $q\to\infty$ and applying
\cite[Theorem 2.8]{Adams_1975}, the desired inequality \eqref{eq:PenalizedEqBoundBetaEps}.
%PF 3-7-2016: Added
This completes the proof of Lemma \ref{lem:TroianielloTheorem4.38}.
\end{proof}

%CP 1.23.2016: This is proved only for the Heston operator
%PF 2-25-2016: Comments added; uniqueness follows from the maximum principle, which covers general operators already
%CP 2.25.2016: Ok
%PF 3-7-2016: Made into remark
\begin{rmk}[Existence of solutions to the penalized equation]
When we specialize the variable-coefficient operator $A$ in \eqref{eq:Operator} to the Heston operator \eqref{eq:OperatorHestonIntro}, then
solutions to \eqref{eq:PenalizedEq1} exist by \cite[Theorem 4.18]{Daskalopoulos_Feehan_statvarineqheston} for all $\eps>0$ and $\lambda\geq \lambda_0$, where $\lambda_0$ is a positive constant depending only on
$\Lambda$ and $\nu_0$ (see \cite[Lemma 3.2]{Daskalopoulos_Feehan_statvarineqheston}), chosen such that $\fa_\lambda$ is coercive; the proof of the corresponding existence result for $A$ in \eqref{eq:Operator} is identical.
\end{rmk}

\begin{rmk}[Alternative proof of Claim \ref{claim:EssBoundedPenaltyTerm} when $\sO$ is unbounded and $c \geq c_0 > 0$]
\label{rmk:ClaimEssBoundedPenaltyTerm_alternative_proof}
When $\sO$ is unbounded but $c \geq c_0 > 0$ a.e. on $\sO$, we may give an alternative proof of Claim \ref{claim:EssBoundedPenaltyTerm}. According to \eqref{eq:Coercive_A_delta}, the bilinear form $\fa$ is coercive for $0 < \delta < \delta_0(c_0,\Lambda,n,\nu_0)$ and we may set $\lambda = 0$. The variational inequality \eqref{eq:Eq_with_penalization_lambda} then simplifies to
\begin{equation}
\label{eq:Eq_with_penalization}
\fa(u_\eps,v) = (f,v)_{L^2(\sO,\fw)} - (\beta_\eps(u_{\eps}-\psi),v)_{L^2(\sO,\fw)} \geq (f,v)_{L^2(\sO,\fw)},
\end{equation}
for all $v\in H^1_0(\sO\cup\Gamma_0,\fw)$ with $v\geq 0$ a.e. on $\sO$. Let $w \in H^1_0(\sO\cup\Gamma_0,\fw)$ be as in the proof of Claim \ref{claim:EssBoundedSolution_u_VI} and recall that $w \in L^\infty(\sO)$. Subtracting the variational equation \eqref{eq:VE_w} from \eqref{eq:Eq_with_penalization} yields
\[
\fa(u_{\eps}-w,v) \geq 0,
\quad\forall\, v\in H^1_0(\sO\cup\Gamma_0,\fw) \text{ with } v\geq 0 \text{ a.e. on } \sO.
\]
By choosing $v:=(u_{\eps}-w)^-$ in the preceding inequality, we obtain
\[
-\fa((u_{\eps}-w)^-, (u_{\eps}-w)^-) \geq 0.
\]
Coercivity of the bilinear form \eqref{eq:Coercive_A_delta} implies that $C_0\|v\|^2_{H^1(\sO,\fw)} \leq \fa(v,v)$. Hence, $-C_0\|v\|^2_{H^1(\sO,\fw)} \geq 0$ and consequently $v = 0$ a.e on $\sO$, and thus $u_\eps \geq w$ a.e on $\sO$. Since $w, \psi\in L^\infty(\sO)$, it follows that $(u_\eps-\psi)^- \in L^\infty(\sO)$ and therefore $\beta_\eps(u_\eps-\psi) \in L^\infty(\sO)$, as desired for this case. Note that this method does \emph{not} require the hypothesis $\tau > 0$.
\end{rmk}

We can now proceed to the

\begin{proof}[Proof of Theorem \ref{thm:MainHolderContinuityVariationalInequality}]
Fix $u\in H^1_0(\sO\cup\Gamma_0,\fw)$ as in the hypothesis of Theorem \ref{thm:MainHolderContinuityVariationalInequality} and, with $f_\lambda$ as in \eqref{eq:flambda} with this choice of $u$, set
\begin{equation}
\label{eq:fpsilambdaepsilon}
f_{\lambda,\eps} := f_\lambda - \beta_\eps(u_\eps-\psi) \in L^2(\sO,\fw).
\end{equation}
%CP 2.29.2016: Added that \lambda=0 when c(x,y)\geq c_0>0.
%PF 3-1-2016: We don't need this.
%When $c\geq c_0$ on $\sO$, for a positive constant $c_0$, we recall that we can choose $\lambda=0$ since \eqref{eq:Coercive_A_delta} holds.
Since $f, \psi\in L^\infty(\sO)$
%PF 3-1-2016: Added
by \eqref{eq:HolderSourceFunctionUpToBdry} and \eqref{eq:HolderObstacleFunctionUpToBdry}
and $u$ is a solution to the variational inequality \eqref{eq:VIProblemHeston} with $g = 0$ a.e. on $\sO$, then $u$ also solves
\begin{gather*}
\fa_\lambda(u,v-u) \geq (f_\lambda,v-u)_{L^2(\sO,\fw)} \quad\hbox{and}\quad u\geq\psi \hbox{ a.e. on }\sO,
\\
\quad \forall\, v\in H^1_0(\sO\cup\Gamma_0,\fw) \hbox{ with } v\geq\psi \hbox{ a.e. on }\sO.
\end{gather*}
\begin{comment}
%PF 2-26-2016: Replaced by Claim \ref{claim:EssBoundedSolution_u_VI}.
We may assume that
%CP 1.23.2016: We have to make an argument independent of r because it does not show up in the more general operator A
%PF 2-25-2016: adjusted
%CP 2.25.2016: Ok
$\lambda+\underline{c}>0$, without loss of generality, and so the weak maximum principle for $\fa_\lambda$ in \cite[Proposition 7.9 and Theorem 8.15]{Feehan_maximumprinciple} implies that
\begin{equation}
\label{eq:uBoundStrip}
%CP 1.23.2016: We have to make an argument independent of r because it does not show up in the more general operator A
%PF 2-25-2016: adjusted
%CP 2.25.2016: Ok
\|u\|_{L^\infty(\sO)} < \frac{1}{\lambda+\underline{c}}\|f\|_{L^\infty(\sO)}\vee \|\psi\|_{L^\infty(\sO)},
\end{equation}
where $x\vee y := \max\{x,y\}$, for all $x,y \in \RR$.
\end{comment}
%CP 2.29.2016: Added the 2 different cases: bounded and unbounded domain.
%PF 3-1-2016: We don't need this. We're trying to apply Corollary 1.17 and the two cases are already included in the two previous claims and lemma
%When the domain $\sO$ is assumed to be bounded,
For a Euclidean ball $E_\delta(z_0)$ with $z_0 \in \Gamma_0$, as in the statement of Corollary \ref{cor:MainContinuity}, we observe that
%PF 3-1-2016: Added label
\begin{equation}
\label{eq:Ls_norm_ball_f_lambda_eps_leq_Linfty_norm_ball}
%CP 2.28.2016: Here we are using that the domain in bounded. For unbounded domains, this is not true, so we need to leave my previous argument.
%PF 2-29-2016: Please add your correction for the case that \sO is unbounded
\|f_{\lambda,\eps}\|_{L^s(E_\delta(z_0))} \leq \vol^{1/s}(E_\delta(z_0),\fw)\|f_{\lambda,\eps}\|_{L^\infty(E_\delta(z_0))},
\quad \forall\, \eps>0,
\end{equation}
where we take $s>2n\vee (n+\beta)$.
Claim \ref{claim:EssBoundedSolution_u_VI} implies that $u \in L^\infty(\sO)$ and the bound \eqref{eq:PenalizedEqBoundBetaEps} for $\beta_\eps(u_{\eps}-\psi)$
%PF 2-26-2016: No longer needed
%, and the uniform bound \eqref{eq:uBoundStrip} for $u$ on $\sO$ imply that
%PF 2-26-2016: Added
and the definitions \eqref{eq:flambda} for $f_\lambda$ and \eqref{eq:fpsilambdaepsilon} for $f_{\lambda,\eps}$ imply that
$$
\|f_{\lambda,\eps}\|_{L^\infty(\sO)}
\leq \|f\|_{L^\infty(\sO)}+\lambda(1+\hbox{height}(\sO))\|u\|_{L^\infty(\sO)} + 2\esssup_{\sO}(A\psi-f)^+, \quad \forall\, \eps \in (0,\eps_0],
$$
where $\eps_0>0$ is as in Lemma \ref{lem:TroianielloTheorem4.38}.
%PF 3-1-2016: We don't need this
\begin{comment}
When $\sO$ is not necessarily bounded, we recall that we choose $\lambda=0$ in \eqref{eq:fpsilambdaepsilon}, and so by Lemma \ref{lem:TroianielloTheorem4.38}, we have that
$$
\|f_{\lambda,\eps}\|_{L^\infty(\sO)}
\leq \|f\|_{L^\infty(\sO)}+ 2\esssup_{\sO}(A\psi-f)^+, \quad \forall\, \eps \in (0,\eps_0].
$$
\end{comment}
%PF 3-1-2016: Added detail in sentence to explain the point better
Because $f, u \in L^\infty(\sO)$ and $\esssup_{\sO}(A\psi-f)^+ < \infty$ by \eqref{eq:ApsifSupBound}, then \eqref{eq:Ls_norm_ball_f_lambda_eps_leq_Linfty_norm_ball} implies that $f_{\lambda,\eps}$ in \eqref{eq:fpsilambdaepsilon} obeys the hypothesis \eqref{eq:LsfL2ucondition_strip} of
Corollary \ref{cor:MainContinuity} and so, by application to the solution $u_\eps
\in H^1_0(\sO\cup\Gamma_0,\fw)$ to \eqref{eq:PenalizedEq1}, that is
$$
\fa_\lambda(u_{\eps},v) = (f_{\lambda,\eps},v)_{L^2(\sO,\fw)}, \quad \forall\, v \in H^1_0(\sO\cup\Gamma_0,\fw),
$$
we see that $u_\eps\in C_s^{\alpha_1}(\bar\sO_{\delta/2})$ satisfies estimate \eqref{eq:MainContinuity5} with $g=0$,
where the H\"older exponent $\alpha_1=\alpha_1(\delta, K, \Lambda,n,\nu_0,s)\in (0,1)$ and the constant $C=C(\delta, K, \Lambda,n,\nu_0,s)>0$ in \eqref{eq:MainContinuity5} are independent of $\eps \in (0,\eps_0]$. By the Arzel\'a-Ascoli Theorem, we can find a subsequence which converges uniformly on compact subsets of $\bar\sO_{\delta/2}$ to a function $u_0 \in C^{\alpha_1}_s(\bar\sO_{\delta/2})$. But \cite[Theorem 6.2]{Daskalopoulos_Feehan_statvarineqheston} and the choice \eqref{eq:flambda} of $f_\lambda=f+\lambda(1+y)u$ imply that $u_\eps\to u$ strongly in $L^2(\sO,\fw)$ (in fact, $H^1_0(\sO\cup\Gamma_0,\fw)$) as
$\eps\downarrow 0$ and thus, after passing to a subsequence, $u_\eps\to u$ pointwise a.e. on $\sO$ as $\eps\downarrow 0$. Therefore, by choosing a diagonal subsequence, we obtain $u=u_0$ a.e. on $\sO_{\delta/2}$, and the result
follows.
\end{proof}

Now we can give the
\begin{proof}[Proof of Corollary \ref{cor:MainHolderContinuityVariationalInequality}]
We reduce the proof to the setting of Theorem \ref{thm:MainHolderContinuityVariationalInequality} by defining
$$
\tilde u:=u-g,\quad \tilde \psi:=\psi-g,\quad\tilde f:=f-Ag.
$$
Notice that $\tilde u$, $\tilde \psi$ and $\tilde f$ satisfy the assumptions of Theorem \ref{thm:MainHolderContinuityVariationalInequality}
 for $u$, $\psi$ and $f$, respectively. Therefore, we obtain that $\tilde u \in C^{\alpha_1}_s(\bar\sO_{\delta/2})$, for a constant $\alpha_1=\alpha_1(\delta, K, \Lambda, \nu_0,n,s)\in (0,1)$. Because we assume $g \in C^{\gamma}_s(\bar\Gamma_1\cap\bar\partial\sO_{\delta/2})$, we see that $u \in C^{\alpha_2}_s(\bar\sO_{\delta_2})$, where we may choose $\alpha_2:=\alpha_1\wedge\gamma$. When $\gamma=0$, that is, $g\in H^2(\sO,\fw)\cap C(\bar\Gamma_1\cap\partial\sO_{\delta/2})$, we see that $\alpha_2=0$, and so $u\in C(\bar\sO_{\delta/2})$.
\end{proof}

\section{Harnack inequality}
\label{sec:Harnack}
In this section, we prove Theorem \ref{thm:MainHarnack}, that is, the Harnack inequality for solutions $u\in H^1_0(\sO\cup\Gamma_0,\fw)$ to the variational equation \eqref{eq:IntroHestonWeakMixedProblemHomogeneous}. The key
differences from the proof of the classical Harnack inequality for variational solutions to non-degenerate elliptic equations \cite[Theorem 8.20]{GilbargTrudinger} are essentially those which we already outlined in \S
\ref{sec:HolderContinuityVariationalEquation} and the proof follows the same pattern as that of
Theorem \ref{thm:MainContinuityInterior}. Therefore, we only point out the major steps in the proof of Theorem \ref{thm:MainHarnack}, as the
details were explained in the preceding sections. We now proceed to the

\begin{proof} [Proof of Theorem \ref{thm:MainHarnack}]
Let $\bar R:= \dist(\partial\sO\cap\HH,\partial\sO'\cap\HH)$, and $R:=\bar R/4$. We first show that there is a positive constant $C=C(\Lambda, \nu_0, n, \bar R)$, such that for all $z_0 \in\Gamma_0\cap\partial_0\sO'$, we have
\begin{equation}
\label{eq:MainHarnack2}
\esssup_{\BB_R(z_0)}u \leq C\essinf_{\BB_R(z_0)}u.
\end{equation}
For clarity, we split the proof into principal steps.

\setcounter{step}{0}
\begin{step} [Energy estimates]
Let $\eta \in C^1_0(\bar{\HH})$ be a non-negative cutoff function with support in
$\bar \BB_{4R}(z_0)$. Let $\eps>0$ and
\begin{equation}
\label{eq:HarnackTestFunctionW}
w=u+\eps.
\end{equation}
We consider $\alpha \in \RR$, $\alpha \neq -1$. We set $H(w)=w^{(\alpha+1)/2}$ and
\begin{equation}
\label{eq:HarnackTestFunctionV}
v = \eta^2 w^{\alpha}.
\end{equation}
Then, $v \in H^1_0(\sO\cup\Gamma_0,\fw)$ is a valid test function in
%CP 1.23.2016: Replaced with the more general bilinear form
%\eqref{eq:HestonWithKillingBilinearForm}
%PF 2-27-2016: OK
\eqref{eq:Operator_A_bilinear_form}
by \cite[Lemma A.4]{Feehan_Pop_regularityweaksoln_v3}. By applying the same arguments as in the proofs of Theorem \ref{thm:MainSupremumEstimatesInterior} and Theorem \ref{thm:MainContinuityInterior}, we obtain the following analogous energy estimate to \eqref{eq:SupEstEnergyEst} and \eqref{eq:EnergyEstimate}, respectively
\begin{equation}
\label{eq:HarnackEnergyEstimates}
\begin{aligned}
{}&\left(\int|\eta H(w)|^p y^{\beta-1}\,dx\,dy\right)^{1/p}
\\
&\quad \leq
\left(C |1+\alpha|\right)^{1/p} \|\sqrt{y}\nabla \eta\|^{2/p}_{L^{\infty}(\HH)} \left(\int_{\hbox{supp }\eta} |H(w)|^2 y^{\beta-1}\,dx\,dy\right)^{1/p},
\end{aligned}
\end{equation}
where $C=C(\Lambda, \nu_0,  n, \bar R)$ is independent of $\eps$.
\end{step}

\begin{step}[Moser iteration]
By applying Moser iteration as described in the proofs of Theorem \ref{thm:MainSupremumEstimatesInterior}, for $\alpha>0$, and of Theorem \ref{thm:MainContinuityInterior}, for $\alpha<0$, we obtain
\begin{equation}
\label{eq:HarnackMoserIterations}
\begin{aligned}
\esssup_{\BB_R(z_0)} w &\leq C       \left(\frac{1}{|\BB_{2R}(z_0)|_{\beta-1}} \int_{\BB_{2R}(z_0)} w^2 y^{\beta-1}\,dx\,dy \right)^{1/2},\\
\essinf_{\BB_R(z_0)} w &\geq C^{-1} \left(\frac{1}{|\BB_{2R}(z_0)|_{\beta-1}} \int_{\BB_{2R}(z_0)} w^{-2} y^{\beta-1}\,dx\,dy \right)^{-1/2},
\end{aligned}
\end{equation}
where $C$ satisfies the same dependencies as the constant in \eqref{eq:HarnackEnergyEstimates}.
\end{step}

\begin{step} [Application of Theorem \ref{eq:AbstractJohnNirenberg}]
In this step, we verify that $w$ satisfies the requirements of the abstract John-Nirenberg inequality (Theorem \ref{thm:AbstractJohnNirenberg}) with $\theta_0=\theta_1=2$ and $S_r=\BB_{(2+r)R}(z_0)$, for all $0\leq r\leq 1$. From the hypotheses, we have that $0<4R<\dist(z_0,\Gamma_1)$, and so
$S_r=\BB_{(2+r)R}(z_0)$, for all $0\leq r\leq 1$, by \eqref{eq:Balls_relative_to_the_half_space} and \eqref{eq:Balls_relative_to_a_subdomain}.
By Proposition \ref{prop:ApplicationAbstractJohnNirenberg}, we see that $w$ satisfies condition \eqref{eq:AbstractJohnNirenberg1} of Theorem \ref{thm:AbstractJohnNirenberg}. Therefore, it remains to verify condition
\eqref{eq:AbstractJohnNirenberg2}, which follows in precisely the same way as in the proof of Theorem \ref{thm:MainContinuityInterior}.
\end{step}

\begin{step} [Proof of the Harnack inequality \eqref{eq:MainHarnack2} on a half-ball]
Because $w$ satisfies the conditions of Theorem \ref{thm:AbstractJohnNirenberg} by the preceding step, there is a positive constant $C$, independent of $\eps$, such that
\begin{equation}
\label{eq:HarnackJohnNirenberg}
\begin{aligned}
{}&\left(\frac{1}{\BB_{2R}(z_0)|_{\beta-1}} \int_{\BB_{2R}(z_0)} w^{2} y^{\beta-1}\,dx\,dy \right)^{1/2}
\\
&\quad\leq  C  \left(\frac{1}{|\BB_{2R}(z_0)|_{\beta-1}} \int_{\BB_{2R}(z_0)} w^{-2} y^{\beta-1}\,dx\,dy \right)^{-1/2}.
\end{aligned}
\end{equation}
Thus, combining inequalities \eqref{eq:HarnackMoserIterations} and \eqref{eq:HarnackJohnNirenberg} and recalling that $w=u+\eps$, we obtain
\[
\esssup_{\BB_R(z_0)}(u+\eps) \leq C \essinf_{\BB_R(z_0)}(u+\eps),
\]
for all $\eps>0$. Taking the limit as $\eps\downarrow 0$, we obtain the Harnack inequality \eqref{eq:MainHarnack2} on a half-ball.
\end{step}

The proof of \eqref{eq:MainHarnack1}, the Harnack inequality on an open subset $\sO'\Subset\sO\cup\Gamma_0$, follows by a standard covering argument similar to that in the proof of \cite[Corollary 8.21]{GilbargTrudinger}, with \eqref{eq:MainHarnack2} replacing \cite[Inequality (8.63)]{GilbargTrudinger} on half-balls centered at boundary points. More details can be found in the proof of \cite[Theorem 1.23]{Feehan_Pop_regularityweaksoln_v3}.
\end{proof}

\appendix

\section{Auxiliary results}
\label{sec:Auxiliary}
In this section we give the proof of Lemma \ref{lem:Extension}. As in \S \ref{sec:SobolevPoincare}, we work under the assumption stated in Remark \ref{rmk:ReferenceSetting}
%PF 2-25-2016: added (otherwise a bit obscure)
%CP 2.25.2016: Ok
that $z_0 = (0,0)$.

\begin{proof} [Proof of Lemma \ref{lem:Extension}]
By \cite[Corollary A.14]{Daskalopoulos_Feehan_statvarineqheston}, it is enough to prove the existence of an extension operator for functions $u \in C^1(\bar \BB_R(z_0))$. Fix a point
%CP 1.23.2016: Changed definition of point to simplify calculations.
%PF 2-27-2016: OK
$z'_0 = (0, y'_0)\in \BB_R(z_0)$, say with $y'_0= R^2/100$.
%PF 2-25-2016: Not sure this makes sense; we assumed R < 100 somewhere?
%CP 2.25.2016: We don't need to assume that. If you take the ball B_R, the point z=(x,y) on the boundary of B_R with x=0 has y=R^2. The above choice of z_0'=(x_0', y_0') with x_0'=0 and y_0'=R^2/100, just means we are fixing a point inside of the ball. It can be any interior point but it is easier in the proof if it is on the axis x=0.
%PF 2-27-2016: OK
We consider two different cases depending on whether $0<y \leq y'_0$ or $y >y'_0$.

First, we consider the points $z=(x,y) \in D \backslash \BB_R(z_0)$ such that $0<y \leq y'_0$.
%PF 2-25-2016: Slightly rephrased
%CP 2.25.2016: Ok
Let $z'=(x',y)$ be the intersection of $\partial \BB_R(z_0)$ with the line through $z$ and $(0,y)$.
Then, we define
$E u(z)$ by reflection (with respect to the point $z'$ in the hyperplane at level $y$)
\[
E u(z) := u \left(\frac{|x'|}{|x|^2}x, y\right).
\]
Next, we consider the case of points $z=(x,y) \in D \backslash \BB_R(z_0)$ such that $y > y'_0$.
%PF 2-25-2016: Slightly rephrased
%CP 2.25.2016: Ok
Again let $z'=(x',y')$ be the intersection of $\partial \BB_R(z_0)$ with the line through $z$ and $z'_0$.
Then, we define $E u(z)$ by reflection
\[
E u(z) := u \left(z'_0 + \frac{|z'-z'_0|}{|z-z'_0|^2} (z-z'_0)\right).
\]
%PF 2-25-2016: Added
%CP 2.25.2016: The display above and below are exactly the same thing, but if you think it makes everything clearer we can leave it
Therefore,
\[
E u(x,y) := u \left(\frac{|z'-z'_0|}{|z-z'_0|^2}x, y_0' + \frac{|z'-z'_0|}{|z-z'_0|^2}(y-y_0')\right).
\]
%PF 2-25-2016: Rephrased
%It is clear that $Eu$ is a continuous extension of $u$ from $\BB_R(z_0)$ to $D$.
%CP 2.25.2016: Ok
and so it is clear that the function $Eu$ is continuous on $D$ and is equal to $u$ on $\BB_R(z_0)$.
%PF 2-25-2016: We're only reflecting through the smooth surface in the upper half plane
%Because $\partial \BB_R(z_0)$ is a piecewise smooth
%CP 2.25.2016: Ok
Because $\HH\cap \partial \BB_R(z_0)$ is a smooth
%CP 1.23.2016: Changed "curve" to "surface"
%PF 2-27-2016: OK
surface, $Eu$ has well-defined weak derivatives in $D$. Next, we show that
\eqref{eq:ExtensionContinuity} holds. For this purpose, we denote by
\begin{equation*}
\begin{aligned}
D_1 &:= \left(D \backslash \BB_R(z_0)\right) \cap \{y<y_0'\},\\
D_2 &:= \left(D \backslash \BB_R(z_0)\right) \cap \{y \geq y_0'\}.
\end{aligned}
\end{equation*}
To prove \eqref{eq:ExtensionContinuity}, it is enough to show there is a positive constant $C$, depending on $R$ and $D$, such that
\begin{equation}
\label{eq:ExtensionAdd1}
\begin{aligned}
\int_{D_1} |Eu(x,y)|^2 y^{\beta-1}\,dx\,dy  & \leq C \int_{\BB_R(z_0)} |u(x,y)|^2 y^{\beta-1} \,dx\,dy,\\
\int_{D_1} |\nabla Eu(x,y)|^2 y^{\beta}  \,dx\,dy &\leq C \int_{\BB_R(z_0)} |\nabla u(x,y)|^2 y^{\beta} \,dx\,dy,\\
\int_{D_2} |Eu(x,y)|^2 y^{\beta-1}\,dx\,dy  & \leq C \int_{\BB_R(z_0)} |u(x,y)|^2 y^{\beta-1} \,dx\,dy,\\
\int_{D_2} |\nabla Eu(x,y)|^2 y^{\beta}  \,dx\,dy &\leq C \int_{\BB_R(z_0)} |\nabla u(x,y)|^2 y^{\beta} \,dx\,dy,
\end{aligned}
\end{equation}
%CP 1.23.2016: Many changes from here on
%PF 2-27-2016: OK
We begin by evaluating the integrals over $D_1$ in \eqref{eq:ExtensionAdd1} and we denote by
\begin{equation}
\label{eq:ExtensionAdd3}
\begin{aligned}
x''=\varphi(x,y) := \frac{|x'|}{|x|^2}x.
\end{aligned}
\end{equation}
We notice that $(\varphi(x),y) \in \BB_{R}(z_0)$, for all $(x,y)\in D_1$, so $Eu(x,y)$ is well-defined on $D_1$.  The coordinate $x'=x'(y)$ is determined by the condition $d((x',y), z_0)=R$. Direct calculations give us that
$$
x'(y)=\left(\frac{R^2+\sqrt{R^4+4Ry}}{2}-y^2\right)^{1/2}\frac{x}{|x|}.
$$
We can find a positive constant $C_1$, depending only on $R$, such that
\[
|x| \geq |x''| \geq C_1, \quad \forall\, (x,y) \in D_1^+,
\]
and using the fact that
$$
\varphi^{-1}(x'') = \frac{|x'|}{|x''|^2}x'',
$$
we can find a positive constant $C_2$, depending on $R$ and $D$, such that
\begin{equation}
\label{eq:ExtensionAdd4}
\begin{aligned}
|\nabla\varphi(x,y)| &\leq C_2,\quad\forall\, (x,y)\in D_1,\\
|\nabla\varphi^{-1}(x'',y)| &\leq C_2,\quad\forall\, (x'',y)\in \varphi(D_1),
\end{aligned}
\end{equation}
Using the change of variable $x''=\varphi(x,y)$, we obtain
\begin{equation}
\label{eq:ExtensionAdd5}
\begin{aligned}
\int_{D_1} |Eu(x,y)|^2 y^{\beta-1}\,dx\,dy  &\leq  \int_{\varphi(D_1)} |u(x'',y)|^2 y^{\beta-1} |\hbox{det}\nabla\varphi^{-1}(x'',y)| dx''dy\\
& \leq C_2 \int_{\BB_R(z_0)} |u(x,y)|^2 y^{\beta-1} \,dx\,dy, \quad\hbox{(by \eqref{eq:ExtensionAdd4}).}
\end{aligned}
\end{equation}
Using
\begin{equation*}
\begin{aligned}
\partial_{x_i} Eu(x,y) &= \sum_{j=1}^{n-1}u_{x_j}(\varphi(x,y),y) \partial_{x_i}\varphi_j(x,y),\quad 1\leq i\leq n-1,\\
\partial_y Eu(x,y) &=  \sum_{j=1}^{n-1}u_{x_j}(\varphi(x,y),y) \partial_{y}\varphi_j(x,y)+ u_y(\varphi(x,y),y),\\
\end{aligned}
\end{equation*}
the change of variable $x''=\varphi(x,y)$ and the upper bound \eqref{eq:ExtensionAdd4}, we obtain for a positive constant $C_3$, depending on $R$ and $D$,
$$
\int_{D_1} |\nabla Eu(x,y)|^2 y^{\beta}\,dx\,dy  \leq  C \int_{\varphi(D_1)} |\nabla u(x'',y)|^2
|\nabla\varphi(x,y)|^2 |\hbox{det}\nabla\varphi^{-1}(x'',y)| y^{\beta} dx''dy,
$$
and thus
\begin{equation}
\label{eq:ExtensionAdd6}
\begin{aligned}
\int_{D_1} |\nabla Eu(x,y)|^2 y^{\beta}\,dx\,dy
& \leq C_3 \int_{\BB_R(z_0)} |\nabla u(x,y)|^2 y^{\beta} \,dx\,dy.
\end{aligned}
\end{equation}
Therefore, \eqref{eq:ExtensionAdd5} and \eqref{eq:ExtensionAdd6} give us the first two inequalities in \eqref{eq:ExtensionAdd1}.

Next, we consider the last two integrals in \eqref{eq:ExtensionAdd1}. Notice that on $D_2$ we have $y\geq y'_0>0$ and so it is enough to show
\begin{equation}
\label{eq:ExtensionAdd7}
\begin{aligned}
\int_{D_2} |Eu(z)|^2\,dz  & \leq C_4 \int_{\BB_R(z_0)} |u(z)|^2  \,dz,\\
\int_{D_2} |\nabla Eu(z)|^2  \,dz &\leq C_4 \int_{\BB_R(z_0)} |\nabla u(z)|^2  \,dz,
\end{aligned}
\end{equation}
for some positive constant $C_4$, depending on $R$ and $D$. For all $z\in D_2$, we now denote
\begin{equation*}
\begin{aligned}
z''=\varphi(z) &:= z'_0 + \frac{z'-z'_0}{|z-z'_0|^2}(z-z'_0).
\end{aligned}
\end{equation*}
Analogous to \eqref{eq:ExtensionAdd4}, we can find a positive constant $C_5$, depending on $R$ and $D$, such that for all $z \in D_2$,
\begin{equation}
\label{eq:ExtensionAdd8}
\begin{aligned}
|\nabla \varphi(z)| &\leq C_5,\quad\forall\, z\in D_2,\\
|\nabla \varphi^{-1}(z'')| &\leq C_5,\quad\forall\, z''\in\varphi(D_2).
\end{aligned}
\end{equation}
We notice that $\varphi(z) \in \BB_R(z_0)$, for all $z \in D_2$. Therefore, using the change of variable $z'' =\varphi(z)$, we obtain
\begin{equation}
\label{eq:ExtensionAdd9}
\begin{aligned}
\int_{D_2} |Eu(z)|^2\,dz &\leq \int_{\varphi(D_2)} |u(z'')|^2 |\text{det} \nabla \varphi^{-1}(z'')|\, dz''\\
& \leq C_5 \int_{\BB_R(z_0)} |u(x,y)|^2 \,dx\,dy \quad\hbox{(by \eqref{eq:ExtensionAdd8})}.
\end{aligned}
\end{equation}
Using
\begin{equation*}
\begin{aligned}
\partial_{x_i} E u(z) &= \sum_{j=1}^{n-1}u_{x_j}(\varphi(z)) \partial_{x_i}\varphi_j(z)
+ u_{y}(\varphi(z)) \partial_{x_i}\varphi_n(z),\quad 1\leq i\leq n-1,\\
\partial_y E u(z) &= \sum_{j=1}^{n-1}u_{x_j}(\varphi(z)) \partial_y\varphi_j(z)+u_y(\varphi(z)) \partial_{y}\varphi_n(z),
\end{aligned}
\end{equation*}
we obtain
\begin{equation}
\label{eq:ExtensionAdd10}
\begin{aligned}
\int_{D_2} |\nabla Eu(z)|^2\,dz &\leq C \int_{\varphi(D_2)} |\nabla u(z'')|^2 |\nabla \varphi(z)|^2
|\text{det} \nabla \varphi^{-1}(z'')| \,dz''\\
& \leq C C_5 \int_{\BB_R(z_0)} |\nabla u(z)|^2 \,dz, \quad\hbox{(by \eqref{eq:ExtensionAdd8}.)}
\end{aligned}
\end{equation}
From \eqref{eq:ExtensionAdd9} and \eqref{eq:ExtensionAdd10}, we obtain \eqref{eq:ExtensionAdd7}. This concludes the proof of Lemma \ref{lem:Extension}.
\end{proof}

%%%%%%%%%%%%%%%%%%%%%%%%%%%%%%%%%%%%%%%%%%%%%%%%%%%%%%%%%%%%%%%%%%%%%%%%%%%%%%%
%
%                                bibliography
%
%%%%%%%%%%%%%%%%%%%%%%%%%%%%%%%%%%%%%%%%%%%%%%%%%%%%%%%%%%%%%%%%%%%%%%%%%%%%%%%

\bibliography{master,mfpde}
\bibliographystyle{amsplain}

\end{document}